\documentclass[11pt,reqno]{amsart}
\usepackage{amssymb}
\usepackage{amsmath}
\usepackage{amsthm}
\usepackage{color}
\usepackage[titletoc]{appendix}
\usepackage[unicode,breaklinks=true,colorlinks=true]{hyperref}
\usepackage{CJKutf8}

\makeatletter
\renewcommand{\subsection}{\@startsection{subsection}{2}{0pt}{\baselineskip}{\baselineskip}{\bfseries}}
\makeatother

\makeatletter
\def\LaTeX{\leavevmode L\raise.42ex
	\hbox{\kern-.3em\size{\sf@size}{0pt}\selectfont A}\kern-.15em\TeX}
\makeatother

\newcommand{\BibTeX}{{\rm B\kern-.05em{\sc
			i\kern-.025emb}\kern-.08em\TeX}}

\makeatletter
\def\@currentlabel{2.1}\label{e:dispaa}
\def\@currentlabel{2.21}\label{e:dispau}
\def\@currentlabel{2.22}\label{e:dispav}
\def\@currentlabel{2.23}\label{e:dispaw}
\def\@currentlabel{2.24}\label{e:dispax}
\def\theequation{\thesection.\@arabic\c@equation}
\makeatother

\newcounter{mnotecount}[section]

\newcommand{\rmnote}[1]{}

\numberwithin{equation}{section}
\renewcommand{\theequation}{\arabic{section}.\arabic{equation}}
\newtheorem{thm}{Theorem}[section]
\newtheorem{lem}[thm]{Lemma}

\newtheorem{prop}[thm]{Proposition}
\theoremstyle{definition}
\newtheorem{defn}{Definition}[section]
\newtheorem{rem}{Remark}[section]

\newcommand{\R}{\mathbb{R}}

\newcommand\supp{\operatorname{supp}}

\allowdisplaybreaks

\begin{document}
	\begin{CJK}{UTF8}{gbsn}
		\author[Y. Yang]{Yifan Yang}
		\address{School of Mathematics and statistics, Xi'an Jiaotong University, Xi'an 710049, P.R. China}
		\email{yangyifan2021@stu.xjtu.edu.cn}
		\title[self-similar solutions to the MHD equations]{Forward self-similar solutions to the MHD equations: existence and pointwise estimates}

		\maketitle
		\begin{abstract}
			In this paper, we study the forward self-similar solutions to the three-dimensional Magnetohydrodynamic equations (MHD equations) in the whole space. By employing the Leray-Schauder theorem and blow-up argument, we construct a global-time forward self-similar solutions, which is smooth in $\R^{3}\times(0,\infty)$.
			 Furthermore, by investigating the regularity of the weak solutions to the corresponding Leray system in the weighted Sobolev space, we can derive the pointwise estimate for the forward self-similar solution.
			
		\end{abstract}
	      
		\setcounter{equation}{0}
		\setcounter{equation}{0}
		\section{Introduction}
		The MHD equations represent a mathematical model in plasma physics, where the interaction between the velocity field and the magnetic field dictates that plasma motion adheres to the principles of both fluid dynamics and electromagnetic field laws.  Consequently, the MHD equations are derived by coupling the Navier-Stokes equations of fluid dynamics with Maxwell's equations of electromagnetism.  This paper will explore the forward self-similar solutions to the following incompressible MHD equations
		\begin{align}\label{E1.1}
			\begin{split}
				\left.
				\begin{aligned}
					\partial_{t} u-\frac{1}{Re}\Delta u+(u\cdot\nabla)u-(b\cdot\nabla)b+\nabla p=0\\
					\partial_{t} b-\frac{1}{Rm}\Delta b+(u\cdot\nabla)b-(b\cdot\nabla)u=0\\
					\mbox{div}\,u =\mbox{div}\,b =0\\
				\end{aligned}\ \right\}\ \mbox{in}\ \R^{3}\times (0,+\infty)
			\end{split}
		\end{align}
		with initial condition
		\begin{equation}\label{E1.1-initial1}
			u(\cdot,0)=u_{0}(x)\ \mbox{and}\ \ b(\cdot,0)=b_{0}(x), \ \ \mbox{in}\ \  \R^3,
		\end{equation}
	where $u:\R^{3}\times\R^{+}\rightarrow\R^{3}$ is the velocity field and $b:\R^{3}\times\R^{+}\rightarrow\R^{3}$ is the magnetic field, the unknown scalar field $p:\R^{3}\times\R^{+}\rightarrow\R$ corresponds to the pressure of the fluid. The nondimensional numbers $Re$ and $Rm$ represent the Reynolds number and the magnetic Reynolds number, respectively. Since the values of $Re$ and $Rm$ don't play any role in our proofs, we will assume, for simplicity, that $Re = Rm = 1 $  throughout this paper.
		
		It is well know that the system \eqref{E1.1} is invariant under the scaling
		\begin{align}\label{E1.1-scal1}
			\left \{
			\begin{array}{ll}
				u_{\lambda}(x,t)=\lambda u(\lambda x,\lambda^{2}t),\\
				b_{\lambda}(x,t)=\lambda b(\lambda x,\lambda^{2}t),\\
				p_{\lambda}(x,t)=\lambda^{2}p(\lambda x,\lambda^{2}t).
			\end{array}
			\right.
		\end{align}
		We say that a solution $(u,b,p)$ of the system \eqref{E1.1} is a forward self-similar solution if $u=u_{\lambda}$, $b=b_{\lambda}$ and $p=p_{\lambda}$ for all $\lambda>0$. Note that, in the study of self-similar solutions, we naturally require that the initial values are also self-similar, that is,
		\begin{align*}
		u_{0\lambda}(x)=\lambda u_{0}(\lambda x)\ \mbox{and} \ b_{0\lambda}(x)=\lambda b_{0}(\lambda x).
		\end{align*}
		This implies that the initial data $(u_{0},b_{0})$ is homogeneous of degree $-1$.
		
		The study of self-similar solutions in fluid mechanics can be traced back to Leray's work \cite{JL} on the Navier-Stokes equations (i.e. $b=0$ in system \eqref{E1.1}) in 1934. Many years later, Giga, Miyakawa and Osada \cite{GMO} proved the existence and uniqueness of self-similar solutions in $\R^{2}$. For $\R^{3}$, under the assumption of small initial values, the existence and uniqueness of self-similar solutions have been established in various frameworks, such as Morrey space \cite{GT}, Besov space \cite{C1,C2},
		Lorentz space \cite{B} and $\mbox{BMO}^{-1}(\R^{3})$ \cite{KT2}.
		However, for large initial values, the contraction mapping argument becomes ineffective. 
	Recently, Jia and \v{S}ver\'{a}k \cite{JS} constructed a forward self-similar solution $u(x,t)$ for any scale-invariant initial values $u_{0}$ which is locally H\"{o}lder continuous in $\R^{3}\setminus\{0\}$. Precisely,
	based on the theory of local Leray solutions introduced by Lemari\'{e}-Rieusset \cite{LR}, they derive the local-in-space H\"{o}lder estimates of $u(x,t)$ near $t=0$. Subsequently, they obtained {\em a priori} estimate for the perturbation term $V(x)=u(x,1)-e^{\Delta}u_{0}$, which enabled them to construct a forward self-similar solution using the {\em Leray-Schauder theorem}.
 Later, Korobkov and Tsai \cite{KT} introduced a new approach for constructing large forward self-similar solutions.
 The difference with the method presented in \cite{JS} is that Korobkov and Tsai directly consider the corresponding Leray systems about $V(x)$, establishes {\em a priori} estimates by the so called {\em blow-up argument} in Sobolev space, and the fixed-point theorem subsequently guarantees the existence of a self-similar solution.
  It is noteworthy that this {\em blow-up argument} likely originated from the work of De Giorgi on minimal surfaces \cite{GE}. Recently, Abe and Giga \cite{AG} employed this technique to demonstrate the analyticity of Stokes semigroups in $L^{\infty}$-type spaces.
In the study of self-similar solutions, one of the advantages of this method is its ability to address the half-space problem. Additionally, it can provide pointwise estimates for self-similar solutions by examining the regularity of weak solutions to the corresponding Leray systems. For example, by investigating the regularity of $V(x)$ within the framework of weighted Sobolev spaces, Lai, Miao and Zheng \cite{LMZ,LMZ1} derived a decay estimate for $V(x)$ and further obtained a pointwise estimate of the self-similar solution $u(x,t)$. Very recently, this approach has been adapted to the Oberbeck-Boussinesq system with Newtonian gravitational field by Brandolese and Karch \cite{LBGK}. For additional works on the existence and regularity of self-similar solutions to the Navier-Stokes equations, please refer to \cite{ BZT,LB,CGNP, JST, NRS, T,T1} and the references cited therein.
		
		The mathematical study of MHD equations \eqref{E1.1} can be traced back to the work of Duvaut and Lions \cite{GDJL}, who constructed a class of global weak solutions, i.e., energy weak solutions. The regularity of these energy weak solutions was further studied by Sermange and Teman \cite{ST}. When the initial values are small, 
		He and Xin \cite{HX} proved the existence and uniqueness of forward self-similar solutions in certain homogeneous spaces via contraction mapping argument. 
		 For large initial data, discretely self-similar solutions have been investigated within various frameworks, which is a broader class that contains forward self-similar solutions. For example, Lai \cite{L} has utilized the classical Galerkin method to construct a discretely self-similar local Leray solution for any initial values $u_{0},b_{0}\in L^{3,\infty}(\R^{3})$.
		Concurrently, Zhang and Zhang \cite{ZZ} proved the existence of discretely self-similar solutions to the generalized MHD equations in Besov space. Later, Fern\'{a}ndez-Dalgo and Jarr\'{i}n \cite{FDJ} also established the existence of discrete self-similar solutions in weighted $L^{2}$-space by examining the linearized equations (the advection-diffusion problem). However, their work does not address the pointwise estimates for large self-similar solutions.
		
		In this work, 
		 we construct a forward self-similar solution to the MHD equations with the $L^{\infty}_{loc}(\R^{3}\setminus\{0\})$ initial data. Moreover, under the assumption of higher regularity for the initial values, we can derive the pointwise estimate for the self-similar solution. The main result of this paper is stated as follows:
			\begin{thm}\label{T1.1}
			Let $(u_{0},b_{0})\in \mathbf{L}^{\infty}_{loc}(\R^{3}\setminus\{0\})$ be homogeneous of degree $-1$, with $\mathrm{div} u_{0}=\mathrm{div} b_{0}=0$. Then, there exists a constant $C=C(u_{0},b_{0})>0$, such that  
				\begin{enumerate}
				\item[(i)](existence and regularity) 
				the problem \eqref{E1.1}-\eqref{E1.1-initial1} exists a forward self-similar solution $(u,b)\in BC_{w}\big([0,\infty);\mathbf{L}^{3,\infty}(\R^{3})\big)$, which is smooth in $\R^{3}\times(0,\infty)$. Moreover, for all $t>0$ and $2\leq p\leq 6$, we have
				\begin{align}\label{T1}
					\begin{split}
						&\|u(t)-e^{t\Delta}u_{0}\|_{p}+\|b(t)-e^{t\Delta}b_{0}\|_{p}\leq Ct^{\frac{3}{2p}-\frac{1}{2}},\\&
						\|\nabla u(t)-\nabla e^{t\Delta}u_{0}\|_{2}+\|\nabla b(t)-\nabla e^{t\Delta}b_{0}\|_{2}\leq Ct^{-\frac{1}{4}}.
					\end{split}
				\end{align}
				\item[(ii)] (pointwise estimate) if $(u_{0},b_{0})\in \mathbf{C}^{0,1}_{loc}(\R^{3}\setminus\{0\})$, then, we have the following pointwise estimate
				\begin{align}\label{T3}
					\begin{split}
						|\nabla^{k} u(x,t)|+|\nabla^{k}b(x,t)|\leq C\frac{1}{(\sqrt{t}+|x|)^{k+1}},\ \ k=0,1
					\end{split}
				\end{align}
				and
				\begin{align}\label{T2-3}
					\begin{split}
						&|u(x,t)-e^{t\Delta}u_{0}|+|b(x,t)-e^{t\Delta}b_{0}|\leq C\frac{t}{(\sqrt{t}+|x|)^{3}}\log{\Big(2+\frac{|x|}{\sqrt{t}}\Big)},\\&
						|\nabla u(x,t)-\nabla e^{t\Delta}u_{0}|+|\nabla b(x,t)-\nabla e^{t\Delta}b_{0}|\leq C\frac{\sqrt{t}}{(\sqrt{t}+|x|)^{3}}.
					\end{split}
				\end{align}
			\item[(iii)] (improved pointwise estimate) if $(u_{0},b_{0})\in \mathbf{C}^{1,\alpha}_{loc}(\R^{3}\setminus\{0\})$ for any $0<\alpha\leq1$, then, the optimal pointwise estimate for the perturbation term is obtained, i.e.,
				\begin{align}\label{T2-4}
					\begin{split}
						|u(x,t)-e^{t\Delta}u_{0}|+|b(x,t)-e^{t\Delta}b_{0}|\leq C\frac{t}{(\sqrt{t}+|x|)^{3}}.
					\end{split}
				\end{align}
				\item[(iv)] (pointwise estimate of the pressure) 
				if $(u_{0},b_{0})\in \mathbf{C}^{1,1}_{loc}(\R^{3}\setminus\{0\})$, then there exists a pressure $p(x,t)$ satisfies
				\begin{equation}\label{Press}
					|\nabla^{k} p(x,t)|\leq Ct^{-\frac{1}{2}}\frac{1}{(\sqrt{t}+|x|)^{k+1}},\ \ k=0,1.
				\end{equation}
			\end{enumerate}
		\end{thm}
		\begin{rem} 
			In Theorem \ref{T1.1}, we construct a smooth, global-times forward self-similar solution such that
			$$
			\big(u(\cdot,t),b(\cdot,t)\big)\overset{*}{\rightharpoonup}(u_{0},b_{0})
			$$
			in $\mathbf{L}^{3,\infty}(\R^{3})$ as $t\rightarrow 0^{+}$. 
			However, we do not involve uniqueness. For the large initial data, the uniqueness of forward self-similar solutions is still open to the Navier-Stokes equations \cite{JS}, and certainly for MHD equations.
		\end{rem}
		\begin{rem} 
			The pointwise estimate \eqref{T2-4} is regarded as optimal in the sense as follows: for $|x|>1$, it holds that 
			\begin{align*}
				\begin{split}
					|u(x,1)-e^{\Delta}u_{0}|\sim& |e^{\Delta}u_{0}||\nabla e^{\Delta}u_{0}|+|e^{\Delta}b_{0}||\nabla e^{\Delta}b_{0}|,\\
					|b(x,1)-e^{\Delta}b_{0}|\sim& |e^{\Delta}u_{0}||\nabla e^{\Delta}b_{0}|+|e^{\Delta}b_{0}||\nabla e^{\Delta}u_{0}|.
				\end{split}
			\end{align*}
			For the Navier-Stokes equations, the pointwise estimate \eqref{T2-4} was first established by \cite{LMZ1} under the assumption of $u_{0}\in C^{1,1}_{loc}(\R^{3}\setminus\{0\})$ . 
			More recently,  \cite{BP} and \cite{Lai} have established that the smoothness assumption can be reduced to $u_{0}\in C^{1,\alpha}_{loc}(\R^{3}\setminus\{0\})$ for any $0<\alpha\leq1$ and they have demonstrated that this condition is optimal.
		 All of our proofs rely strictly on the special structure and pointwise estimates of the Oseen and heat kernel.
		\end{rem}
		
		\begin{rem} For the half space $\R^{3}_{+}$, the existence of self-similar solutions to the Navier-Stokes equations was established by Korobkov and Tsai \cite{KT}. However, it remains unclear whether this solution exhibits the corresponding decay behavior. In our future work, we will investigate the self-similar solutions of the Navier-Stokes and MHD equations in half space.
		\end{rem}
		
	To illuminate the motivations of this paper,	let us briefly outline the main idea underlying the proof of our results.
		Assume that $(u,b)$ is a self-similar solution of equation \eqref{E1.1}, choosing $\lambda=\frac{1}{\sqrt{t}}$ in \eqref{E1.1-scal1}, $(u,b)$ can be written as
		\begin{equation}
			u(x,t)=\frac{1}{\sqrt{t}}U\Big(\frac{x}{\sqrt{t}}\Big) \ \ \mbox{and}\  \ b(x,t)=\frac{1}{\sqrt{t}}B\Big(\frac{x}{\sqrt{t}}\Big),
		\end{equation}
		where
		$$
		U(x)=u(x,1)\ \ \mbox{and} \ \ B(x)=b(x,1).
		$$
		Then one can obtain a time-independent profile $(U,B)$, which solves the following Leray system for the MHD equations
		\begin{align}\label{LS1}
			\begin{split}
				\left.
				\begin{aligned}
					-\Delta U-\frac{1}{2}U-\frac{1}{2}x\cdot\nabla U+(U\cdot\nabla)U-(B\cdot\nabla)B+\nabla P=0\\
					-\Delta B-\frac{1}{2}B-\frac{1}{2}x\cdot\nabla B+(U\cdot\nabla)B-(B\cdot\nabla)U=0\\
					\mbox{div}\,U =\mbox{div}\,B =0\\
				\end{aligned}\ \right\} \ \mbox{in}\ \R^{3}.
			\end{split}
		\end{align}
		Notice that if we multiply the first equation of \eqref{LS1} by $U$, and the second equation of \eqref{LS1} by $B$, then add the resulting equations and integrate over $\R^{3}$, we can derive
		$$
		\frac{1}{4}(\|U\|_{2}^{2}+\|B\|_{2}^{2})+\|\nabla U\|_{2}^{2}+\|\nabla B\|_{2}^{2}=0.
		$$
	The above identity implies that we cannot construct a non-trivial weak solution $(U,B)\in \mathbf{H}^1(\R^{3})$ to the system \eqref{LS1}.
	Thus, inspired by \cite{T}, we examine the difference
		$$
		V=U-U_{0}\ \ \mbox{and}\ \ G=B-B_{0},
		$$
		where 
		$$
		U_{0}=e^{\Delta}u_{0}\ \ \mbox{and}\ \ B_{0}=e^{\Delta}b_{0}.
		$$
		By the scaling property, it is easy to verify that $(V,E)$ satisfies the following \textbf{perturbed Leray system} (\textbf{PLS})
		\begin{align}\label{PLS}
			\left\{
			\begin{aligned}
				&-\Delta V-\frac{1}{2}V-\frac{1}{2}x\cdot\nabla V+(V+U_{0})\cdot\nabla(V+U_{0})-(G+B_{0})\cdot\nabla(G+B_{0})+\nabla P=0,\\&
			-\Delta G-\frac{1}{2}G-\frac{1}{2}x\cdot\nabla G+(V+U_{0})\cdot\nabla(G+B_{0})-(G+B_{0})\cdot\nabla(V+U_{0})=0,\\&
			\mbox{div}\,V =\mbox{div}\,G =0.\\
			\end{aligned}
			\right.
		\end{align}
		Our first goal is to construct a weak solution (see Definition \ref{D-0}) $(V,G)\in \mathbf{H}^1(\R^{3})$ to \eqref{PLS}. More specifically, we employ the {\em blow-up argument} to establish {\em a prior} estimate of $(V,G)$ and then utilize the {\em Leray-Schauder theorem} along with the technique of invading domains to prove the existence of weak solutions to the \textbf{PLS} in the whole space, for detail, see Theorem \ref{T1.2}. Next, we study the pointwise estimate of the weak solution $(V,G)$.
		Assume that $(u_{0},b_{0})\in \mathbf{C}^{0,1}_{loc}(\R^{3}\setminus\{0\})$, motivated by \cite{LMZ}, we can derive the following weighted estimate of $(V,G)$ by choosing the suitable test function,
		$$
			\big\|(|x|V,|x|G)\big\|_{\mathbf{H}^{1}(\R^{3})}\leq C(U_{0},B_{0}).
		$$
		Using the above estimates as a starting point, we can further obtain the following crucial estimates by means of the energy method
		\begin{align*}
			\begin{split}
				&|V(x)|+|G(x)|\leq C(1+|x|)^{-1},\\&
				|\nabla V(x)|+|\nabla G(x)|\leq C(1+|x|)^{-1}.
			\end{split}
		\end{align*}
		Next, according to the linear theory of Stokes and the heat equation, we can further improve the decay rate of the weak solution $(V,G)$ to
		\begin{align}\label{k1}
			\begin{split}
			&|V(x)|+|G(x)|\leq C(1+|x|)^{-3}\log(2+|x|),\\&
				|\nabla V(x)|+
				|\nabla G(x)|\leq C(1+|x|)^{-3}.
			\end{split}
		\end{align}
	Moreover, if $(u_{0},b_{0})\in \mathbf{C}^{1,\alpha}_{loc}(\R^{3}\setminus\{0\})$ for any $0<\alpha\leq1$, with the help of a rigorous pointwise estimate of the Oseen kernel allows us to eliminate the logarithmic loss present in \eqref{k1} to obtain the optimal estimate of $(V,G)$:
		$$
		|V(x)|+|G(x)|\leq C(1+|x|)^{-3}.
		$$
		For detail, see Theorem \ref{T1-1} in Section 3. Additionally, since
		$$
			 P=(-\Delta)^{-1}\mbox{divdiv}\Big((V+U_{0})\otimes(V+U_{0})-(G+B_{0})\otimes(G+B_{0})\Big),
		$$
		 the decay estimate of $P(x)$ can be derived using classical potential theory, provided that $(u_{0},b_{0})\in \mathbf{C}^{1,1}_{loc}(\R^{3}\setminus\{0\})$.
		 Then, by applying the scaling property, we can directly  deduce the existence and pointwise estimate of the self-similar solution, as stated in Theorem \ref{T1.1}.

		At the end of this section, we provide the definition of the weak solution to the \textbf{PLS} \eqref{PLS} and some notations.
	\begin{defn}(Weak solution)\label{D-0}
		A triple $(V,G,P)$ is called a weak solution of \eqref{PLS} in $\R^{3}$, if
			\begin{enumerate}
			\item[(i)] $(V,G)\in \mathbf{H}^{1}_{\sigma}(\R^{3})$ and $P\in L^{2}(\R^{3})$;
			\item[(ii)] $(V,G,P)$ satisfies \eqref{PLS}  in the sense of distribution in $\R^{3}$, i.e., for all vector fields $(\varphi,\psi)\in \mathbf{H}^{1}(\R^{3})$ with $\big\|\big(|\cdot|\varphi,|\cdot|\psi\big)\big\|_{2}<\infty$, the triple $(V,G,P)$ fulfills
			\begin{align}\label{W1}
				\begin{split}
					&\int_{\R^{3}}\nabla V:\nabla\varphi dx-\frac{1}{2}\int_{\R^{3}}V\cdot\varphi dx-\frac{1}{2}\int_{\R^{3}}x\cdot\nabla V\cdot\varphi dx-\int_{\R^{3}}P\mathrm{div}\varphi dx\\&=\int_{\R^{3}}\Big(-(V+U_{0})\cdot\nabla(V+U_{0})+(G+B_{0})\cdot\nabla(G+B_{0})\Big)\cdot\varphi dx
				\end{split}
			\end{align}
			and
			\begin{align}\label{W2}
				\begin{split}
					\int_{\R^{3}}&\nabla G:\nabla\psi dx-\frac{1}{2}\int_{\R^{3}}G\cdot\psi dx-\frac{1}{2}\int_{\R^{3}}x\cdot\nabla G\cdot\psi dx\\&=\int_{\R^{3}}\Big(-(V+U_{0})\cdot\nabla(G+B_{0})+(G+B_{0})\cdot\nabla(V+U_{0})\Big)\cdot\psi dx.
				\end{split}
			\end{align}
			\end{enumerate}
	\end{defn}
		\noindent\textbf{Notation:}
		We denote $B_{R}(x_{0})$ as a ball centered at $x_{0}$ with radius $R$, and $B_{R}:=B_{R}(0)$.
		Let
		$u,v:\R^{3}\rightarrow\R^{3}$ be two vector fields, the $i^{th}$-difference quotient of size $h$ is defined as follows:
		$$
		D_{i}^{h}u(x)=\frac{u(x+he_{i})-u(x)}{h},\ \ h\in\R,\ h\ne0.
		$$	 We denote $u\otimes v=(u_{i}v_{j})_{1\leq i,j\leq3}$ is a second order tensor. Throughout this paper, we use the Einstein summation convention. Thus, we write $(v\cdot\nabla v)_{j}=v_{i}\partial_{i}v_{j}$.
		For a matrix valued function $f=(f_{ij})$, $(\mbox{div} f)_{i}=\partial_{j}f_{ij}$; for two matrices $C$ and $D$,
		$$
		C:D=C_{ij}D_{ij},\ \ \mbox{and}\ \ |C|=(C:C)^{\frac{1}{2}}.
		$$
		The subscript $\sigma$ indicates that divergence-free conditions, for example,
		$$
		C_{0,\sigma}^{\infty}(\Omega)=\{u\in C_{0}^{\infty}(\Omega)|\mbox{div}\,u=0\}
		$$
		and
		$$
		L_{\sigma}^{p}(\Omega)=\mbox{the closure of $C_{0,\sigma}^{\infty}(\Omega)$ under the norm $L^{p}$}.
		$$
		Let $X$ be a Banach space, $X^{\prime}$ represent the dual space of $X$, and $\langle\cdot,\cdot\rangle$ be the dual pairing of $X$ and $X^{\prime}$. In this paper, we use bold symbols to denote
		function space of 6-dimensional vector-valued functions. For example, 
		$$
		\mathbf{L}^{2}(\Omega)=L^{2}(\Omega)^{3}\times L^{2}(\Omega)^{3},\ \ \mathbf{L}_{\sigma}^{2}(\Omega)=L_{\sigma}^{2}(\Omega)\times L_{\sigma}^{2}(\Omega).
		$$
		In order to simplify, for $u,v\in X$, we define
		$$
		\|(u,v)\|_{\mathbf{X}}=\|u\|_{X}+\|v\|_{X}
		$$
		and
		$$
		\|(u,v)\|_{\mathbf{X}}^{2}=\|u\|_{X}^{2}+\|v\|_{X}^{2}.
		$$
		If $\Omega=\R^{n}$, for simplicity, we denote $\|u\|_{p}=\|u\|_{L^{p}(\Omega)}$ and $\|(u,v)\|_{p}=\|(u,v)\|_{\mathbf{L}^{p}(\Omega)}$. We denote $C$ as an absolute constant, while $C(\alpha_{1},\alpha_{2},\ldots,\alpha_{n})$ indicates that the constant depends on the parameters $\alpha_{1}, \alpha_{2}, \ldots, \alpha_{n}$.
		
		This paper is organized as follows: in Sect. 2, we provide some preliminaries,
		which includes some basic knowledge of Lorentz space, some standard facts about the Stokes operator, and several useful lemmas. In Sect. 3, we establish the existence of a weak solution to the perturbed Leray system and obtain its decay estimate by improving the regularity of the weak solution, which is the core of this paper.
		 Finally, we prove Theorem \ref{T1.1} in Sect. 4.
		\setcounter{equation}{0}
		\setcounter{equation}{0}
		\section{Preliminaries}
		\subsection{The Lorentz space}
		Let $\Omega\subseteq\R^{3}$ be a measurable set. For $1<p<+\infty$ and $1\leq q\leq+\infty$, the Lorentz space $L^{p,q}(\Omega)$ consists of all measurable functions with a finite norm
		\begin{align*}
			\|u\|_{L^{p,q}(\Omega)}=\left \{
			\begin{array}{ll}
				\Big(\int_{0}^{\infty}\big(t^{\frac{1}{p}}u^{\ast}(t)\big)^{\frac{1}{q}}\frac{dt}{t}\Big)^{\frac{1}{q}},\ \ & {\rm if}\ \ 1< q<\infty,\\
				\sup_{t>0}t^{\frac{1}{p}}u^{\ast}(t),\ \ & {\rm if}\ \ q=\infty,
			\end{array}
			\right.
		\end{align*}
		where $u^{\ast}(t)=\inf\{\varrho:|\{x\in\Omega:|u(x)|>\varrho\}|\leq t\}$ is nonincreasing and right semicontinuous. When $q=\infty$, an equivalent norm for the weak-$L^{p}$ space is
		$$
		\|u\|_{L^{p,\infty}(\Omega)}=\sup_{\varrho>0}\varrho|\{x\in\Omega:|u(x)|>\varrho\}|^{\frac{1}{p}}.
		$$
		In particular, for all $q\in[1,+\infty]$, we have $L^{q,q}(\Omega)$ is equivalent to the standard Lebesgue integrable function space $L^{q}(\Omega)$. 	Moreover, the Lorentz space can also be defined through interpolation theory (\cite{AF},\cite{LG1},\cite{LR}) by
		\begin{equation}\label{A-1}
			L^{p,q}(\Omega)=\big(L^{p_{0}}(\Omega),L^{p_{1}}(\Omega)\big)_{\theta,q}
		\end{equation}
		provided $1<p_{0}<p<p_{1}<\infty$ and $\frac{1}{p}=\frac{1-\theta}{p_{0}}+\frac{\theta}{p_{1}}$ for some $\theta\in(0,1)$. We can get various interesting properties from interpolation theory. For example
		\begin{equation}\label{A-01}
				L^{1}(\Omega)\cap  L^{\infty}(\Omega)\ \ \mbox{is dense in}\ L^{p,q}(\Omega),\ \ \ \forall\ 1<p<\infty \ \mbox{and}\ 1\leq q<\infty.
		\end{equation}
		Next, we recall the dual space of $L^{p,1}(\Omega)$, for $1<p<\infty$,
		$$
		\big(L^{p,1}(\Omega)\big)'=L^{p',\infty}(\Omega),\ \ p'=\frac{p}{p-1}.
		$$
		The definition of the dual space of a Lorentz space is quite complex, for other cases, see \cite{LG2} for details. The space $BC_{w}([0,\infty);L^{3,\infty}(\R^{3}))$ consists of functions that are bounded and weak-$\ast$ continuous in $L^{3,\infty}(\R^{3})$ with respect to time variable $t$.
		
			For $q\geq1$, we denote $D^{1,q}$ is the homogeneous Sobolev space of the form
		$$
		D^{1,q}(\Omega):=\{f\in W_{loc}^{1,q}(\Omega)| \nabla f\in L^{q}(\Omega)\}.
		$$
		We now introduce some fundamental inequalities in Lorentz spaces that are frequently used in this paper.
		\begin{lem}[\cite{AF},\cite{LG1},\cite{LG2},\cite{LR}]\label{HYI}
		\leavevmode\par
				\begin{enumerate}
				\item[(i)](H\"{o}lder inequality) Let $0<p,p_{1},p_{2}\leq\infty$, $0<q,q_{1},q_{2}\leq\infty$ satisfies $\frac{1}{p}=\frac{1}{p_{1}}+\frac{1}{p_{2}}$ and $\frac{1}{q}=\frac{1}{q_{1}}+\frac{1}{q_{2}}$. Then we have
				$$
				\|fg\|_{L^{p,q}(\R^{3})}\leq C(p_{1},p_{2},q_{1},q_{2})\|f\|_{L^{p_{1},q_{1}}(\R^{3})}\|g\|_{L^{p_{2},q_{2}}(\R^{3})}.
				$$
				\item[(ii)](Young's inequality) Let $1<p,p_{1},p_{2}<\infty$, $0<q,q_{1},q_{2}\leq\infty$ satisfies $\frac{1}{p}+1=\frac{1}{p_{1}}+\frac{1}{p_{2}}$ and $\frac{1}{q}=\frac{1}{q_{1}}+\frac{1}{q_{2}}$. Then we have
				$$
				\|f\ast g\|_{L^{p,q}(\R^{3})}\leq C(p_{1},p_{2},q_{1},q_{2})\|f\|_{L^{p_{1},q_{1}}(\R^{3})}\|g\|_{L^{p_{2},q_{2}}(\R^{3})}.
				$$
				\item[(iii)](Young's inequality at endpoint) Let $1<p<\infty$, $1\leq q\leq\infty$, satisfies  $\frac{1}{p}+\frac{1}{p_{1}}=1$ and $\frac{1}{q}+\frac{1}{q_{1}}=1$. Then we have
				$$
				\|f\ast g\|_{L^{\infty}(\R^{3})}\leq C(p,q,p_{1},q_{1})\|f\|_{L^{p,q}(\R^{3})}\|g\|_{L^{p_{1},q_{1}}(\R^{3})}
				$$
				and
				$$
				\|f\ast g\|_{L^{p,q}(\R^{3})}\leq C(p,q)\|f\|_{L^{1}(\R^{3})}\|g\|_{L^{p,q}(\R^{3})}.
				$$
				\end{enumerate}
		\end{lem}
		\begin{lem}[\cite{LG1}](Sobolev-Gagliardo-Nirenberg embedding)\label{SGN}
			Let $1\leq p<n$. Then there exists a constant $C=C(n,p)$ such that for every function $u\in D^{1,p}(\R^{n})$ vanishing at infinity,
			$$
			\|u\|_{L^{\frac{np}{n-p},p}(\R^{n})}\leq C\|\nabla u\|_{L^{p}(\R^{n})}.
			$$
		\end{lem}
	\subsection{The fundamental solution of the non-stationary Stokes equations}
	Let us start by recalling the Gauss-Weierstrass kernel
	$$
	\Gamma_{t}(x)=\Gamma(x,t)=\frac{1}{(4\pi t)^{\frac{3}{2}}}e^{-\frac{|x|^{2}}{4t}},\ \ \ (x,t)\in\R^{3}\times(0,+\infty),
	$$
	which is self-similar and in particular
	\begin{equation*}
		\Gamma_{t}(x)=t^{-\frac{3}{2}}\Gamma_{1}\Big(\frac{x}{\sqrt{t}}\Big).
	\end{equation*}
	Let $\mathbb{Z}^{+}=\{0,1,2,...\}$ be the set of non-negative integers, we have
	\begin{equation}\label{heat1}
		|D^{k}\Gamma_{1}(x)|\leq C(k)(1+|x|)^{-5-k},\ \ \forall\ k\in\mathbb{Z}^{+}.
	\end{equation}
	Thus, we can immediately obtain
	the well-known pointwise estimate of the heat kernel
	\begin{equation}\label{heat2}
		|\partial_{t}^{l}D_{x}^{k}\Gamma_{t}(x)|\leq C(k,l)t^{-l}(\sqrt{t}+|x|)^{-3-k},\ \ \forall  \ k,l\in\mathbb{Z}^{+}.
	\end{equation}
	Next, we consider the nonstationary Stokes system as follows:
		\begin{align}\label{stokes-1}
			\begin{split}
				\left.
				\begin{aligned}
					\partial_{t} v-\Delta v+\nabla q=f\\
					\mbox{div}\,v=0\\
					v(x,0)=0\\
				\end{aligned}\ \right\}\ \mbox{in}\ \R^{3}\times (0,+\infty),
			\end{split}
		\end{align}
		where $f(x,t)=0$ for $t<0$. It is well-know that a solution $(v,q)$ has the form
		$$
		v(x,t)=\int_{0}^{t}e^{(t-s)\Delta}\mathbb{P}f(\cdot,s)ds=\int_{0}^{t}\int_{\R^{n}}S_{ij}(x-y,t-s)f_{j}(y,s)dyds
		$$
		and
		$$
		q(x,t)=-\int_{\R^{3}}\partial_{j}\Phi(x-y)f_{j}(y,t)dy,
		$$
		where
		$$
		\mathbb{P}=\mbox{Id}-\nabla(-\Delta)^{-1}\mbox{div}
		$$
		is the Leray projector,  $\Phi(x)=\frac{1}{4\pi|x|}$  is the fundamental solution of $-\Delta$ in $\R^{3}$, and
		$$
		S(x,t)=S_{ij}(x,t)=\delta_{ij} \Gamma(x,t)+\partial_{i}\partial_{j}\int_{\R^{3}}\Phi(x-y)G(y,t)dy
		$$
		is the Oseen kernel \cite{O}. It is similar to the heat kernel, for the Oseen kernel, we also have
		$$
		S_{ij}(x,t)=t^{-\frac{3}{2}}S_{ij}\Big(\frac{x}{\sqrt{t}}\Big).
		$$			
		and
		\begin{equation}\label{oseen}
			|\partial_{t}^{l}D_{x}^{k}S(x,t)|\leq C(k,l)t^{-l}(\sqrt{t}+|x|)^{-3-k},\ \ \forall  \ k,l\in\mathbb{Z}^{+}.
		\end{equation}
		This pointwise estimate was first proved by \cite{VAS}. Moreover, Dong and Li \cite{DL} was also studied  the generalized
		Oseen kernel and
		establish the following pointwise estimate
		$$
		|D^{k}(-\Delta)^{\frac{\beta}{2}}S(x,1)|\leq C(k,\beta)(1+|x|)^{-3-k-\beta},\ \ \beta\in(-1,1].
		$$
		According to scaling property and \eqref{oseen}, we can derive
		\begin{equation}\label{frac}
			|D_{x}^{k}(-\Delta)^{\frac{\beta}{2}}S(x,t)|\leq C(k,\beta)(\sqrt{t}+|x|)^{-3-k-\beta},\ \ \beta\in(-1,1].
		\end{equation}	
		Here, $(-\Delta)^\alpha$ is a pseudodifferential operator defined in terms of the Fourier transform
		$$
		\widehat{(-\Delta)^\alpha} u(\xi)=|\xi|^{2\alpha} \widehat{u}(\xi), u \in \mathcal{S}(\R^{n}),
		$$
		where $\mathcal{S}(\R^{n})$ represents the Schwartz space of smooth functions that decrease rapidly, either real or complex-valued. 
		Alternatively, if $0<\alpha\leq1$, we can equivalently define 
		the $(-\Delta)^\alpha$ as follows:
		$$
		(-\Delta)^\alpha u(x) \triangleq c_{n, \alpha} \text { p.v. } \int_{\mathbb{R}^n}(u(x)-u(z)) \frac{\mathrm{d} z}{|x-z|^{n+2 \alpha}},\  u \in \mathcal{S}(\R^{n}),
		$$
		where p.v. stands for the Cauchy principle value and $c_{n, \alpha}$ is a normalization constant.
		For more information on fractional Laplacian $(-\Delta)^\alpha$, interested readers are referred to \cite{GN}.
		The pointwise estimate \eqref{frac} plays a key role in removing the logarithmic loss 
	   of the self-similar solutions, see Lemma \ref{stokes} in Section 3.
		\subsection{Some useful lemmas}
	In this subsection, we present several important lemmas that are useful in proving the existence and pointwise estimates of self-similar solutions.
	Consider a function $f(x)$ which is homogeneous of degree $-1$, i.e.
	$$
	f(x)=\frac{\kappa(\bar{x})}{|x|}, \ \ \bar{x}=\frac{x}{|x|},
	$$
	we denote 
	$$
	f_{I}(x,t)=e^{t\Delta}f=\Gamma_{t}\ast f \ \ \mbox{and} \ \  F(x)=f_{I}(x,1)=\Gamma_{1}\ast f.
	$$
	Next, we will investigate some properties of $f_{I}(x,t)$ and $F(x)$.
		\begin{lem}\label{inital}
			Assume that $f\in L^{\infty}_{loc}(\R^{3}\setminus\{0\})$ be homogeneous of degree $-1$, with $\mathrm{div} f=0$.  Then,
			\begin{enumerate}
				\item[(i)]for all $x\in\R^{3}$, 
				$F(x)\in C^{\infty}(\R^{3})$, $\mathrm{div} F=0$ and
				\begin{equation}\label{inital-1}
					|\nabla^{k} F(x)|\leq C(k)(1+|x|)^{-1},\ \ \forall \ k\in\mathbb{Z}^{+}.
				\end{equation}
				Moreover, for all $(x,t)\in\R^{3}\times(0,\infty)$, we have
				\begin{equation}\label{inital-1-1}
					f_{I}(x,t)\in BC_{w}([0,+\infty),L^{3,\infty}(\R^{3}));
				\end{equation}
				\item[(ii)]if $f\in C_{loc}^{0,1}(\R^{3}\setminus\{0\})$, i.e. $\kappa(x)=\kappa(x/|x|)\in C^{0,1}(\mathbb{S}^{2})$, then, we have
				\begin{equation}\label{inital-2}
					|\nabla^{k}F(x)|\leq C(1+|x|)^{-2},\ \ \forall \ k\geq1,
				\end{equation}
				where the constant $C$ depend on $k$ and $\|\kappa\|_{C^{0,1}(\mathbb{S}^{2})}$.
				\item[(iii)]if $f\in C_{loc}^{1,\alpha}(\R^{3}\setminus\{0\})$ with $0<\alpha\leq1$, then, for any $\beta\in(0,\alpha)$
				\begin{equation}\label{inital-3}
					|\nabla^{k}(-\Delta)^{\frac{\beta}{2}}F(x)|\leq C(\alpha,\beta)(1+|x|)^{-1-k-\beta},\ \ k=0,1.
				\end{equation}
				In particular, if $\alpha=1$, then
				\begin{equation}\label{inital-4}
					|\nabla^{2}F(x)|\leq C(1+|x|)^{-3}.
				\end{equation}
			\end{enumerate}
		\end{lem}
		\begin{proof}
			The proof of this lemma is essential, for detail, pleas see \cite{Lai,LMZ,LMZ1}.
			
			Here, we provide an alternative proof of the estimates \eqref{inital-1} and \eqref{inital-2} that is more straightforward.
			We  will focus solely on the proof of \eqref{inital-2}, as \eqref{inital-1} can be obtained in a similar way and is, in fact, simpler.  
			Thanks to the rapid decay of the exponential function, it follows that yields
			$$
			|x|^{k}\Gamma_{1}(x)\in L^{p}(\R^{3}), \ \ \mbox{for all}\ \ 1\leq p\leq\infty\ \ \mbox{and}\ \ k\geq0.
			$$
			Thus, for all $1<p<\infty$ and $1\leq q\leq\infty$, we have
			$$
			|x|^{k}\Gamma_{1}(x)\in L^{p,q}(\R^{3}),
			$$
			where we use the fact \eqref{A-01}. 
			Since
			$$
			F(x)=\int_{\R^{3}}\Gamma_{1}(x-y)f(y)dy.
			$$
			On one hand, by applying Lemma \ref{HYI} we can drive
			\begin{align*}
				\begin{split}
					|\nabla^{k}F(x)|&\leq C\sum_{l=0}^{k}\int_{\R^{3}}|x-y|^{l}\Gamma_{1}(x-y)|f(y)|dy\\&\leq C\sum_{l=0}^{k}\Big\||\cdot|^{l}\Gamma_{1}\Big\|_{L^{\frac{3}{2},1}(\R^{3})}\Big\|\frac{1}{|\cdot|}\Big\|_{L^{3,\infty}(\R^{3})}\leq C(k).
				\end{split}
			\end{align*}	
		On the other hand, 
		for any $\varepsilon>0$
			\begin{align*}
				\begin{split}
				&\int_{\R^{3}\setminus B_{\varepsilon}}\nabla_{x}^{k}\Gamma_{1}(x-y)f(y)dy	\\=&-\int_{\R^{3}\setminus B_{\varepsilon}}\nabla_{y}\nabla_{x}^{k-1}\Gamma_{1}(x-y)f(y)dy
					\\=&-\int_{\partial B_{\varepsilon}}\nabla_{x}^{k-1}\Gamma_{1}(x-y)f(y)\nu dS+\int_{\R^{3}\setminus B_{\varepsilon}}\nabla_{x}^{k-1}\Gamma_{1}(x-y)\nabla f(y)dy
				\end{split}\\=&I_{\varepsilon}+J_{\varepsilon}, 
			\end{align*}
			where $\nu$ is the unit outward normal vector of $\partial B_{\varepsilon}$.
		Letting $\varepsilon\rightarrow0^{+}$,	we readily verify
			$$
			|I_{\varepsilon}|\leq \|\nabla^{k}\Gamma_{1}(x)\|_{\infty}\int_{\partial B_{\varepsilon}}|f(y)|dS\leq C(k)\varepsilon\rightarrow0
			$$
		and
		$$
		J_{\varepsilon}\rightarrow \int_{\R^{3}}\nabla_{x}^{k-1}\Gamma_{1}(x-y)\nabla f(y)dy,
		$$
		which enable us to derive
			\begin{align*}
				\begin{split}
					\nabla^{k}F(x)=\int_{\R^{3}}\nabla^{k-1}\Gamma_{1}(x-y)\nabla f(y)dy.
				\end{split}
			\end{align*}
		Thus
			\begin{align*}
				\begin{split}
					\Big||x|^{2}\nabla^{k}F(x)\Big|=&\Big|\int_{\R^{3}}\big(|x-y|^{2}+|y|^{2}+2(x-y)\cdot y\big)\nabla^{k-1}\Gamma_{1}(x-y)\nabla f(y)dy\Big|\\ \leq& C\bigg(\Big|\int_{\R^{3}}|x-y|^{2}\nabla^{k-1}\Gamma_{1}(x-y)\frac{1}{|y|^{2}}dy \Big|+\Big|\int_{\R^{3}}\nabla^{k-1}\Gamma_{1}(x-y)|y|^{2}\frac{1}{|y|^{2}}dy \Big|\\&+\Big|\int_{\R^{3}}|x-y|\nabla^{k-1}\Gamma_{1}(x-y)|y|\frac{1}{|y|^{2}}dy \Big|\bigg)\\ \leq& C\bigg(\sum_{l=2}^{k+1}\Big\||\cdot|^{l}\Gamma_{1}\Big\|_{L^{3,1}(\R^{3})}\Big\|\frac{1}{|\cdot|^{2}}\Big\|_{L^{\frac{3}{2},\infty}(\R^{3})}+\sum_{n=0}^{k-1}\Big\||\cdot|^{n}\Gamma_{1}\Big\|_{L^{1}(\R^{3})}\\&+\sum_{m=1}^{k}\Big\||\cdot|^{m}\Gamma_{1}\Big\|_{L^{\frac{3}{2},1}(\R^{3})}\Big\|\frac{1}{|\cdot|}\Big\|_{L^{3,\infty}(\R^{3})}\bigg)\\ \leq& C,
				\end{split}
			\end{align*}
			where the constant $C$  depends on $\|\kappa\|_{W^{1,\infty}(\mathbb{S}^{2})}$ and $k$, does not depend on $x$. Therefore, we conclude \eqref{inital-2}.
		\end{proof}
		Next, we introduce a crucial lemma that plays a significant role in our proof of {\em a priori} estimates.
		\begin{lem}\label{L0}
			Let $\Omega$ be a connected domain in $\R^{3}$ with Lipschitz boundary, the function $(v,b)\in \mathbf{H}^{1}_{0,\sigma}(\Omega)$ and $p\in D^{1,\frac{3}{2}}(\Omega)\cap L^{3}(\Omega)$ satisfy the following system
			\begin{equation}\label{L1}
				\begin{cases}
					(v\cdot\nabla) v-(b\cdot\nabla) b+\nabla p=0\ \ \ \ &\mbox{in}\ \ \Omega,\\
					(v\cdot\nabla) b-(b\cdot\nabla) v=0\ \ \ \ &\mbox{in}\ \ \Omega,\\
					{\rm div}\,v={\rm div}\,b=0\ \ \ \ \ &\mbox{in}\ \ \Omega,\\
					v=b=0 \ \ \ \ &\mbox{on}\ \ \partial\Omega.
				\end{cases}
			\end{equation}
			Then,
			$$
			\exists\ \  c\in\R \ \mbox{such that}\ p(x)=c\ \mbox{for $\mathcal{H}^{2}$-almost all} \ x\in\partial\Omega,
			$$
			where $\mathcal{H}^{m}$ denoted by the $m$-dimensional Hausdorff measure.
		\end{lem}
		\begin{proof}
			The proof of this lemma follows the ideas presentation in \cite{Amick,ABLS,LMZ}. It is similar to the stationary Euler equation, with a few minor differences.
			
			 For each  $x_{0}\in\partial\Omega$, we can select a new orthogonal coordinate system $y=(y',y_{3})=(y_{1},y_{2},y_{3})$, up to rotation and translation, such that the origin is in $x_{0}$ and the $y_{3}$-axis aligns with the inward normal to $\partial\Omega$ at $x_{0}$. Then we consider the local coordinate in $x_{0}$, that is, for sufficiently small $r,\delta>0$, there exists a $C^{0,1}$-function $h(y')$ such that
			$$
			\Sigma=\Sigma_{r,\delta}:=\{(y',y_{3})\in\R^{3};h(y')<y_{3}<h(y')+\delta,|y'|<r\}\subset\widetilde{\Omega},
			$$
			where $\widetilde{\Omega}$ represents the domain of $\Omega$ in the new coordinate system $y$. We also have $\big(v(y),b(y),p(y)\big)$ solves the equation \eqref{L1} in $\widetilde{\Omega}$.
			For any scalar function $\varphi(y')\in C_{0}^{\infty}(B_{r})$ and $i=1,2$, integrating by parts yields
			\begin{align}\label{L2}
				\begin{split}
					\int_{\Sigma}p(y)\partial_{i}\varphi dy&=\int_{0}^{\delta}\int_{B_{r}}p(y',h(y')+\tilde{y}_{3})\partial_{i}\varphi dy'd\tilde{y}_{3}\\&=-\int_{0}^{\delta}\int_{B_{r}}\varphi(y')\big(\partial_{i}p+\partial_{3}p\partial_{i}h\big)dy'd\tilde{y}_{3}\\&
					=-\int_{\Sigma}\varphi(y')\big(\partial_{i}p+\partial_{3}p\partial_{i}h\big)dy.
				\end{split}
			\end{align}
			Then, by the Lebesgue theorem and \eqref{L2}, one can easily get
			\begin{align}\label{L4}
				\begin{split}
					\Big|\int_{B_{r}}p\big(y',h(y')\big)\partial_{i}\varphi dy'\Big|&=\Big|\lim_{\delta\rightarrow0}\frac{1}{\delta}\int_{\Sigma}p(y)\partial_{i}\varphi dy\Big|\\&
					=\Big|\lim_{\delta\rightarrow0}\frac{1}{\delta}\int_{\Sigma}\varphi(y')\big(\partial_{i}p+\partial_{3}p\partial_{i}h\big)dy\Big|\\&
					\leq C\|h\|_{W^{1,\infty}(B_{r})}\lim_{\delta\rightarrow0}\frac{1}{\delta}\int_{\Sigma}|\nabla p||\varphi(y')|dy\\&
					\leq C\lim_{\delta\rightarrow0}\int_{\Sigma}\frac{|b\cdot\nabla b|+|v\cdot\nabla v|}{|y_{3}-h(y')|}dy,
				\end{split}
			\end{align}
			here we use the fact that $h(y')\in W^{1,\infty}(B_{r})$. We claim that
			$$
			\lim_{\delta\rightarrow0}\int_{\Sigma}\frac{|b\cdot\nabla b|+|v\cdot\nabla v|}{|y_{3}-h(y')|}dy=0.
			$$
			By H\"{o}lder inequality, we have
			$$
			\int_{\Sigma}\frac{|b\cdot\nabla b|+|v\cdot\nabla v|}{|y_{3}-h(y')|}dy\leq\Big\|\frac{b}{y_{3}-h(y')}\Big\|_{L^{2}(\Sigma)}\|\nabla b\|_{L^{2}(\Sigma)}+\Big\|\frac{v}{y_{3}-h(y')}\Big\|_{L^{2}(\Sigma)}\|\nabla v\|_{L^{2}(\Sigma)}.
			$$
			In fact, according to the Hardy inequality, one can derive
			\begin{align*}
				\begin{split}
					\int_{\Sigma}\frac{b^{2}}{|y_{3}-h(y')|^{2}}dy&=\int_{B_{r}}\int_{h(y')}^{h(y')+\delta}\frac{b^{2}}{|y_{3}-h(y')|^{2}}dy_{3}dy'
					\\&=\int_{B_{r}}\int_{h(y')}^{h(y')+\delta}\frac{1}{|y_{3}-h(y')|^{2}}\Big(\int_{h(y')}^{y_{3}}\partial_{3}b(y',\cdot)d\tilde{y}_{3}\Big)^{2}dy_{3}dy'
					\\&\leq C\int_{B_{r}}\int_{h(y')}^{h(y')+\delta}|\partial_{3}b(y',\cdot)|^{2}dy_{3}dy'\\&
					\leq C\|\nabla b\|_{L^{2}(\Sigma)}^{2}.
				\end{split}
			\end{align*}
			Similarly,
			$$
			\int_{\Sigma}\frac{v^{2}}{|y_{3}-h(y')|^{2}}dy\leq C\|\nabla v\|_{L^{2}(\Sigma)}^{2}.
			$$
			Combining the above estimates yields
			$$
			\lim_{\delta\rightarrow0}\int_{\Sigma}\frac{|b\cdot\nabla b|+|v\cdot\nabla v|}{|y_{3}-h(y')|}dy\leq C\lim_{\delta\rightarrow0}(\|\nabla b\|^{2}_{L^{2}(\Sigma)}+\|\nabla v\|^{2}_{L^{2}(\Sigma)})=0.
			$$
			Let us return to \eqref{L4}, one obtains
			$$
			\int_{B_{r}}p\big(y',h(y')\big)\partial_{i}\varphi dy'=-\int_{B_{r}}\partial_{i}p\big(y',h(y')\big)\varphi dy'=0.
			$$
			Thus, due to the arbitrariness of $\varphi$, we conclude our proof.
		\end{proof}
		
		We conclude this section by introducing the well-known {\em Leray-Schauder theorem}.
		\begin{thm}[ Leray-Schauder theorem \cite{GT0} ]\label{T2.1}
			Let $\mathcal{S}$ be a compact mapping of a Banach space $X$ into itself, and suppose there exists a constant $M$ such that
			$$
			\|x\|_{X}<M
			$$
			for all $x\in X$ and $\lambda\in[0,1]$ satisfying $x=\lambda \mathcal{S}x$. Then $\mathcal{S}$ has a fixed point.
		\end{thm}
		\section{Analysis of the perturbed Leray system} 
		\subsection{Existence of weak solutions to the PLS in $\R^{3}$}
		In this subsection, we will construct a weak solution $(V,G,P)$ to the system \eqref{PLS} in $\R^{3}$. Let's start by constructing a more general weak solution, that is, the couple $(V,G)\in \mathbf{H}_{\sigma}^{1}(\R^{3})$ satisfies \eqref{PLS} in the following distributional sense
		\begin{align}\label{weak1}
			\begin{split}
				\int_{\R^{3}}&\nabla V:\nabla\varphi dx\\&=\int_{\R^{3}}\Big(\frac{1}{2}V+\frac{1}{2}x\cdot \nabla V-(V+U_{0})\cdot\nabla(V+U_{0})+(G+B_{0})\cdot\nabla(G+B_{0})\Big)\varphi dx,
			\end{split}
		\end{align}
		\begin{align}\label{weak2}
			\begin{split}
				\int_{\R^{3}}&\nabla G:\nabla\psi dx\\&=\int_{\R^{3}}\Big(\frac{1}{2}G+\frac{1}{2}x\cdot \nabla G-(V+U_{0})\cdot\nabla(G+B_{0})+(G+B_{0})\cdot\nabla(V+U_{0})\Big)\psi dx,
			\end{split}
		\end{align}
		for all $(\varphi,\psi)\in \mathbf{C}_{0,\sigma}^{\infty}(\R^{3})$. 
		It is important to note that the difference between the weak formulation presented above and Definition \ref{D-0} lies in our choice of a divergence-free test function, which allows us to skip the discussion of pressure. In fact, if we establish the existence of a weak solution $(V,G)\in \mathbf{H}_{\sigma}^{1}(\R^{3})$, the existence of pressure $P$ can be derived from the well-known De Rham theorem \cite{TR}. Consequently, we can obtain a weak solution that satisfies Definition \ref{D-0} through a suitable approximation.
		
		Now, we provide a few simple properties of self-similar solutions to the heat equation.
		\begin{prop}
		Let $(u_{0},b_{0})\in \mathbf{L}^{\infty}_{loc}(\R^{3}\setminus\{0\})$ be homogeneous of degree -1 with $\mathrm{div} u_{0}=\mathrm{div} b_{0}=0$. Then the profiles $(U_{0},B_{0})=(e^{\Delta u_{0}},e^{\Delta b_{0}})$ satisfy the following system
		\begin{align}\label{LS2}
			\begin{split}
				\left.
				\begin{aligned}
					\Delta U_{0}+\frac{1}{2}U_{0}+\frac{1}{2}x\cdot\nabla U_{0}=0\\
					\Delta B_{0}+\frac{1}{2}B_{0}+\frac{1}{2}x\cdot\nabla B_{0}=0\\
					\mathrm{div}\,U_{0} =\mathrm{div}\,B_{0} =0\\
				\end{aligned}\ \right\}\ \ \mbox{in}\ \ \ \R^{3}.
			\end{split}
		\end{align}
		Moreover, 
		\begin{align}\label{initial1}
			\begin{split}
			|\nabla^{k} U_{0}(x)|+|\nabla^{k}B_{0}(x)|\leq C(k)(1+|x|)^{-1},\ \ \forall \ k\in\mathbb{Z}^{+}.
	\end{split}
\end{align}
		\end{prop}
		\begin{proof}
			Setting
		\begin{equation}\label{LI}
			\begin{cases}
				u_{I}(x,t)=e^{t\Delta}u_{0}=\Gamma_{t}\ast u_{0}\\
				b_{I}(x,t)=e^{t\Delta}b_{0}=\Gamma_{t}\ast b_{0},
			\end{cases}
		\end{equation}
		which solve the heat equation with the initial condition $u_{0}$ and $b_{0}$, respectively. Then, the scaling properties and Lemma \ref{inital} 
		allow us to directly prove this proposition.
		\end{proof}
		Next, we will construct a weak solution to \eqref{PLS} through the {\em Leray-Schauder theorem}. 
	Let $\Omega\subset\R^{3}$ be a bounded domain with a smooth boundary $\partial\Omega$.  
	Our first goal is to derive {\em a priori} estimates independent of $\lambda\in[0,1]$ of weak solutions to the following system
		\begin{equation}\label{ESp2}
			\begin{cases}
				-\Delta V+\nabla P=\lambda\Big(\frac{1}{2}V+\frac{1}{2}x\cdot\nabla V-U_{0}\cdot\nabla U_{0}+B_{0}\cdot\nabla B_{0}-F_{1}+F_{2}\Big)\ &\mbox{in} \ \ \Omega,\\
				-\Delta G=\lambda\Big(\frac{1}{2}G+\frac{1}{2}x\cdot\nabla G-U_{0}\cdot\nabla B_{0}+B_{0}\cdot\nabla U_{0}-F_{3}+F_{4}\Big)\ &\mbox{in} \ \ \Omega,\\
				{\rm div}\,V={\rm div}\,G=0\ &\mbox{in} \ \ \Omega,\\
				V=G=0 \ &\mbox{on} \ \ \partial\Omega,
			\end{cases}
		\end{equation}
		where, for simplicity, we set
		\begin{align}\label{Sim}
			\begin{split}
				\ F_{1}=(V+U_{0})\cdot\nabla V+V\cdot\nabla U_{0}, \ \ F_{2}=(G+B_{0})\cdot\nabla G+G\cdot\nabla B_{0},\\
				\ F_{3}=(V+U_{0})\cdot\nabla G+V\cdot\nabla B_{0}, \ \ F_{4}=(G+B_{0})\cdot\nabla V+G\cdot\nabla U_{0}.
			\end{split}
		\end{align}
		\begin{lem}[a priori estimate]\label{L5}
			Let $\Omega$ be a bounded domain in $\R^{3}$ with a smooth boundary. Assume that $(U_{0},B_{0})$ satisfies the decay estimates \eqref{initial1}. Let $\lambda\in[0,1]$ and $(V,G)\in \mathbf{H}_{0,\sigma}^{1}(\Omega)$ be a solution to problem \eqref{ESp2}. Then we have
			\begin{equation}\label{LE1}
				\|(V,G)\|_{\mathbf{H}_{0,\sigma}^{1}(\Omega)}
				\leq C,
			\end{equation}
			where the constant $C=C(\Omega,U_{0},B_{0})$, independent on $\lambda$.
		\end{lem}
		\begin{proof}
			In fact, by Poincar\'{e} inequation, it is sufficient to prove
			$$
			\|(\nabla V,\nabla G)\|_{\mathbf{L}^{2}(\Omega)}
			\leq C.
			$$
			We will argue by contradiction. Assume that the assertion is false. Then, there exists a sequence $\lambda_{k}\subset[0,1]$ and a sequence of solutions $(V_{k},G_{k})\subset \mathbf{H}_{0,\sigma}^{1}(\Omega)$ to problem \eqref{ESp2} with $\lambda_{k}$ instead of $\lambda$, such that
			\begin{equation}\label{ESp3}
				\begin{cases}
					-\Delta V_{k}+\nabla P_{k}=\lambda_{k}\Big(\frac{1}{2}V_{k}+\frac{1}{2}x\cdot\nabla V_{k}-U_{0}\cdot\nabla U_{0}+B_{0}\cdot\nabla B_{0}-F_{1k}+F_{2k}\Big)\ &\mbox{in} \ \ \Omega,\\
					-\Delta G_{k}=\lambda_{k}\Big(\frac{1}{2}G_{k}+\frac{1}{2}x\cdot\nabla G_{k}-U_{0}\cdot\nabla B_{0}+B_{0}\cdot\nabla U_{0}-F_{3k}+F_{4k}\Big)\ &\mbox{in} \ \ \Omega,\\
					{\rm div}\,V_{k}={\rm div}\,G_{k}=0\ &\mbox{in} \ \ \Omega,\\
					V_{k}=G_{k}=0 \ &\mbox{on} \ \ \partial\Omega,
				\end{cases}
			\end{equation}
			where it is similar to \eqref{Sim}
			\begin{align*}\label{Sim2}
				\begin{split}
					F_{1k}=(V_{k}+U_{0})\cdot\nabla V_{k}+V_{k}\cdot\nabla U_{0}, \ \ F_{2k}=(G_{k}+B_{0})\cdot\nabla G_{k}+G_{k}\cdot\nabla B_{0},\\
					F_{3k}=(V_{k}+U_{0})\cdot\nabla G_{k}+V_{k}\cdot\nabla B_{0}, \ \ F_{4k}=(G_{k}+B_{0})\cdot\nabla V_{k}+G_{k}\cdot\nabla U_{0}.
				\end{split}
			\end{align*}
			Moreover,
			\begin{equation*}
				I_{k}^{2}:=\int_{\Omega}(|\nabla V_{k}|^{2}+|\nabla G_{k}|^{2})dx\rightarrow+\infty.
			\end{equation*}
			We normalized $(V_{k},G_{k})$ by defining
			\begin{equation*}
				\hat{V}_{k}=\frac{V_{k}}{I_{k}}, \ \ \hat{G}_{k}=\frac{G_{k}}{I_{k}}.
			\end{equation*}
			Since $(\hat{V}_{k},\hat{G}_{k})$ is bounded in $\mathbf{H}_{0,\sigma}^{1}(\Omega)$. Then we can extract a subsequence, still denoted by $(\hat{V}_{k},\hat{G}_{k})$ such that
			\begin{equation}\label{wc}
				(\hat{V}_{k},\hat{G}_{k})\rightharpoonup(\bar{V},\bar{G}),\ \ \mbox{in}\ \ \mathbf{H}_{0,\sigma}^{1}(\Omega),
			\end{equation}
			and by the compact embedding theorem
			\begin{equation}\label{sc}
				(\hat{V}_{k},\hat{G}_{k})\rightarrow(\bar{V},\bar{G}),\ \ \mbox{in}\ \ \mathbf{L}^{p}(\Omega),\ \ \mbox{for} \ \ p\in[2,6).
			\end{equation}
		We can also assume that $\lambda_{k}\rightarrow\lambda_{0}\in[0,1]$.
			By multiplying the first equation of the system \eqref{ESp3} by $V_{k}$ and the second equation of the system \eqref{ESp3} by $G_{k}$, respectively, we can derive
			\begin{align}\label{LE9}
				\begin{split}
					\frac{\lambda_{k}}{4}\int_{\Omega}&|V_{k}|^{2}dx+\int_{\Omega}|\nabla V_{k}|^{2}dx\\&=\lambda_{k}\Big(\int_{\Omega}(B_{0}\cdot\nabla B_{0}-U_{0}\cdot\nabla U_{0})V_{k}dx+\int_{\Omega}(F_{2k}-F_{1k})V_{k}dx\Big)
				\end{split}
			\end{align}
			and
			\begin{align}\label{LE10}
				\begin{split}
					\frac{\lambda_{k}}{4}\int_{\Omega}&|G_{k}|^{2}dx+\int_{\Omega}|\nabla G_{k}|^{2}dx\\&=\lambda_{k}\Big(\int_{\Omega}(B_{0}\cdot\nabla U_{0}-U_{0}\cdot\nabla B_{0})G_{k}dx+\int_{\Omega}(F_{4k}-F_{3k})G_{k}dx\Big).
				\end{split}
			\end{align}
			By summing  equations \eqref{LE9} and \eqref{LE10}, we proceed to multiply the resulting equation by $\frac{1}{I_{k}^{2}}$ and subsequently analyze the limit as $k\rightarrow\infty$. We investigate the limit of each term using \eqref{wc}-\eqref{sc} and the assumption of $(U_{0},B_{0})$. We can easily get
			\begin{align*}
				\begin{split}
					&\frac{1}{I_{k}^{2}}(\int_{\Omega}|\nabla V_{k}|^{2}dx+\int_{\Omega}|\nabla G_{k}|^{2}dx)=1,\\&
					\frac{1}{I_{k}^{2}}\int_{\Omega}|V_{k}|^{2}dx\rightarrow\int_{\Omega}|\bar{V}|^{2}dx,\ \ \ \frac{1}{I_{k}^{2}}\int_{\Omega}|G_{k}|^{2}dx\rightarrow\int_{\Omega}|\bar{G}|^{2}dx.
				\end{split}
			\end{align*}
			From \eqref{initial1}, one can conclude that
			$$
			U_{0},\nabla U_{0},B_{0},\nabla B_{0}\in L^{4}(\R^{3}).
			$$
			Thus,
			\begin{align*}
				\begin{split}
					\frac{1}{I_{k}^{2}}\int_{\Omega}(B_{0}\cdot\nabla B_{0}-U_{0}\cdot\nabla U_{0})V_{k}dx\rightarrow0,\\ \frac{1}{I_{k}^{2}}\int_{\Omega}(B_{0}\cdot\nabla U_{0}-U_{0}\cdot\nabla B_{0})G_{k}dx\rightarrow0.
				\end{split}
			\end{align*}
				Moreover,  by applying the divergence free condition, we can derive
			$$
			\int_{\Omega}(V_{k}+U_{0})\cdot\nabla V_{k} V_{k}dx=\int_{\Omega}(V_{k}+U_{0})\cdot\nabla G_{k} G_{k}dx=0
			$$
			and
			$$
			\int_{\Omega}(G_{k}+B_{0})\cdot\nabla G_{k} V_{k}dx+\int_{\Omega}(G_{k}+B_{0})\cdot\nabla V_{k} G_{k}dx=0.
			$$
		According to \eqref{initial1} and \eqref{sc}, we have
			$$
		\frac{1}{I_{k}^{2}}\int_{\Omega}F_{1k}V_{k}dx+\frac{1}{I_{k}^{2}}\int_{\Omega}F_{3k}G_{k}dx\rightarrow\int_{\Omega}(\bar{V}\cdot\nabla U_{0})\bar{V}dx+\int_{\Omega}(\bar{V}\cdot\nabla B_{0})\bar{G}dx
		$$
		and
			$$
		\frac{1}{I_{k}^{2}}\int_{\Omega}F_{2k}V_{k}dx+\frac{1}{I_{k}^{2}}\int_{\Omega}F_{4k}G_{k}dx\rightarrow\int_{\Omega}(\bar{G}\cdot\nabla B_{0})\bar{V}dx+\int_{\Omega}(\bar{G}\cdot\nabla U_{0})\bar{G}dx.
		$$
			The calculations presented above yield the following identity
			\begin{align}\label{LE13}
				\begin{split}
					&1+\frac{\lambda_{0}}{4}\int_{\Omega}|\bar{V}|^{2}dx+\frac{\lambda_{0}}{4}\int_{\Omega}|\bar{G}|^{2}dx\\
					=&-\lambda_{0}\Big(\int_{\Omega}\bar{V}\cdot\nabla U_{0}\bar{V}dx-\int_{\Omega}\bar{G}\cdot\nabla B_{0}\bar{V}dx+\int_{\Omega}\bar{V}\cdot\nabla B_{0}\bar{G}dx-\int_{\Omega}\bar{G}\cdot\nabla U_{0}\bar{G}dx\Big)\\
					=&\lambda_{0}\Big(\int_{\Omega}\bar{V}\cdot\nabla\bar{V}U_{0}dx-\int_{\Omega}\bar{G}\cdot\nabla\bar{G}U_{0}dx
					+\int_{\Omega}\bar{V}\cdot\nabla\bar{G}B_{0}dx-\int_{\Omega}\bar{G}\cdot\nabla\bar{V}B_{0}dx\Big),
				\end{split}
			\end{align}
			which implies that $\lambda_{0}\neq0$. Therefore, for sufficiently large $k$, we have $\lambda_{k}\neq0$, which allows us to normalize the pressure by setting
			$$
			\hat{P}_{k}=\frac{P_{k}}{\lambda_{k}I_{k}^{2}}.
			$$
			Next, we return to the first and second equations in \eqref{ESp3}. Dividing by $\lambda_{k}I_{k}^{2}$, we obtain
			\begin{align}\label{LE14}
				\begin{split}
					\hat{V}_{k}\cdot\nabla\hat{V}_{k}&-\hat{G}_{k}\cdot\nabla \hat{G}_{k}+\nabla\hat{P}_{k}=\frac{1}{I_{k}}\Big(\frac{\Delta \hat{V}_{k}}{\lambda_{k}}+\frac{1}{2}\hat{V}_{k}+\frac{1}{2}x\cdot\nabla\hat{V}_{k}\\&+\frac{1}{I_{k}}B_{0}\cdot\nabla B_{0}-\frac{1}{I_{k}}U_{0}\cdot\nabla U_{0}-U_{0}\nabla\hat{V}_{k}-\hat{V}_{k}\nabla U_{0}+B_{0}\nabla\hat{G}_{k}+\hat{G}_{k}\nabla B_{0}\Big)\\
				\end{split}
			\end{align}
			and
			\begin{align}\label{LE15}
				\begin{split}
					\hat{V}_{k}\cdot\nabla\hat{G}_{k}&-\hat{G}_{k}\cdot\nabla\hat{V}_{k}=\frac{1}{I_{k}}\Big(\frac{\Delta \hat{G}_{k}}{\lambda_{k}}+\frac{1}{2}\hat{G}_{k}+\frac{1}{2}x\cdot\nabla\hat{G}_{k}\\&+\frac{1}{I_{k}}B_{0}\cdot\nabla U_{0}-\frac{1}{I_{k}}U_{0}\cdot\nabla B_{0}-U_{0}\nabla\hat{G}_{k}-\hat{V}_{k}\nabla B_{0}+B_{0}\nabla\hat{V}_{k}+\hat{G}_{k}\nabla U_{0}\Big).
				\end{split}
			\end{align}
			Now, we will consider the weak formulation of equations \eqref{LE14} and \eqref{LE15}. Specifically,  testing equation \eqref{LE14} with an arbitrary function $\varphi\in C_{0,\sigma}^{\infty}(\Omega)$. After integrating by parts and taking the limit as $k\rightarrow\infty$, we obtain
			$$
			\frac{1}{I_{k}}\int_{\Omega}\Big(\frac{\Delta \hat{V}_{k}}{\lambda_{k}}+\frac{1}{2}\hat{V}_{k}+\frac{1}{2}x\cdot\nabla\hat{V}_{k}\Big)\varphi dx\rightarrow0.
			$$
			Thanks to \eqref{initial1}, we have
			$$
			\frac{1}{I^{2}_{k}}\int_{\Omega}\Big(B_{0}\cdot\nabla B_{0}-U_{0}\cdot\nabla U_{0}\Big)\varphi dx\rightarrow0,
			$$
			and
			$$
			\frac{1}{I_{k}}\int_{\Omega}\Big(-U_{0}\nabla\hat{V}_{k}-\hat{V}_{k}\nabla U_{0}+B_{0}\nabla\hat{G}_{k}+\hat{G}_{k}\nabla B_{0}\Big)\varphi dx\rightarrow0.
			$$
			Combined with the fact $\int_{\Omega}\nabla\hat{P}_{k}\cdot\varphi dx=0$, we get
			\begin{equation}\label{LE16}
				\int_{\Omega}(\bar{V}\cdot\nabla\bar{V}-\bar{G}\cdot\nabla\bar{G})\varphi dx=0,\ \ \ \mbox{for all}\ \ \varphi\in C_{0,\sigma}^{\infty}(\Omega).
			\end{equation}
			Similarly, testing with an arbitrary $\psi\in C_{0,\sigma}^{\infty}(\Omega)$ to \eqref{LE15}, we have
			\begin{equation}\label{LE17}
				\int_{\Omega}(\bar{V}\cdot\nabla\bar{G}-\bar{G}\cdot\nabla\bar{V})\psi dx=0,\ \ \ \mbox{for all}\ \ \psi\in C_{0,\sigma}^{\infty}(\Omega).
			\end{equation}
			The equations \eqref{LE16} and \eqref{LE17} imply that $(\bar{V},\bar{G})\in \mathbf{H}_{0,\sigma}^{1}(\Omega)$ is a weak solution of the system \eqref{L1}. Based on the well-known De Rham Theorem \cite{TR}, there exists a pressure $\bar{P}\in D^{1,\frac{3}{2}}(\Omega)\cap L^{3}(\Omega)$ such that $(\bar{V},\bar{G},\bar{P})$ solves the system \eqref{L1}. Furthermore, by Lemma \ref{L0}, there exists a constant $c\in\R$ such that $\bar{P}(x)=c$ on $\partial\Omega$. Notice that $U_{0},B_{0}\in L^{4}(\Omega)$, approximating $U_{0},B_{0}$ in the $L^{4}-$norm by test functions implies that
			\begin{equation}\label{LE19}
				\int_{\Omega}\bar{V}\cdot\nabla\bar{V}U_{0}dx-\int_{\Omega}\bar{G}\cdot\nabla\bar{G}U_{0}dx=-\int_{\Omega}\nabla\bar{P}\cdot U_{0}dx
			\end{equation}
			and
			\begin{equation}\label{LE20}
				\int_{\Omega}\bar{V}\cdot\nabla\bar{G}B_{0}dx-\int_{\Omega}\bar{G}\cdot\nabla\bar{V}B_{0}dx=0.
			\end{equation}
			Now, let us go back to \eqref{LE13}. By using \eqref{LE19} and \eqref{LE20}, we have
			\begin{align*}\label{LE21}
				\begin{split}
					1+\frac{\lambda_{0}}{4}\int_{\Omega}|\bar{V}|^{2}dx+\frac{\lambda_{0}}{4}\int_{\Omega}|\bar{G}|^{2}dx&=-\lambda_{0}\int_{\Omega}\nabla\bar{P}\cdot U_{0}dx\\&
					=-\lambda_{0}\int_{\Omega}\nabla\cdot(\bar{P}U_{0})dx\\&=-c\lambda_{0}\int_{\partial\Omega}U_{0}\cdot\vec{\nu} d\mathcal{H}^{2}\\&
					=-c\lambda_{0}\int_{\partial\Omega}\nabla\cdot U_{0} dx=0,
				\end{split}
			\end{align*}
			where $\vec{\nu}$ is the unit outward normal vector of $\partial\Omega$.
			The obtained contradiction completed the proof of Lemma \ref{L2}.
		\end{proof}
		With this {\em a prior} estimate in hand, we can initiate the proof of the existence of solutions to \eqref{PLS} in bounded domain by {\em Leray-Schauder theorem}. To begin, we introduce the linear map $L$ as follows,
		\begin{equation*}
			\begin{aligned}
				L(V,G):=\Big(\frac{1}{2}V+&\frac{1}{2}x\cdot\nabla V-U_{0}\cdot\nabla V-V\cdot\nabla U_{0}+B_{0}\cdot\nabla G+G\cdot\nabla B_{0},\\& \frac{1}{2}G+\frac{1}{2}x\cdot\nabla G-U_{0}\cdot\nabla G-V\cdot\nabla B_{0}+B_{0}\cdot\nabla V+G\cdot\nabla U_{0}\Big)
			\end{aligned}
		\end{equation*}
		and the nonlinear map $N$,
		\begin{equation*}
			\begin{aligned}
				N(V,G):=\Big(-V\cdot\nabla V&-U_{0}\cdot\nabla U_{0}+G\cdot\nabla G+B_{0}\cdot\nabla B_{0},\\&
				-V\cdot\nabla G-U_{0}\cdot\nabla B_{0}+G\cdot\nabla V+B_{0}\cdot\nabla U_{0}\Big).
			\end{aligned}
		\end{equation*}
		In this way, the system \eqref{ESp2} can be rewritten as
		\begin{equation}\label{re1}
			(-\Delta V+\nabla P,-\Delta G)=\lambda A(V,G),
		\end{equation}
		where
		\begin{equation}\label{nlm}
			A(V,G)=L(V,G)+N(V,G).
		\end{equation}
		Notice that the system \eqref{PLS} can be reduced to the case $\lambda=1$ in the \eqref{re1}.
		Since $\Omega$ is bounded, we can utilize the Poincar\'{e} inequality to define the scalar product in the Hilbert space $\mathbf{H}_{0,\sigma}^{1}(\Omega)$ as follows:
		$$
		\big\langle(V,G),(V^{\prime},G^{\prime})\big\rangle_{\mathbf{H}_{0,\sigma}^{1}(\Omega)}=\int_{\Omega}\nabla V:\nabla V^{\prime}dx+\int_{\Omega}\nabla G:\nabla G^{\prime}dx.
		$$
		Therefore, the weak formulation of equation \eqref{re1} can be further rewritten as
		$$
		\big\langle(V,G),\Upsilon\big\rangle_{\mathbf{H}_{0,\sigma}^{1}(\Omega)}=\int_{\Omega}A(V,G)\Upsilon dx, \ \ \ \forall\ \ \Upsilon\in\mathbf{C}^{\infty}_{c,\sigma}(\Omega).
		$$
		Denote $\mathbf{H}^{-1}(\Omega)$ as the dual space of $\mathbf{H}_{0,\sigma}^{1}(\Omega)$.  According to the Riesz representation theorem, for any $f\in\mathbf{H}^{-1}(\Omega)$ there exists an isomorphism mapping $\mathbb{T}:\mathbf{H}^{-1}(\Omega)\rightarrow\mathbf{H}_{0,\sigma}^{1}(\Omega)$ such that
		$$
		\langle\mathbb{T}(f),\zeta\rangle=\int_{\Omega}f\cdot\zeta dx, \ \ \ \forall\ \ \zeta\in\mathbf{H}_{0,\sigma}^{1}(\Omega),
		$$
		and
		$$
		\|\mathbb{T}(f)\|_{\mathbf{H}_{0,\sigma}^{1}(\Omega)}\leq \|f\|_{\mathbf{H}^{-1}(\Omega)}.
		$$
		Thus, the system \eqref{ESp2} can ultimately be rewritten as
		\begin{equation*}\label{re2}
			(V,G)=\lambda(\mathbb{T}\circ A)(V,G)\triangleq \lambda \mathcal{S}(V,G).
		\end{equation*}
	Now, we will prove the existence of a solution for the system \eqref{ESp2} when $\lambda=1$.
		\begin{lem}[continuity and compactness]\label{L3}
			Let $\Omega$ be a bounded domain with a smooth boundary. Then,
			$$
			A:\mathbf{H}_{0,\sigma}^{1}(\Omega)\rightarrow\mathbf{L}^{\frac{3}{2}}(\Omega)
			$$
			is continuous, and
			$$
			A:\mathbf{H}_{0,\sigma}^{1}(\Omega)\rightarrow\mathbf{H}^{-1}(\Omega)
			$$
			is compact, where $A$ is defined in \eqref{nlm}.
		\end{lem}
		\begin{proof}
			First, we prove $A:\mathbf{H}_{0,\sigma}^{1}(\Omega)\rightarrow\mathbf{L}^{\frac{3}{2}}(\Omega)$ is continuous. The linear component is straightforward; thus, we will focus on the nonlinear map $N$. Here we will examine only a few terms, as the others are easier. By applying the Sobolev embedding theorem, we get $(V,G)\in\mathbf{H}_{0,\sigma}^{1}(\Omega)\subset \mathbf{L}^{6}(\Omega)$.
			Hence, we conclude that  $V\cdot\nabla V\in L^{\frac{3}{2}}(\Omega)$. From the decay estimate \eqref{initial1}, it follows that $U_{0}\cdot\nabla U_{0}\in L^{\frac{3}{2}}(\Omega)$, when $\Omega$ is bounded.
			Since the Sobolev embedding $\mathbf{H}_{0,\sigma}^{1}(\Omega)\subset \mathbf{L}^{6}(\Omega)$ is continuous, we get $A:\mathbf{H}_{0,\sigma}^{1}(\Omega)\rightarrow\mathbf{L}^{\frac{3}{2}}(\Omega)$ is also continuous. Moreover, every function $h\in L^{\frac{3}{2}}(\Omega)$ can be identified as an element of $\mathbf{H}^{-1}(\Omega)$. Thus, $A:\mathbf{H}_{0,\sigma}^{1}(\Omega)\rightarrow\mathbf{H}^{-1}(\Omega)$ is continuous.
			
			Next, we will prove that the operator $A:\mathbf{H}_{0,\sigma}^{1}(\Omega)\rightarrow\mathbf{H}^{-1}(\Omega)$ is compact. Suppose that $\|(V_{k},G_{k})\|_{\mathbf{H}_{0,\sigma}^{1}(\Omega)}\leq C$. After extracting a subsequence, there exists $(\hat{V},\hat{G})\in \mathbf{H}_{0,\sigma}^{1}(\Omega)$ such that
			$$
			(v_{k},g_{k})=(\hat{V}-V_{k},\hat{G}-G_{k})\rightharpoonup0\ \ \mbox{in}\ \ \mathbf{H}_{0,\sigma}^{1}(\Omega)
			$$
			and
			$$
			(v_{k},g_{k})=(\hat{V}-V_{k},\hat{G}-G_{k})\rightarrow0\ \ \mbox{in}\ \ \mathbf{L}^{3}(\Omega).
			$$
			First, we consider the linear term $L$, using the decay estimate \eqref{initial1} of $U_{0},B_{0}$, for any $\Psi=(\varphi,\psi)\in\mathbf{H}_{0,\sigma}^{1}(\Omega)$, we have
			\begin{align*}
				\Big|\int_{\Omega}L(v_{k},g_{k})\cdot(\varphi,\psi)dx\Big|&\leq C\big(\|v_{k}\|_{L^{3}(\Omega)}+\|g_{k}\|_{L^{3}(\Omega)}\big)\big(\|\varphi\|_{H_{0,\sigma}^{1}(\Omega)}+\|\psi\|_{H_{0,\sigma}^{1}(\Omega)}\big)\\&
				\leq C\big(\|v_{k}\|_{L^{3}(\Omega)}+\|g_{k}\|_{L^{3}(\Omega)}\big)\|\Psi\|_{\mathbf{H}_{0,\sigma}^{1}(\Omega)},
			\end{align*}
			where the constant $C$ independent on $k$ and $\Psi$. Thus,
			\begin{align}\label{Linear part}
				\begin{split}
					\|L(\hat{V},\hat{G})-L(V_{k},G_{k})\|_{\mathbf{H}^{-1}(\Omega)}&:=\sup_{\|\Psi\|_{\mathbf{H}_{0,\sigma}^{1}(\Omega)}=1}
					\Big|\int_{\Omega}L(v_{k},g_{k})\cdot \Psi dx\Big|\\&\leq C\big(\|v_{k}\|_{L^{3}(\Omega)}+\|g_{k}\|_{L^{3}(\Omega)}\big)\rightarrow0.
				\end{split}
			\end{align}
			For the nonlinear terms, notice that
			\begin{align*}
				\begin{split}
					N(\hat{V},\hat{G})-N(V_{k},G_{k})=&\Big(-(v_{k}+V_{k})\cdot\nabla v_{k}-v_{k}\cdot\nabla V_{k}+(g_{k}+G_{k})\cdot\nabla g_{k}+g_{k}\cdot\nabla G_{k},\\&-(v_{k}+V_{k})\cdot\nabla g_{k}-v_{k}\cdot\nabla G_{k}+(g_{k}+G_{k})\cdot\nabla v_{k}+g_{k}\cdot\nabla V_{k}\Big).
				\end{split}
			\end{align*}
			Therefore, for any $\Psi\in\mathbf{H}_{0,\sigma}^{1}(\Omega)$, after performing some integration by parts, one can obtain
			\begin{align*}
				\begin{split}
					\Big|\int_{\Omega}\Big(N(\hat{V},\hat{G})-N(V_{k},G_{k})\Big)\cdot \Psi dx\Big|\leq C\big(\|v_{k}\|_{L^{3}(\Omega)}+\|g_{k}\|_{L^{3}(\Omega)}\big)\|\Psi\|_{\mathbf{H}_{0,\sigma}^{1}(\Omega)},
				\end{split}
			\end{align*}
			with a constant $C$ independent on $k$ and $\Psi$. It follows that,
			\begin{align}\label{nonlinear}
				\begin{split}
					\|N(\hat{V},\hat{G})-N(V_{k},G_{k})\|_{\mathbf{H}^{-1}(\Omega)}
					&:=\sup_{\|\Psi\|_{\mathbf{H}_{0,\sigma}^{1}(\Omega)}=1}|\int_{\Omega}\big(N(\hat{V},\hat{G}) -N(V_{k},G_{k})\big)\cdot \Psi dx|\\&\leq C\big(\|v_{k}\|_{L^{3}(\Omega)}+\|g_{k}\|_{L^{3}(\Omega)}\big)\rightarrow0.
				\end{split}
			\end{align}
			Collecting  \eqref{Linear part} and \eqref{nonlinear}, we have
			\begin{align*}
				\|A(V_{k},G_{k})-A(\hat{V},\hat{G})\|_{\mathbf{H}^{-1}(\Omega)}\rightarrow0,
			\end{align*}
			which means that $A:\mathbf{H}_{0,\sigma}^{1}(\Omega)\rightarrow\mathbf{H}^{-1}(\Omega)$ is compact.
		\end{proof}
		\begin{prop}[existence in bounded domains]\label{p2}
			Let $\Omega$ be a bounded domain with a smooth boundary. Assume that $(U_{0},B_{0})$ satisfies the decay estimate \eqref{initial1}. Then the following problem
			\begin{equation}\label{ESp4}
				\begin{cases}
					-\Delta V+\nabla P=\frac{1}{2}V+\frac{1}{2}x\cdot\nabla V-U_{0}\cdot\nabla U_{0}+B_{0}\cdot\nabla B_{0}-F_{1}+F_{2}\ &\mbox{in} \ \ \Omega,\\
					-\Delta G=\frac{1}{2}G+\frac{1}{2}x\cdot\nabla G-U_{0}\cdot\nabla B_{0}+B_{0}\cdot\nabla U_{0}-F_{3}+F_{4}\ &\mbox{in} \ \ \Omega,\\
					{\rm div}\,V={\rm div}\,G=0\ &\mbox{in} \ \ \Omega,\\
					V=G=0 \ &\mbox{on} \ \ \partial\Omega,
				\end{cases}
			\end{equation}
			has a solution $(V,G)\in \mathbf{H}_{0,\sigma}^{1}(\Omega)$.
		\end{prop}
		\begin{proof}
			By Lemma \ref{L5}, for each $\lambda\in[0,1]$, if $(V,G)$ is a weak solution of problem \eqref{ESp2}, i.e.,
			$$
			(V,G)=\lambda(\mathbb{T}\circ A)(V,G)\triangleq\lambda  \mathcal{S}(V,G).
			$$
			Then, we have $\|(V,G)\|_{\mathbf{H}_{0,\sigma}^{1}(\Omega)}\leq C$, where $C$ independent on $\lambda$.
			Next, we will consider the weak formulation of problem \eqref{ESp4}, namely,
			\begin{equation*}\label{LE24}
				(V,G)=(\mathbb{T}\circ A)(V,G)\triangleq  \mathcal{S}(V,G).
			\end{equation*}
			By Lemma \ref{L3}, the nonlinear map $A$ is compact. Hence, the operator $\mathcal{S}$ is compact on $\mathbf{H}_{0,\sigma}^{1}(\Omega)$.
		The {\em Leray-Schauder theorem} (Theorem \ref{T2.1}) implies that the mapping $(V,G)\rightarrow \mathcal{S}(V,G)$ has a fixed point $(V,G)\in\mathbf{H}_{0,\sigma}^{1}(\Omega)$, such that $\|(V,G)\|_{\mathbf{H}_{0,\sigma}^{1}(\Omega)}\leq C$.
		\end{proof}
		Next, we will employ the invading domains technique to
		prove the existence of weak solutions to \eqref{PLS} in the whole space. 
		\begin{lem}[a priori bound for invading method]\label{L6}
			Assume that $(U_{0},B_{0})$ satisfies \eqref{initial1}, and let $(V_{k},G_{k})\in \mathbf{H}_{0,\sigma}^{1}(B_{k})$ be a solution to the problem \eqref{ESp4} with $\Omega=B_{k}$. Then, we have
			\begin{equation}\label{uniform}
			\|(V_{k},G_{k})\|_{\mathbf{H}_{0}^{1}(B_{k})}\leq C(U_{0},B_{0}),
			\end{equation}
		where the constant $C(U_{0},B_{0})$ is independent on $k$.	
		\end{lem}
		\begin{proof}
			This proof has a structure analogous to that of Lemma \ref{L5}, employing a proof by contradiction. By assuming the assertion to be false, we can obtain a sequence of solutions  $(V_{k},G_{k})\in \mathbf{H}_{0,\sigma}^{1}(B_{k})$ such that
			$$
			I_{k}^{2}:=\int_{B_{k}}\Big(\frac{1}{4}|V_{k}|^{2}+\frac{1}{4}|G_{k}|^{2}+|\nabla V_{k}|^{2}+|\nabla G_{k}|^{2}\Big)dx\rightarrow+\infty,
			$$
			and consider the normalized sequence
			\begin{equation*}
				\hat{V}_{k}=\frac{V_{k}}{I_{k}}, \ \ \hat{G}_{k}=\frac{G_{k}}{I_{k}},\ \ \mbox{and}\ \ \hat{P}_{k}=\frac{P_{k}}{I_{k}^{2}}.
			\end{equation*}
			The sequence $(\hat{V}_{k},\hat{G}_{k})$ is bounded in $\mathbf{H}_{\sigma}^{1}(B_{k})$. By the classical extension theorem, there exists $(\bar{V},\bar{G})\in\mathbf{H}_{\sigma}^{1}(\R^{3})$, after extracting a subsequence, which will still be denoted by $(\hat{V}_{k},\hat{G}_{k})$, it follows that
			$$
			(\hat{V}_{k},\hat{G}_{k})\rightharpoonup(\bar{V},\bar{G}),\ \ \mbox{in}\ \ \mathbf{H}_{\sigma}^{1}(\R^{3})
			$$
			and
			$$
			(\hat{V}_{k},\hat{G}_{k})\rightarrow(\bar{V},\bar{G}),\ \ \mbox{locally in}\ \ \mathbf{L}^{p}(\R^{3})\ \ \mbox{for all}\ \ 2\leq p<6.
			$$
			Let us multiply the first equation of system \eqref{ESp4} by $V_{k}$ and the second equation of system \eqref{ESp4} by $G_{k}$, respectively, and integrating on $B_{k}$, one obtains
			\begin{align}\label{LE32}
				\begin{split}
					\frac{1}{4}\int_{B_{k}}|V_{k}|^{2}dx&+\int_{B_{k}}|\nabla V_{k}|^{2}dx\\&=\int_{B_{k}}(B_{0}\cdot\nabla B_{0}-U_{0}\cdot\nabla U_{0})V_{k}dx+\int_{B_{k}}(F_{2k}-F_{1k})V_{k}dx
				\end{split}
			\end{align}
			and
			\begin{align}\label{LE33}
				\begin{split}
					\frac{1}{4}\int_{B_{k}}|G_{k}|^{2}dx&+\int_{B_{k}}|\nabla G_{k}|^{2}dx\\&=\int_{B_{k}}(B_{0}\cdot\nabla U_{0}-U_{0}\cdot\nabla B_{0})G_{k}dx+\int_{B_{k}}(F_{4k}-F_{3k})G_{k}dx.
				\end{split}
			\end{align}
			Adding \eqref{LE32} and \eqref{LE33}, we multiply by $\frac{1}{I_{k}^{2}}$ and subsequently take the limit as $k\rightarrow\infty$. Following the calculations outlined in Lemma \ref{L2}, we have
			\begin{align}\label{LE34}
				\begin{split}
					1&=-\int_{\R^{3}}\bar{V}\cdot\nabla U_{0}\bar{V}dx+\int_{\R^{3}}\bar{G}\cdot\nabla B_{0}\bar{V}-\int_{\R^{3}}\bar{V}\cdot\nabla B_{0}\bar{G}dx+\int_{\R^{3}}\bar{G}\cdot\nabla U_{0}\bar{G}dx
					\\&=\int_{\R^{3}}\bar{V}\cdot\nabla\bar{V}U_{0}dx-\int_{\R^{3}}\bar{G}\cdot\nabla\bar{G}U_{0}dx
					+\int_{\R^{3}}\bar{V}\cdot\nabla\bar{G}B_{0}dx-\int_{\R^{3}}\bar{G}\cdot\nabla\bar{V}B_{0}dx.
				\end{split}
			\end{align}
			Dividing the first and second equations of \eqref{ESp4} by $I_{k}^{2}$, we obtain
			\begin{align}\label{LE35}
				\begin{split}
					\hat{V}_{k}\cdot\nabla\hat{V}_{k}&-\hat{G}_{k}\cdot\nabla\hat{G}_{k}+\nabla\hat{P}_{k}=\frac{1}{I_{k}}\Big(\Delta \hat{V}_{k}+\frac{1}{2}\hat{V}_{k}+\frac{1}{2}x\cdot\nabla\hat{V}_{k}\\&+\frac{1}{I_{k}}B_{0}\cdot\nabla B_{0}- \frac{1}{I_{k}}U_{0}\cdot\nabla U_{0}-U_{0}\nabla\hat{V}_{k}-\hat{V}_{k}\nabla U_{0}+B_{0}\nabla\hat{G}_{k}+\hat{G}_{k}\nabla B_{0}\Big)\\
				\end{split}
			\end{align}
			and
			\begin{align}\label{LE36}
				\begin{split}
					\hat{V}_{k}\cdot\nabla\hat{G}_{k}&-\hat{G}_{k}\cdot\nabla\hat{V}_{k}=\frac{1}{I_{k}}\Big(\Delta \hat{G}_{k}+\frac{1}{2}\hat{G}_{k}+\frac{1}{2}x\cdot\nabla\hat{G}_{k}\\&+\frac{1}{I_{k}}B_{0}\cdot\nabla U_{0}- \frac{1}{I_{k}}U_{0}\cdot\nabla B_{0}-U_{0}\nabla\hat{G}_{k}-\hat{V}_{k}\nabla B_{0}+B_{0}\nabla\hat{V}_{k}+\hat{G}_{k}\nabla U_{0}\Big).
				\end{split}
			\end{align}
			Next, consider the weak formulation of \eqref{LE35}-\eqref{LE36}, i.e., testing an arbitrary $\varphi\in C_{0,\sigma}^{\infty}(\R^{3})$ to \eqref{LE35} and $\psi\in C_{0,\sigma}^{\infty}(\R^{3})$ to \eqref{LE36}, respectively. Taking $k\rightarrow+\infty$, we have
			\begin{equation*}
				\int_{\R^{3}}(\bar{V}\cdot\nabla\bar{V}-\bar{G}\cdot\nabla\bar{G})\varphi dx=0,\ \ \ \mbox{for all}\ \ \varphi\in C_{0,\sigma}^{\infty}(\R^{3})
			\end{equation*}
			and
			\begin{equation*}
				\int_{\R^{3}}(\bar{V}\cdot\nabla\bar{G}-\bar{G}\cdot\nabla\bar{V})\psi dx=0,\ \ \ \mbox{for all}\ \ \psi\in C_{0,\sigma}^{\infty}(\R^{3}).
			\end{equation*}
			Approximating $(U_{0},B_{0})$ by the test function $(\varphi,\psi)\in \mathbf{C}_{0,\sigma}^{\infty}(\R^{3})$ in $\mathbf{L}^{4}-$norm implies that
			\begin{equation}\label{LE37}
				-\int_{\R^{3}}\bar{V}\cdot\nabla U_{0}\bar{V}dx+\int_{\R^{3}}\bar{G}\cdot\nabla U_{0}\bar{G}dx=\int_{\R^{3}}(\bar{V}\cdot\nabla\bar{V}-\bar{G}\cdot\nabla\bar{G})U_{0}dx=0
			\end{equation}
			and
			\begin{equation}\label{LE38}
				\int_{\R^{3}}\bar{G}\cdot\nabla B_{0}\bar{V}dx-\int_{\R^{3}}\bar{V}\cdot\nabla B_{0}\bar{G}dx=\int_{\R^{3}}(\bar{V}\cdot\nabla\bar{G}-\bar{G}\cdot\nabla\bar{V})B_{0}dx=0.
			\end{equation}
			Plugging \eqref{LE37} and \eqref{LE38} into \eqref{LE34} yields a contradiction, which allows us to conclude the proof.
		\end{proof}
		Once the uniform bound \eqref{uniform} has been established, 
		we can construct a weak solution $(V,G,P)$ satisfies Definition \ref{D-0} through an appropriate approximation.
		\begin{thm}[existence in $\R^{3}$]\label{T1.2}
			Assume that $(U_{0},B_{0})$ satisfies \eqref{initial1}. Then, the elliptic system \eqref{PLS} has a weak solution $(V,G,P)\in H_{\sigma}^{1}(\R^{3})\times H_{\sigma}^{1}(\R^{3})\times L^{2}(\R^{3})$ in the sense of Definition \ref{D-0}, and
			$$
			\|(V,G)\|_{\mathbf{H}^{1}(\R^{3})}\leq C(U_{0},B_{0}).
			$$
		\end{thm}
		\begin{proof}
			According to Proposition \ref{p2}, there exists a sequence of solutions $(V_{k},G_{k})\in\mathbf{H}_{0,\sigma}^{1}(B_{k})$ for the \textbf{PLS} \eqref{PLS} with $\Omega=B_{k}$. Furthermore, by Lemma \ref{L6}, this sequence $(V_{k},G_{k})$ satisfies
			$$
			\|(V_{k},G_{k})\|_{\mathbf{H}_{0,\sigma}^{1}(B_{k})}\leq C(U_{0},B_{0}).
			$$
			The uniform bound stated above implies that there exists $(V,G)\in\mathbf{H}_{\sigma}^{1}(\R^{3})$ and a convergent subsequence $(V_{k_{i}},G_{k_{i}})$, such that
			\begin{equation}\label{E3.1}
				(V_{k_{i}},G_{k_{i}})\rightharpoonup(V,G), \ \ \mbox{locally in}\ \ \mathbf{H}_{\sigma}^{1}(\R^{3})
			\end{equation}
			and
			\begin{equation}\label{E3.2}
				(V_{k_{i}},G_{k_{i}})\rightarrow(V,G), \ \ \mbox{locally in}\ \ \mathbf{L}^{p}(\R^{3}), \ \ \mbox{for each}\ \ 2\leq p<6,
			\end{equation}
			as $i\rightarrow+\infty$. We assert that $(V,G)$ is a solution to the \textbf{PLS} \eqref{PLS} in the whole space. First, we examine the weak formulation of the first equation in \eqref{PLS}. For any $\varphi\in C_{0,\sigma}^{\infty}(\R^{3})$, taking $k$ large enough such that $\supp\varphi\subset B_{k}$, one can easily get
			\begin{align}\label{E33}
				\begin{split}
					\int_{\R^{3}}&\nabla V_{k_{i}}:\nabla\varphi dx+\int_{\R^{3}}\frac{1}{2}V_{k_{i}}\varphi dx+\int_{\R^{3}}\frac{1}{2}x\cdot\nabla V_{k_{i}}\varphi dx\\&
					\rightarrow\int_{\R^{3}}\nabla V:\nabla\varphi dx+\int_{\R^{3}}\frac{1}{2}V\varphi dx+\int_{\R^{3}}\frac{1}{2}x\cdot\nabla V\varphi dx,\ \ \mbox{as}\ \  i\rightarrow+\infty.
				\end{split}
			\end{align}
			Next, we will consider the nonlinear term. We claim that
			\begin{equation}\label{E3.3}
				\int_{\R^{3}}V_{k_{i}}\cdot\nabla V_{k_{i}}\varphi dx\rightarrow\int_{\R^{3}}V\cdot\nabla V\varphi dx,  \ \mbox{as}\ \  i\rightarrow+\infty.
			\end{equation}
			In fact, a straightforward calculation follows that
			\begin{align*}
				\begin{split}
					\int_{\R^{3}}V_{k_{i}}&\cdot\nabla V_{k_{i}}\varphi dx-\int_{\R^{3}}V\cdot\nabla V\varphi dx\\&=\int_{\R^{3}}\big(V_{k_{i}}\otimes (V-V_{k_{i}})\big)\nabla\varphi dx+\int_{\R^{3}}\big((V-V_{k_{i}})\otimes V\big)\nabla\varphi dx.
				\end{split}
			\end{align*}
			By applying the H\"{o}der inequality and \eqref{E3.2}, taking $i\rightarrow\infty$, we immediately get
			$$
			\int_{\R^{3}}\big(V_{k_{i}}\otimes (V-V_{k_{i}})\big)\nabla\varphi dx\leq\|V_{k_{i}}\|_{4}\|V-V_{k_{i}}\|_{L^{2}(B_{k})}\|\nabla\varphi\|_{2}\rightarrow0.
			$$
			Similarly,
			$$
			\int_{\R^{3}}\big((V-V_{k_{i}})\otimes V\big)\nabla\varphi dx\rightarrow0,\ \mbox{as}\ \  i\rightarrow\infty.
			$$
			Thus, \eqref{E3.3} has been proven. By employing a similar calculation,
			\begin{equation}\label{E3.4}
				\int_{\R^{3}}G_{k_{i}}\cdot\nabla G_{k_{i}}\varphi dx\rightarrow\int_{\R^{3}}G\cdot\nabla G\varphi dx,\ \mbox{as}\ \  i\rightarrow\infty.
			\end{equation}
			Now, we consider the other term, utilizing the fact that $U_{0},\nabla U_{0}\in L^{\infty}(\R^{3})$, which enables us to obtain
			\begin{align}\label{E34}
				\begin{split}
					&\int_{\R^{3}}U_{0}\cdot\nabla V_{k_{i}}\varphi dx+\int_{\R^{3}}V_{k_{i}}\cdot\nabla U_{0}\varphi dx\\&=-\int_{\R^{3}}U_{0}\otimes V_{k_{i}}\nabla\varphi dx-\int_{\R^{3}}V_{k_{i}}\otimes U_{0}\nabla\varphi dx\\&\rightarrow-\int_{\R^{3}}U_{0}\otimes V\nabla\varphi dx-\int_{\R^{3}}V\otimes U_{0}\nabla\varphi dx\\&=\int_{\R^{3}}U_{0}\cdot\nabla V\varphi dx+\int_{\R^{3}}V\cdot\nabla U_{0}\varphi dx,\ \mbox{as}\ \  i\rightarrow\infty.
				\end{split}
			\end{align}
			Since $B_{0},\nabla B_{0}\in L^{\infty}(\R^{3})$, as previously mentioned
			\begin{align}\label{E35}
				\begin{split}
					&\int_{\R^{3}}B_{0}\cdot\nabla G_{k_{i}}\varphi dx+\int_{\R^{3}}G_{k_{i}}\cdot\nabla B_{0}\varphi dx\\&\rightarrow\int_{\R^{3}}B_{0}\cdot\nabla G\varphi dx+\int_{\R^{3}}G\cdot\nabla B_{0}\varphi dx,\ \mbox{as}\ \  i\rightarrow\infty.
				\end{split}
			\end{align}
			Collecting \eqref{E33}-\eqref{E35}, one can conclude that
			\begin{align}\label{E36}
				\begin{split}
					&\int_{\R^{3}}\nabla V:\nabla\varphi dx\\=&\int_{\R^{3}}\Big(\frac{1}{2}V+\frac{1}{2}x\cdot \nabla V-(V+U_{0})\cdot\nabla(V+U_{0})+(G+B_{0})\cdot\nabla(G+B_{0})\Big)\varphi dx.
				\end{split}
			\end{align}
			Similarly, testing the second equation in \eqref{PLS} to an arbitrary function $\psi\in C_{0,\sigma}^{\infty}(\R^{3})$ and taking 
			$i\rightarrow\infty$
			allows us to derive
			\begin{align}\label{E37}
				\begin{split}
					&\int_{\R^{3}}\nabla G_{k_{i}}:\nabla\psi dx+\int_{\R^{3}}\frac{1}{2}G_{k_{i}}\psi dx+\int_{\R^{3}}\frac{1}{2}x\cdot\nabla G_{k_{i}}\psi dx\\&
					\rightarrow\int_{\R^{3}}\nabla G:\nabla\psi dx+\int_{\R^{3}}\frac{1}{2}G\psi dx+\int_{\R^{3}}\frac{1}{2}x\cdot\nabla G\psi dx
				\end{split}
			\end{align}
			and
			\begin{align}\label{E38}
				\begin{split}
					&\int_{\R^{3}}\Big((V_{k_{i}}+U_{0})\cdot\nabla G_{k_{i}}+V_{k_{i}}\cdot\nabla B_{0}-(G_{k_{i}}+B_{0})\cdot\nabla V_{k_{i}} +G_{k_{i}}\cdot\nabla U_{0}\Big)\psi dx\\&\rightarrow\int_{\R^{3}}\Big((V+U_{0})\cdot\nabla G+V\cdot\nabla B_{0}-(G+B_{0})\cdot\nabla V+G\cdot\nabla U_{0}\Big)\psi dx.
				\end{split}
			\end{align}
			From \eqref{E37} and \eqref{E38}, we obtain
			\begin{align}\label{E39}
				\begin{split}
					&\int_{\R^{3}}\nabla G:\nabla\psi dx\\&=\int_{\R^{3}}\Big(\frac{1}{2}G+\frac{1}{2}x\cdot \nabla G-(V+U_{0})\cdot\nabla(G+B_{0})+(G+B_{0})\cdot\nabla(V+U_{0})\Big)\psi dx.
				\end{split}
			\end{align}
			Therefore, we construct a solution $(V,G)\in \mathbf{H}^{1}_{\sigma}(\R^{3})$ to \textbf{PLS} \eqref{PLS} and 
			$$
			\|(V,G)\|_{\mathbf{H}^{1}(\R^{3})}\leq C(U_{0},B_{0}).
			$$
		Moreover, from identity \eqref{E36}, the existence of pressure $P$ can be obtained by the well-known De Rham theorem \cite{TR} and
			\begin{align}\label{press}
				\begin{split}
					-\Delta P&=\mbox{divdiv}\Big((V+U_{0})\otimes(V+U_{0})-(G+B_{0})\otimes(G+B_{0})\Big).
				\end{split}
			\end{align}
			 Under the assumption \eqref{initial1}, along with the classical Calder\'{o}n-Zygmund theorem, for all $p\in(\frac{3}{2},3]$ ,we have
			\begin{align*}
				\begin{split}
					\|P\|_{p}&\leq C\Big(\|(V,G)\|_{2p}^{2}+\|V\|_{2p}\|U_{0}\|_{2p}+\|(U_{0},B_{0})\|_{2p}^{2}\Big)\\&\leq C\Big(\|(V,G)\|_{\mathbf{H}^{1}(\R^{3})}^{2}+\|(U_{0},B_{0})\|_{2p}^{2}\Big)<\infty.
				\end{split}
			\end{align*}
			Therefore, by employing an appropriate approximation, we construct a weak solution $(V,G,P)\in H^{1}_{\sigma}(\R^{3})\times H^{1}_{\sigma}(\R^{3})\times L^{2}(\R^{3})$ solves \eqref{PLS} in the sense of Definition \ref{D-0}.
		\end{proof}
		We will now demonstrate that
	 the weak solution $(V,G,P)$ constructed in Theorem \ref{T1.2} is, in fact, smooth. 
		\begin{lem}[\cite{LMZ}]\label{L0-1}
			Let $p\in(1,\infty)$, $k\geq0$ and $F=(F_{ij})\in W^{k,p}(\R^{n})$. Then the equation
			\begin{equation}\label{L1-1}
				\begin{cases}
					-\Delta u+u+\nabla\pi={\rm div}\,F\ \ \ \ &x\in\R^{n},\\
					{\rm div}\,u=0\ \ \ \ &x\in\R^{n},
				\end{cases}
			\end{equation}
			has a solution $v\in W^{k+1,p}(\R^{n})$, which is a unique solution in $L^{p}(\R^{n})$.
		\end{lem}
		\begin{thm}\label{T1.3}
			Let $(V,G,P)$ be the weak solution of the system \eqref{PLS} established in Theorem \ref{T1.2}. Then it is smooth.
		\end{thm}
		\begin{proof}
		The proof is fundamental and relies on the standard bootstrapping argument, along with the linear theory of the Stokes and Poisson equations. For the convenience of the reader, we provide a detailed proof.
		
		Note that the system \eqref{PLS} can be rewritten as
		\begin{align*}
			\begin{split}
				\left.
				\begin{aligned}
						-\Delta V+V+\nabla P=&\mbox{div}\,F\\
					-\Delta G+G=&\mbox{div}\,H\\
					\mbox{div}\,V =\mbox{div}\,G =&0
				\end{aligned}\ \right\}\ \ \mbox{in}\ \  \R^{3},
			\end{split}
		\end{align*}
		where 
		$$
		F=\frac{1}{2}x\otimes V-(V+U_{0})\otimes(V+U_{0})+(G+B_{0})\otimes(G+B_{0})
		$$
		and
		$$
		H=\frac{1}{2}x\otimes G-(V+U_{0})\otimes(G+B_{0})+(G+B_{0})\otimes(V+U_{0}).
		$$
		Indeed, for any $x_{0}\in\R^{3}$ and a fixed $R>0$,  it suffices to demonstrate that $(V,G)$ is smooth in $B_{R}(x_{0})$. Without loss of generality, we may assume that $x_{0}$ is the origin.
		
		{\bf Step 1.} First, we consider
	\begin{align}\label{first}
		\begin{split}
		-\Delta V_{1}+V_{1}+\nabla P_{1}&=\mbox{div}\,F_{1},\\ \mbox{div}\,V_{1}&=0,
	\end{split}
   \end{align}
		where $F_{1}=\xi_{1}F$ and 
		$$
		\xi_{1}\in C_{0}^{\infty}(B_{2R}),\ \  \xi_{1}\equiv1\  \mbox{in}\ B_{(1+\frac{1}{2})R}.
		$$
		By the estimate \eqref{initial1} and the fact $(V,G)\in \mathbf{H}^{1}(\R^{3})$, we have
		$$
		F_{1}\in W^{1,\frac{3}{2}}(\R^{3}).
		$$
		Thus, Lemma \ref{L0-1} allows us to get a unique solution $V_{1}\in W^{2,\frac{3}{2}}(\R^{3})$ to the system \eqref{first}. If we take $W_{1}=V-V_{1}$, then it solves
		$$
		-\Delta W_{1}+W_{1}+\nabla P_{1}^{\prime}=\mbox{div}\,F_{1}^{\prime}\ \ \mbox{in}\ \R^{3},
		$$
		where $P_{1}^{\prime}=P-P_{1}$ and 
	    $$
	    F_{1}^{\prime}=0 \ \ \mbox{in} \ B_{(1+\frac{1}{2})R}.
	    $$
		Due to the fact that $P_{1}^{\prime}$ is harmonic in $B_{(1+\frac{1}{2})R}$, it follows that $P_{1}^{\prime}$ is smooth in $B_{(1+\frac{1}{2})R}$. Let $\zeta(x)\in C_{0}^{\infty}(B_{(1+\frac{1}{4})R})$ such that $\zeta(x)\equiv1$ for $x\in B_{(1+\frac{1}{2^{3}})R}$. Therefore, We have $\zeta(x)\nabla P_{1}^{\prime}\in C_{0}^{\infty}(\R^{3})$. Based on classical elliptic theory, if $\tilde{W}_{1}$ is a solution of 
		$$
		-\Delta \tilde{W}_{1}+\tilde{W}_{1}=-\zeta(x)\nabla P_{1}^{\prime},
		$$ 
		we can get $\tilde{W}_{1}\in C^{\infty}(\R^{3})\cap L^{\infty}(\R^{3})$. Setting $\tilde{V}_{1}=W_{1}-\tilde{W}_{1}$ again, we find that $\tilde{V}_{1}\in W^{1,\frac{3}{2}}(B_{(1+\frac{1}{2^{3}})R})$ is a solution of
		$$
		-\Delta\tilde{V}_{1}+\tilde{V}_{1}=0 \ \ \mbox{in} \ B_{(1+\frac{1}{2^{3}})R}.
		$$
		Based on the regularity of Possion equation, we get $\tilde{V}_{1}\in C^{\infty}(B_{(1+\frac{1}{2^{3}})R})$. Consequently, we have $W_{1}\in C^{\infty}(B_{(1+\frac{1}{2^{3}})R})$. Since $V=V_{1}+W_{1}$, we conclude that
		$$
		V\in W^{2,\frac{3}{2}}(B_{(1+\frac{1}{2^{4}})R}).
		$$
		We can also obtain that 
		$$
		G\in W^{2,\frac{3}{2}}(B_{(1+\frac{1}{2^{4}})R})
		$$
		in a similar way.
		
		{\bf Step 2.} Similarly, we consider
		\begin{align}\label{second}
			\begin{split}
				-\Delta V_{2}+V_{2}+\nabla P_{2}&=\mbox{div}\,F_{2},\\ \mbox{div}\,V_{2}&=0,
			\end{split}
		\end{align}
		where $F_{2}=\xi_{2}F$ and 
		$$
		\xi_{2}\in C_{0}^{\infty}B_{(1+\frac{1}{2^{4}})R}),\ \  \xi_{1}\equiv1\  \mbox{in}\ B_{(1+\frac{1}{2^{5}})R}.
		$$
		Since $(V,G)\in \mathbf{W}^{2,\frac{3}{2}}(B_{(1+\frac{1}{2^{4}})R})$, the Sobolev embedding theorem allows us to drive 
		$$
		F_{2}\in W^{1,p}(\R^{3}),\ \mbox{for all}\ \frac{3}{2}\leq p<3.
		$$  
		In particular, by applying Lemma \ref{L0-1}, we obtain a unique solution  $V_{2}\in W^{2,2}(\R^{3})$. Using a similar approach to Step 1 and the Sobolev embedding theorem, we derive
		$$
		V\in W^{2,2}(B_{(1+\frac{1}{2^{7}})R})\subset L^{\infty}(B_{(1+\frac{1}{2^{7}})R}).
		$$
       Similarly,
		$$
		G\in W^{2,2}(B_{(1+\frac{1}{2^{7}})R})\subset L^{\infty}(B_{(1+\frac{1}{2^{7}})R}).
		$$
		Thus,
		$$
		V\cdot\nabla V,G\cdot\nabla G,V\cdot\nabla G,G\cdot\nabla V\in W^{1,2}(B_{(1+\frac{1}{2^{7}})R}),
		$$
		and by using the argument of Step 1 again, we obtain
		$$
		(V,G)\in \mathbf{W}^{2,2}(B_{(1+\frac{1}{2^{7+3}})R}).
		$$
		After repeating this bootstrapping argument, we conclude that
		$$
		(V,G)\in \mathbf{W}^{k,2}(B_{R}), \ \mbox{for all}\ k>0,
		$$
		which implies $(V,G)$ is smooth in $\R^{3}$. Once we establish that $(V,G)$ is smooth, by \eqref{press} and classical elliptic theory, we can conclude that $P$ is smooth.
	\end{proof}
		\subsection{The decay estimate of the weak solution $(V,G)$}
	In this section, we will establish the decay estimate for the weak solution of the \textbf{PLS} \eqref{PLS}. Motivated by the work of Lai, Miao and Zheng \cite{LMZ}, we use the energy method to derive the $\mathbf{H}^{2}-$estimate of $(V,G)$ and $(|x|V,|x|G)$. Next, one can use the linear theory of the Stokes and heat equations to improve the decay rate of $(V,G)$.
	
	By Lemma \ref{inital}, if $(u_{0},b_{0})\in \mathbf{C}^{0,1}_{loc}(\R^{3}\setminus\{0\})$, we have
	\begin{align}\label{initial2}
		\begin{split}
			&|U_{0}(x)|+B_{0}(x)|\leq C(1+|x|)^{-1},\\&
			|\nabla^{k}U_{0}(x)|+|\nabla^{k}B_{0}(x)|\leq C(1+|x|)^{-2},\ \ \mbox{for all} \ k\geq1.
		\end{split}
	\end{align}
	Furthermore, if $(u_{0},b_{0})\in \mathbf{C}^{1,\alpha}_{loc}(\R^{3}\setminus\{0\})$ for any $0<\alpha\leq1$, then, for all $\beta\in(0,\alpha)$
	\begin{align}\label{initial3}
		\begin{split}
			|\nabla^{k}(-\Delta)^{\frac{\beta}{2}} U_{0}(x)|+|\nabla^{k}(-\Delta)^{\frac{\beta}{2}} B_{0}(x)|\leq C(1+|x|)^{-k-\beta-1}, \ \ k=0,1.
		\end{split}
	\end{align}
	In particular, when $\alpha=1$, it follows that
	\begin{align}\label{initial4}
		\begin{split}
			|\nabla^{k} U_{0}(x)|+|\nabla^{k}B_{0}(x)|\leq C(1+|x|)^{-k-1}, \ \ k=0,1,2.
		\end{split}
	\end{align}
	
	Now, we study the decay estimate of $(V,G)$ under the assumption \eqref{initial2}.
	\begin{thm}\label{T-3}
		Assume that $(U_{0},B_{0})$ satisfies \eqref{initial2}.
		Let $(V,G,P)$ be the weak solution of the system \eqref{PLS} as constructed in Theorem \ref{T1.2}. Then, for all $x\in\R^{3}$, there exists a constant $C=C(U_{0},B_{0})$, such that
		\begin{equation}\label{T-3-1}
			|V(x)|+|G(x)|\leq C(1+|x|)^{-1}
		\end{equation}
		and 
		\begin{equation}\label{T-3-2}
			|\nabla V(x)|+|\nabla G(x)|\leq C(1+|x|)^{-1}.
		\end{equation}
	\end{thm}
	The proof of this theorem is quite complex, let us demonstrate it step by step.
	
{\bf Step 1:} $\mathbf{H}^{1}$-estimate of $(|x|V,|x|G)$. 
	\begin{prop}\label{p3-1}
		Assume that $(U_{0},B_{0})$ satisfies \eqref{initial2}.
		Let $(V,G,P)$ be the weak solution of the system \eqref{PLS} construct in Theorem \ref{T1.2}. Then we have
		\begin{equation}\label{p-3-1}
			\big\|(|\cdot|V,|\cdot|G)\big\|_{\mathbf{H}^{1}(\R^{3})}\leq C(U_{0},B_{0}).
		\end{equation}
	\end{prop}
	\begin{proof}	
		For any $\varepsilon>0$,
		taking $\varphi(x)=h^{2}_{\varepsilon}(x)V(x)$ and $\psi(x)=h^{2}_{\varepsilon}(x)G(x)$ in the weak formulation \eqref{W1}-\eqref{W2},
		where 
		$$
		h_{\varepsilon}(x)=\frac{|x|}{(1+\varepsilon|x|^{2})^{\frac{3}{4}}}.
		$$
		It is easy to verify that the test function $(\varphi,\psi)$ is valid.
		Denoting $V_{\varepsilon}(x)=h_{\varepsilon}V(x)$ and $G_{\varepsilon}(x)=h_{\varepsilon}G(x)$, we get
		\begin{align}\label{P3-1}
			\begin{split}
				&\int_{\R^{3}}\nabla V:\nabla(h_{\varepsilon}V_{\varepsilon})-\frac{1}{2}\int_{\R^{3}}V_{\varepsilon}^{2} dx-\frac{1}{2}\int_{\R^{3}}x\cdot\nabla V\cdot(h_{\varepsilon}V_{\varepsilon}) dx\\=&\int_{\R^{3}}P\mathrm{div}(h_{\varepsilon}V_{\varepsilon}) dx-\int_{\R^{3}}(V+U_{0})\cdot\nabla(V+U_{0})\cdot(h_{\varepsilon}V_{\varepsilon}) dx\\&+\int_{\R^{3}}(G+B_{0})\cdot\nabla(G+B_{0})\cdot(h_{\varepsilon}V_{\varepsilon}) dx
			\end{split}
		\end{align}
		and
		\begin{align}\label{P3-2}
			\begin{split}
				&\int_{\R^{3}}\nabla G:\nabla(h_{\varepsilon}G_{\varepsilon}) dx-\frac{1}{2}\int_{\R^{3}}G_{\varepsilon}^{2} dx-\frac{1}{2}\int_{\R^{3}}x\cdot\nabla G\cdot(h_{\varepsilon}G_{\varepsilon}) dx\\=&-\int_{\R^{3}}(V+U_{0})\cdot\nabla(G+B_{0})\cdot(h_{\varepsilon}G_{\varepsilon}) dx\\&+\int_{\R^{3}}(G+B_{0})\cdot\nabla(V+U_{0})\cdot(h_{\varepsilon}G_{\varepsilon}) dx.
			\end{split}
		\end{align}
		Some straightforward calculations give us
		\begin{align}\label{P3-3}
			\begin{split}
				&\int_{\R^{3}}\nabla V:\nabla(h_{\varepsilon}V_{\varepsilon})dx+	\int_{\R^{3}}\nabla G:\nabla(h_{\varepsilon}G_{\varepsilon}) dx\\=&\|\nabla V_{\varepsilon}\|_{2}^{2}+\|\nabla G_{\varepsilon}\|_{2}^{2}+\int_{\R^{3}}\nabla V:(\nabla h_{\varepsilon}\otimes V_{\varepsilon})dx-\int_{\R^{3}}(\nabla h_{\varepsilon}\otimes V):\nabla V_{\varepsilon}dx\\&+\int_{\R^{3}}\nabla G:(\nabla h_{\varepsilon}\otimes G_{\varepsilon})dx-\int_{\R^{3}}(\nabla h_{\varepsilon}\otimes G):\nabla G_{\varepsilon}dx.
			\end{split}
		\end{align}
		Notice that
		$$
		\nabla h_{\varepsilon}=\frac{x}{|x|(1+\varepsilon|x|^{2})^{\frac{3}{4}}}-\frac{3}{2}\frac{\varepsilon|x|x}{(1+\varepsilon|x|^{2})^{\frac{7}{4}}}.
		$$
		Thus, write $\eta_{\varepsilon}=(1+\varepsilon|x|^{2})^{-\frac{1}{2}}$, we can easily derive
		\begin{align}\label{P3-4}
			\begin{split}
				&-\frac{1}{2}\int_{\R^{3}}x\cdot\nabla V\cdot(h_{\varepsilon}V_{\varepsilon}) dx-\frac{1}{2}\int_{\R^{3}}x\cdot\nabla G\cdot(h_{\varepsilon}G_{\varepsilon}) dx\\=&-\frac{1}{2}\int_{\R^{3}}x\cdot\nabla V_{\varepsilon}\cdot V_{\varepsilon} dx+\frac{1}{2}\int_{\R^{3}}x\cdot(\nabla h_{\varepsilon}\otimes V)\cdot V_{\varepsilon} dx\\&-\frac{1}{2}\int_{\R^{3}}x\cdot\nabla G_{\varepsilon}\cdot G_{\varepsilon}dx+\frac{1}{2}\int_{\R^{3}}x\cdot(\nabla h_{\varepsilon}\otimes G)\cdot G_{\varepsilon} dx\\=&\frac{3}{4}\int_{\R^{3}}V_{\varepsilon}^{2}dx+\frac{1}{2}\int_{\R^{3}}V_{\varepsilon}^{2}dx-\frac{3}{4}\int_{\R^{3}}\frac{\varepsilon|x|^{2}}{1+\varepsilon|x|^{2}}V_{\varepsilon}^{2}dx\\&+\frac{3}{4}\int_{\R^{3}}G_{\varepsilon}^{2}dx+\frac{1}{2}\int_{\R^{3}}G_{\varepsilon}^{2}dx-\frac{3}{4}\int_{\R^{3}}\frac{\varepsilon|x|^{2}}{1+\varepsilon|x|^{2}}G_{\varepsilon}^{2}dx\\=&\frac{1}{2}\int_{\R^{3}}V_{\varepsilon}^{2}dx+\frac{3}{4}\int_{\R^{3}}\eta_{\varepsilon}^{2}V_{\varepsilon}^{2}dx+\frac{1}{2}\int_{\R^{3}}G_{\varepsilon}^{2}dx+\frac{3}{4}\int_{\R^{3}}\eta_{\varepsilon}^{2}G_{\varepsilon}^{2}dx.
			\end{split}
		\end{align}
		Since $\mathrm{div} V=\mathrm{div} G=0$, we have
		\begin{align}\label{P3-5}
			\begin{split}
				&-\int_{\R^{3}}V\cdot\nabla V\cdot(h_{\varepsilon}V_{\varepsilon})dx-\int_{\R^{3}}V\cdot\nabla G\cdot(h_{\varepsilon}G_{\varepsilon})dx\\=&\int_{\R^{3}}V\cdot\nabla (h_{\varepsilon}V_{\varepsilon})\cdot Vdx+\int_{\R^{3}}V\cdot\nabla(h_{\varepsilon}G_{\varepsilon}) \cdot Gdx\\=&\int_{\R^{3}}V\cdot\nabla V_{\varepsilon}\cdot V_{\varepsilon}dx+\int_{\R^{3}}V\cdot(\nabla h_{\varepsilon}\otimes V_{\varepsilon})\cdot Vdx\\&+\int_{\R^{3}}V\cdot\nabla G_{\varepsilon}\cdot G_{\varepsilon}dx+\int_{\R^{3}}V\cdot(\nabla h_{\varepsilon}\otimes G_{\varepsilon})\cdot Gdx\\=&\int_{\R^{3}}V\cdot(\nabla h_{\varepsilon}\otimes V_{\varepsilon})\cdot Vdx+\int_{\R^{3}}V\cdot(\nabla h_{\varepsilon}\otimes G_{\varepsilon})\cdot Gdx
			\end{split}
		\end{align}
		and
		\begin{align}\label{P3-6}
			\begin{split}
				&\int_{\R^{3}}G\cdot\nabla G\cdot(h_{\varepsilon}V_{\varepsilon})dx+\int_{\R^{3}}G\cdot\nabla V\cdot(h_{\varepsilon}G_{\varepsilon})dx\\=&-\int_{\R^{3}}G\cdot\nabla (h_{\varepsilon}V_{\varepsilon})\cdot Gdx-\int_{\R^{3}}G\cdot\nabla(h_{\varepsilon}G_{\varepsilon}) \cdot Vdx\\=&-\int_{\R^{3}}G\cdot\nabla V_{\varepsilon}\cdot G_{\varepsilon}dx-\int_{\R^{3}}G\cdot(\nabla h_{\varepsilon}\otimes V_{\varepsilon})\cdot Gdx\\&-\int_{\R^{3}}G\cdot\nabla G_{\varepsilon}\cdot V_{\varepsilon}dx-\int_{\R^{3}}G\cdot(\nabla h_{\varepsilon}\otimes G_{\varepsilon})\cdot Vdx\\=&-\int_{\R^{3}}G\cdot(\nabla h_{\varepsilon}\otimes V_{\varepsilon})\cdot Gdx-\int_{\R^{3}}G\cdot(\nabla h_{\varepsilon}\otimes G_{\varepsilon})\cdot Vdx.
			\end{split}
		\end{align}
		Adding equations \eqref{P3-1} and \eqref{P3-2}, then combined with \eqref{P3-3}-\eqref{P3-6} to get
		\begin{align}\label{P3-7}
			\begin{split}
				\big\|\big(\nabla V_{\varepsilon},\nabla G_{\varepsilon}\big)\big\|_{2}^{2}+\frac{3}{4}\big\|\big(\eta_{\varepsilon}V_{\varepsilon},\eta_{\varepsilon}G_{\varepsilon}\big)\big\|_{2}^{2}=\sum_{i=1}^{9}I_{i},
			\end{split}
		\end{align}
		where
		\begin{align*}
			\begin{split}
				&I_{1}:=\int_{\R^{3}}P\mathrm{div}(h_{\varepsilon}V_{\varepsilon}) dx,\\&
				I_{2}:=-\int_{\R^{3}}\nabla V:(\nabla h_{\varepsilon}\otimes V_{\varepsilon})dx-\int_{\R^{3}}\nabla G:(\nabla h_{\varepsilon}\otimes G_{\varepsilon})dx,\\&
				I_{3}:=\int_{\R^{3}}(\nabla h_{\varepsilon}\otimes V):\nabla V_{\varepsilon}dx+\int_{\R^{3}}(\nabla h_{\varepsilon}\otimes G):\nabla G_{\varepsilon}dx,\\&
				I_{4}:=\int_{\R^{3}}V\cdot(\nabla h_{\varepsilon}\otimes V_{\varepsilon})\cdot Vdx+\int_{\R^{3}}V\cdot(\nabla h_{\varepsilon}\otimes G_{\varepsilon})\cdot Gdx,\\&
				I_{5}:=-\int_{\R^{3}}G\cdot(\nabla h_{\varepsilon}\otimes V_{\varepsilon})\cdot Gdx-\int_{\R^{3}}G\cdot(\nabla h_{\varepsilon}\otimes G_{\varepsilon})\cdot Vdx,\\&
				I_{6}:=-\int_{\R^{3}}V\cdot\nabla U_{0}\cdot(h_{\varepsilon}V_{\varepsilon}) dx-\int_{\R^{3}}V\cdot\nabla B_{0}\cdot(h_{\varepsilon}G_{\varepsilon}) dx,\\&
				I_{7}:=-\int_{\R^{3}}U_{0}\cdot\nabla(V+U_{0})\cdot(h_{\varepsilon}V_{\varepsilon}) dx-\int_{\R^{3}}U_{0}\cdot\nabla(G+B_{0})\cdot(h_{\varepsilon}G_{\varepsilon}) dx,\\&
				I_{8}:=\int_{\R^{3}}G\cdot\nabla B_{0}\cdot(h_{\varepsilon}V_{\varepsilon}) dx+\int_{\R^{3}}G\cdot\nabla U_{0}\cdot(h_{\varepsilon}G_{\varepsilon}) dx,\\&I_{9}:=\int_{\R^{3}}B_{0}\cdot\nabla(G+B_{0})\cdot(h_{\varepsilon}V_{\varepsilon}) dx+\int_{\R^{3}}B_{0}\cdot\nabla(V+U_{0})\cdot(h_{\varepsilon}G_{\varepsilon}) dx.
			\end{split}
		\end{align*}
		Using the fact that $P\in L^{2}(\R^{3})$ and $\mathrm{div} V=0$, we can apply H\"{o}lder inequality and Young's inequality to obtain
		\begin{align*}
			\begin{split}
				I_{1}&=2\int_{\R^{3}}P\nabla h_{\varepsilon}\cdot V_{\varepsilon}dx\\&=2\int_{\R^{3}}P\frac{x}{|x|(1+\varepsilon|x|^{2})^{\frac{3}{4}}} V_{\varepsilon}dx-3\int_{\R^{3}}P\frac{\varepsilon|x|x}{(1+\varepsilon|x|^{2})^{\frac{7}{4}}} V_{\varepsilon}dx\\&\leq C\|P\|_{2}\|\eta_{\varepsilon}V_{\varepsilon}\|_{2}\\&\leq C\|P\|_{2}^{2}+\frac{1}{64}\|\eta_{\varepsilon}V_{\varepsilon}\|_{2}^{2}.
			\end{split}
		\end{align*}
		By the H\"{o}lder inequality and Young's inequality again
		\begin{align*}
			\begin{split}
				I_{2}=&-\int_{\R^{3}}\nabla V:\bigg(\frac{x}{|x|(1+\varepsilon|x|^{2})^{\frac{3}{4}}}\otimes V_{\varepsilon}\bigg)dx+\frac{3}{2}\int_{\R^{3}}\nabla V:\bigg(\frac{\varepsilon|x|x}{(1+\varepsilon|x|^{2})^{\frac{7}{4}}}\otimes V_{\varepsilon}\bigg)dx\\&-\int_{\R^{3}}\nabla G:\bigg(\frac{x}{|x|(1+\varepsilon|x|^{2})^{\frac{3}{4}}}\otimes G_{\varepsilon}\bigg)dx+\frac{3}{2}\int_{\R^{3}}\nabla G:\bigg(\frac{\varepsilon|x|x}{(1+\varepsilon|x|^{2})^{\frac{7}{4}}}\otimes G_{\varepsilon}\bigg)dx\\ \leq&\big\|\nabla V\big\|_{2}\big\|\eta_{\varepsilon}V_{\varepsilon}\big\|_{2}+\frac{3}{2}\big\|\nabla V\big\|_{2}\big\|\eta_{\varepsilon}V_{\varepsilon}\big\|_{2}+\big\|\nabla G\big\|_{2}\big\|\eta_{\varepsilon}G_{\varepsilon}\big\|_{2}+\frac{3}{2}\big\|\nabla G\big\|_{2}\big\|\eta_{\varepsilon}G_{\varepsilon}\big\|_{2}\\ \leq& C\big\|\big(\nabla V,\nabla G\big)\big\|_{2}^{2}+\frac{1}{64}\big\|\big(\eta_{\varepsilon}V_{\varepsilon},\eta_{\varepsilon}G_{\varepsilon}\big)\big\|_{2}^{2}.
			\end{split}
		\end{align*}
		Similarly,
		\begin{align*}
			\begin{split}
				I_{3}=&\int_{\R^{3}}\bigg(\frac{x}{|x|(1+\varepsilon|x|^{2})^{\frac{3}{4}}}\otimes V\bigg):\nabla V_{\varepsilon}dx-\frac{3}{2}\int_{\R^{3}}\bigg(\frac{\varepsilon|x|x}{(1+\varepsilon|x|^{2})^{\frac{7}{4}}}\otimes V\bigg):\nabla V_{\varepsilon}dx\\&+\int_{\R^{3}}\bigg(\frac{x}{|x|(1+\varepsilon|x|^{2})^{\frac{3}{4}}}\otimes G\bigg):\nabla G_{\varepsilon}dx-\frac{3}{2}\int_{\R^{3}}\bigg(\frac{\varepsilon|x|x}{(1+\varepsilon|x|^{2})^{\frac{7}{4}}}\otimes G\bigg):\nabla G_{\varepsilon}dx\\ \leq&\big\| V\big\|_{2}\big\|\nabla V_{\varepsilon}\big\|_{2}+\frac{3}{2}\big\| V\big\|_{2}\big\|\nabla V_{\varepsilon}\big\|_{2}+\big\| G\big\|_{2}\big\|\nabla G_{\varepsilon}\big\|_{2}+\frac{3}{2}\big\| G\big\|_{2}\big\|\nabla G_{\varepsilon}\big\|_{2}\\ \leq& C\big\|\big(V, G\big)\big\|_{2}^{2}+\frac{1}{64}\big\|\big(\nabla V_{\varepsilon},\nabla G_{\varepsilon}\big)\big\|_{2}^{2}.
			\end{split}
		\end{align*}
		By the embedding theorem, one obtains
		\begin{align*}
			\begin{split}
				I_{4}=&\int_{\R^{3}}V\cdot\bigg(\frac{x}{|x|(1+\varepsilon|x|^{2})^{\frac{3}{4}}}\otimes V_{\varepsilon}\bigg)\cdot Vdx-\frac{3}{2}\int_{\R^{3}}V\cdot\bigg(\frac{\varepsilon|x|x}{(1+\varepsilon|x|^{2})^{\frac{7}{4}}}\otimes V_{\varepsilon}\bigg)\cdot Vdx\\&+\int_{\R^{3}}V\cdot\bigg(\frac{x}{|x|(1+\varepsilon|x|^{2})^{\frac{3}{4}}}\otimes G_{\varepsilon}\bigg)\cdot Gdx-\frac{3}{2}\int_{\R^{3}}V\cdot\bigg(\frac{\varepsilon|x|x}{(1+\varepsilon|x|^{2})^{\frac{7}{4}}}\otimes G_{\varepsilon}\bigg)\cdot Gdx\\ \leq&\big\| V\big\|_{\frac{12}{5}}^{2}\big\| V_{\varepsilon}\big\|_{6}+\frac{3}{2}\big\| V\big\|_{\frac{12}{5}}^{2}\big\| V_{\varepsilon}\big\|_{6}+\big\| V\big\|_{\frac{12}{5}}\big\| G\big\|_{\frac{12}{5}}\big\| G_{\varepsilon}\big\|_{6}+\frac{3}{2}\big\| V\big\|_{\frac{12}{5}}\big\| G\big\|_{\frac{12}{5}}\big\| G_{\varepsilon}\big\|_{6}\\ \leq& C\big\| V\big\|_{\frac{12}{5}}^{4}+C\big\|V\big\|_{\frac{12}{5}}^{2}\big\| G\big\|_{\frac{12}{5}}^{2}+\frac{1}{64}\big\|\big(\nabla V_{\varepsilon},\nabla G_{\varepsilon}\big)\big\|_{2}^{2}.
			\end{split}
		\end{align*}
		Thanks to the decay estimate of $\nabla U_{0}$,
		\begin{align*}
			\begin{split}
				I_{6}&=-\int_{\R^{3}}\frac{|x|}{(1+\varepsilon|x|^{2})^{\frac{3}{4}}}V\cdot\nabla U_{0}\cdot V_{\varepsilon}dx-\int_{\R^{3}}\frac{|x|}{(1+\varepsilon|x|^{2})^{\frac{3}{4}}}V\cdot\nabla B_{0}\cdot G_{\varepsilon}dx\\&\leq \big\||\cdot|\nabla U_{0}\big\|_{\infty}\big\|V\big\|_{2}\big\|\eta_{\varepsilon}V_{\varepsilon}\big\|_{2}+\big\||\cdot|\nabla B_{0}\big\|_{\infty}\big\|V\big\|_{2}\big\|\eta_{\varepsilon}G_{\varepsilon}\big\|_{2}\\&\leq C\big\|V\big\|_{2}^{2}+\frac{1}{64}\big\|\big(\eta_{\varepsilon}V_{\varepsilon},\eta_{\varepsilon}G_{\varepsilon}\big)\big\|_{2}^{2}.
			\end{split}
		\end{align*}
		Notice that $\nabla U_{0},\nabla B_{0}\in L^{2}(\R^{3})$, so we get
		\begin{align*}
			\begin{split}
				I_{7}&\leq \big\||\cdot| U_{0}\big\|_{\infty}\Big(\big\|\nabla V\big\|_{2}+\big\|\nabla U_{0}\big\|_{2}\Big)\big\|\eta_{\varepsilon}V_{\varepsilon}\big\|_{2}+\big\||\cdot| U_{0}\big\|_{\infty}\Big(\big\|\nabla G\big\|_{2}+\big\|\nabla B_{0}\big\|_{2}\Big)\big\|\eta_{\varepsilon}G_{\varepsilon}\big\|_{2}\\&\leq C\big\||\cdot| U_{0}\big\|^{2}_{\infty}\Big(\big\|\big(\nabla V,\nabla G\big)\big\|_{2}^{2}+\big\|\big(\nabla U_{0},\nabla B_{0}\big)\big\|_{2}^{2}\Big)+\frac{1}{64}\big\|\big(\eta_{\varepsilon}V_{\varepsilon},\eta_{\varepsilon}G_{\varepsilon}\big)\big\|_{2}^{2}.
			\end{split}
		\end{align*}
		By the same calculation as $I_{4},I_{6}$ and $I_{7}$, we can immediately obtain
		\begin{align*}
			\begin{split}
				&I_{5}\leq C\big\| G\big\|_{\frac{12}{5}}^{4}+C\big\|V\big\|_{\frac{12}{5}}^{2}\big\| G\big\|_{\frac{12}{5}}^{2}+\frac{1}{64}\big\|\big(\nabla V_{\varepsilon},\nabla G_{\varepsilon}\big)\big\|_{2}^{2},\\&
				I_{8}\leq C\big\|G\big\|_{2}^{2}+\frac{1}{64}\big\|\big(\eta_{\varepsilon}V_{\varepsilon},\eta_{\varepsilon}G_{\varepsilon}\big)\big\|_{2}^{2}
			\end{split}
		\end{align*}
		and
		\begin{align*}
			\begin{split}
				I_{9}\leq C\big\||\cdot| B_{0}\big\|^{2}_{\infty}\Big(\big\|\big(\nabla V,\nabla G\big)\big\|_{2}^{2}+\big\|\big(\nabla U_{0},\nabla B_{0}\big)\big\|_{2}^{2}\Big)+\frac{1}{64}\big\|\big(\eta_{\varepsilon}V_{\varepsilon},\eta_{\varepsilon}G_{\varepsilon}\big)\big\|_{2}^{2}.
			\end{split}
		\end{align*}
		Collecting the estimate $I_{1}-I_{9}$ to \eqref{P3-7}, we have
		\begin{align}\label{P3-16}
			\begin{split}
				\big\|\big(\nabla V_{\varepsilon},\nabla G_{\varepsilon}\big)\big\|_{2}^{2}+\frac{3}{4}\big\|\big(\eta_{\varepsilon}V_{\varepsilon},\eta_{\varepsilon}G_{\varepsilon}\big)\big\|_{2}^{2}\leq C_{\ast},
			\end{split}
		\end{align}
		where the constant $C_{\ast}$ relay on $U_{0},B_{0}$ and the $H^{1}(\R^{3})$ bounded of $(V,G)$. Then, by the Dominated convergence theorem, letting $\varepsilon\to0^{+}$ in \eqref{P3-16}, we can get the desired estimate \eqref{p-3-1}.
	\end{proof}
	{\bf Step 2:} $\mathbf{H}^{2}$-estimate of $(|x|V,|x|G)$.
	\begin{lem}\label{L3-1}
		Assume that $(U_{0},B_{0})$ satisfies \eqref{initial2}. Let $(V,G,P)$ be the weak solution of the system \eqref{PLS} constructed in Theorem \ref{T1.2}. Then, we have $(V,G)\in \mathbf{H}^{2}(\R^{3})$ and 
		\begin{equation}\label{L-3-1}
			\big\|(V,G)\big\|_{\mathbf{H}^{2}(\R^{3})}\leq C(U_{0},B_{0}).
		\end{equation}
		Furthermore,
		$$
		|\cdot|P\in D^{1,2}(\R^{3}).
		$$
		
	\end{lem}
	\begin{proof}Note that the system \eqref{PLS} can be  rewritten as
		\begin{align}\label{PLS-1}
			\begin{split}
				\left.
				\begin{aligned}
					-\Delta V-\frac{1}{2}V-\frac{1}{2}x\cdot\nabla V+V\cdot\nabla V-G\cdot\nabla G+\nabla P=f\\
					-\Delta G-\frac{1}{2}G-\frac{1}{2}x\cdot\nabla G+V\cdot\nabla G-G\cdot\nabla V=g\\
					\mbox{div}\,V =\mbox{div}\,G=0\\
				\end{aligned}\ \right\}
			\end{split}
		\end{align}
		where
		\begin{align*}
			\begin{split}
				f=-V\cdot\nabla U_{0}-U_{0}\cdot\nabla V-U_{0}\cdot\nabla U_{0}+G\cdot\nabla B_{0}+B_{0}\cdot\nabla G+B_{0}\cdot\nabla B_{0},\\
				g=-V\cdot\nabla B_{0}-U_{0}\cdot\nabla G-U_{0}\cdot\nabla B_{0}+G\cdot\nabla U_{0}+B_{0}\cdot\nabla V+B_{0}\cdot\nabla U_{0}.
			\end{split}
		\end{align*}
		By \eqref{initial2}, we can easily obtain $f,g\in L^{2}(\R^{3})$.
		Thus, we can improve the regularity of $(V,G)$ by means of the difference quotient. That is, in the weak formulation of \eqref{PLS-1}, we consider the test function as
		$$
		\varphi(x)=-D_{k}^{-h}D_{k}^{h}V(x),\ \  \phi(x)=-D_{k}^{-h}D_{k}^{h}G(x).
		$$
		Thanks to $\mathrm{div} V=0$, we obtain
		\begin{align}\label{L3-2}
			\begin{split}
				&-\int_{\R^{3}}\nabla V:\nabla D_{k}^{-h}D_{k}^{h}V dx+\frac{1}{2}\int_{\R^{3}}V\cdot D_{k}^{-h}D_{k}^{h}V dx+\frac{1}{2}\int_{\R^{3}}x\cdot\nabla V\cdot D_{k}^{-h}D_{k}^{h}V dx\\&-\int_{\R^{3}}\nabla G:\nabla D_{k}^{-h}D_{k}^{h}G dx+\frac{1}{2}\int_{\R^{3}}G\cdot D_{k}^{-h}D_{k}^{h}G dx+\frac{1}{2}\int_{\R^{3}}x\cdot\nabla G\cdot D_{k}^{-h}D_{k}^{h}G dx\\=&\int_{\R^{3}}V\cdot\nabla V\cdot D_{k}^{-h}D_{k}^{h}Vdx-\int_{\R^{3}}G\cdot\nabla G\cdot D_{k}^{-h}D_{k}^{h}Vdx-\int_{\R^{3}}f\cdot D_{k}^{-h}D_{k}^{h}Vdx\\&+\int_{\R^{3}}V\cdot\nabla G\cdot D_{k}^{-h}D_{k}^{h}Gdx-\int_{\R^{3}}G\cdot\nabla V\cdot D_{k}^{-h}D_{k}^{h}Gdx-\int_{\R^{3}}g\cdot D_{k}^{-h}D_{k}^{h}Gdx.
			\end{split}
		\end{align}
		On one hand, 
		\begin{align*}
			\begin{split}
				&-\int_{\R^{3}}\nabla V:\nabla D_{k}^{-h}D_{k}^{h}V dx-\int_{\R^{3}}\nabla G:\nabla D_{k}^{-h}D_{k}^{h}G dx\\=&\int_{\R^{3}}D_{k}^{h}\nabla V:D_{k}^{h}\nabla Vdx+\int_{\R^{3}}D_{k}^{h}\nabla G:D_{k}^{h}\nabla Gdx=\big\|\big(D_{k}^{h}\nabla V,D_{k}^{h}\nabla G\big)\big\|_{2}^{2}.
			\end{split}
		\end{align*}
		Some readily calculations give us
		\begin{align*}
			\begin{split}
				&\frac{1}{2}\int_{\R^{3}}V\cdot D_{k}^{-h}D_{k}^{h}V dx+\frac{1}{2}\int_{\R^{3}}x\cdot\nabla V\cdot D_{k}^{-h}D_{k}^{h}V dx\\=&-\frac{1}{2}\int_{\R^{3}}D_{k}^{h}V\cdot D_{k}^{h}V dx-\frac{1}{2}\int_{\R^{3}}D_{k}^{h}(x\cdot\nabla V)\cdot D_{k}^{h}V dx\\=&-\frac{1}{2}\int_{\R^{3}}D_{k}^{h}V\cdot D_{k}^{h}V dx-\frac{1}{2}\int_{\R^{3}}(x+he_{k})\cdot\nabla D_{k}^{h}V\cdot D_{k}^{h}V dx-\frac{1}{2}\int_{\R^{3}}\partial_{k}V\cdot D_{k}^{h}V dx\\=&\frac{1}{4}\|D_{k}^{h}V\|_{2}^{2}-\frac{1}{2}\int_{\R^{3}}\partial_{k}V\cdot D_{k}^{h}V dx.
			\end{split}
		\end{align*}
		Similarly,
		\begin{align*}
			\begin{split}
				&\frac{1}{2}\int_{\R^{3}}G\cdot D_{k}^{-h}D_{k}^{h}G dx+\frac{1}{2}\int_{\R^{3}}x\cdot\nabla G\cdot D_{k}^{-h}D_{k}^{h}G dx\\=&\frac{1}{4}\|D_{k}^{h}G\|_{2}^{2}-\frac{1}{2}\int_{\R^{3}}\partial_{k}G\cdot D_{k}^{h}Gdx.
			\end{split}
		\end{align*}
		On the other hand, according to $\mathrm{div} V=\mathrm{div} G=0$, one can obtain
		\begin{align*}
			\begin{split}
				&\int_{\R^{3}}V\cdot\nabla V\cdot D_{k}^{-h}D_{k}^{h}Vdx+\int_{\R^{3}}V\cdot\nabla G\cdot D_{k}^{-h}D_{k}^{h}Gdx\\=&-\int_{\R^{3}}V\cdot\nabla (D_{k}^{-h}D_{k}^{h}V)\cdot Vdx-\int_{\R^{3}}V\cdot\nabla (D_{k}^{-h}D_{k}^{h}G)\cdot Gdx\\=&\int_{\R^{3}}V(x+he_{k})\cdot\nabla D_{k}^{h}V\cdot D_{k}^{h}Vdx+\int_{\R^{3}}D_{k}^{h}V\cdot\nabla D_{k}^{h}V\cdot Vdx\\&+\int_{\R^{3}}V(x+he_{k})\cdot\nabla D_{k}^{h}G\cdot D_{k}^{h}Gdx+\int_{\R^{3}}D_{k}^{h}V\cdot\nabla D_{k}^{h}G\cdot Gdx\\=&\int_{\R^{3}}D_{k}^{h}V\cdot D_{k}^{h}\nabla V\cdot Vdx+\int_{\R^{3}}D_{k}^{h}V\cdot D_{k}^{h}\nabla G\cdot Gdx,
			\end{split}
		\end{align*}
		and
		\begin{align*}
			\begin{split}
				&-\int_{\R^{3}}G\cdot\nabla G\cdot D_{k}^{-h}D_{k}^{h}Vdx-\int_{\R^{3}}G\cdot\nabla V\cdot D_{k}^{-h}D_{k}^{h}Gdx\\=&-\int_{\R^{3}}G(x+he_{k})\cdot\nabla D_{k}^{h}V\cdot D_{k}^{h}Gdx-\int_{\R^{3}}D_{k}^{h}G\cdot\nabla D_{k}^{h}V\cdot Gdx\\&-\int_{\R^{3}}G(x+he_{k})\cdot\nabla D_{k}^{h}G\cdot D_{k}^{h}Vdx-\int_{\R^{3}}D_{k}^{h}G\cdot\nabla D_{k}^{h}G\cdot Vdx\\=&-\int_{\R^{3}}D_{k}^{h}G\cdot D_{k}^{h}\nabla V\cdot Gdx-\int_{\R^{3}}D_{k}^{h}G\cdot D_{k}^{h}\nabla G\cdot Vdx.
			\end{split}
		\end{align*}
		Thus, returning to \eqref{L3-2}, we get
		\begin{align*}
			\begin{split}
				\big\|\big(D_{k}^{h}\nabla V,D_{k}^{h}\nabla G\big)\big\|_{2}^{2}+\frac{1}{4}\big\|\big(D_{k}^{h} V,D_{k}^{h} G\big)\big\|_{2}^{2}=\sum_{i=1}^{5}J_{i},
			\end{split}
		\end{align*}
		where
		\begin{align*}
			\begin{split}
				&J_{1}=\frac{1}{2}\int_{\R^{3}}\partial_{k}V\cdot D_{k}^{h}V dx+\frac{1}{2}\int_{\R^{3}}\partial_{k}G\cdot D_{k}^{h}G dx,\\&J_{2}=\int_{\R^{3}}D_{k}^{h}V\cdot D_{k}^{h}\nabla V\cdot Vdx,\\&J_{3}=\int_{\R^{3}}D_{k}^{h}V\cdot D_{k}^{h}\nabla G\cdot Gdx,\\&J_{4}=-\int_{\R^{3}}D_{k}^{h}G\cdot D_{k}^{h}\nabla V\cdot Gdx-\int_{\R^{3}}D_{k}^{h}G\cdot D_{k}^{h}\nabla G\cdot Vdx,\\&J_{5}=-\int_{\R^{3}}f\cdot D_{k}^{-h}D_{k}^{h}Vdx-\int_{\R^{3}}g\cdot D_{k}^{-h}D_{k}^{h}Gdx.
			\end{split}
		\end{align*}	
		By the H\"{o}lder inequality and Young's inequality,
		$$
		J_{1}\leq C\big\|\big(\nabla V,\nabla G\big)\big\|_{2}^{2}
		$$ 
		and
		\begin{align*}
			\begin{split}
				J_{5}&\leq  \|f\|_{2}\|D_{k}^{-h}D_{k}^{h}V\|_{2}+\|g\|_{2}\|D_{k}^{-h}D_{k}^{h}G\|_{2}\\&\leq C(\|f\|_{2}^{2}+\|g\|_{2}^{2})+\frac{1}{64}\big\|\big(D_{k}^{h}\nabla V,D_{k}^{h}\nabla G\big)\big\|_{2}^{2}.
			\end{split}
		\end{align*}
		According to the interpolation inequality,
		\begin{align*}
			\begin{split}
				J_{2}&\leq  \|D_{k}^{h}V\|_{3}\|D_{k}^{h}\nabla V\|_{2}\|V\|_{6}\\&\leq C\big\|D_{k}^{h}V\big\|_{2}^{\frac{1}{2}}\big\|D_{k}^{h}\nabla V\big\|_{2}^{\frac{3}{2}}\|V\|_{6}
				\\&\leq C\|D_{k}^{h}V\|_{2}^{2}\|\nabla V\|_{2}^{4}+\frac{1}{64}\|D_{k}^{h}\nabla V\|_{2}^{2},
			\end{split}
		\end{align*}
		and
		\begin{align*}
			\begin{split}
				J_{3}&\leq  \|D_{k}^{h}V\|_{3}\|D_{k}^{h}\nabla G\|_{2}\|G\|_{6}\\&\leq C\big\|D_{k}^{h}V\big\|_{2}^{\frac{1}{2}}\big\|D_{k}^{h}\nabla V\big\|_{2}^{\frac{1}{2}}\|D_{k}^{h}\nabla G\|_{2}\|G\|_{6}
				\\&\leq C\big\|D_{k}^{h}V\big\|_{2}\|G\|_{6}^{2}\big\|D_{k}^{h}\nabla V\big\|_{2}+\frac{1}{64}\|D_{k}^{h}\nabla G\|_{2}^{2}\\&\leq C\big\|D_{k}^{h}V\big\|_{2}^{2}\|G\|_{6}^{4}+\frac{1}{64}\big\|\big(D_{k}^{h}\nabla V,D_{k}^{h}\nabla G\big)\big\|_{2}^{2}.
			\end{split}
		\end{align*}
		Similarly,
		$$
		J_{4}\leq C\Big(\|D_{k}^{h}G\|_{2}^{2}\|G\|_{6}^{4}+\|D_{k}^{h}G\|_{2}^{2}\|V\|_{6}^{4}\Big)+\frac{1}{64}\big\|\big(D_{k}^{h}\nabla V,D_{k}^{h}\nabla G\big)\big\|_{2}^{2}.
		$$
		The above estimate enables us to obtain
		\begin{align*}
			\begin{split}
				\big\|\big(D_{k}^{h}\nabla V,D_{k}^{h}\nabla G\big)\big\|_{2}^{2}+\frac{1}{4}\big\|\big(D_{k}^{h} V,D_{k}^{h} G\big)\big\|_{2}^{2}\leq C(U_{0},B_{0}).
			\end{split}
		\end{align*}
		Letting $h\to0$, we conclude that
		$$
		\big\|(V,G)\big\|_{\mathbf{H}^{2}(\R^{3})}\leq C(U_{0},B_{0}).
		$$
		The embedding theorem implies that the following $L^{\infty}$-bounded
		\begin{equation}\label{infty}
		\|(V,G)\|_{\infty}\leq C(U_{0},B_{0}).
		\end{equation}
		Next, we will use the duality argument to prove that 
		$$
		\bar{P}=|\cdot|P\in D^{1,2}(\R^{3}).
		$$
		According to the first equation of \eqref{PLS}, one obtains
		\begin{align*}
			\begin{split}
				-\Delta\bar{P}=&\underbrace{|x|{\rm div}\big((V+U_{0})\cdot\nabla(V+U_{0})-(G+B_{0})\cdot\nabla(G+B_{0})\big)}_{F_{1}}\\&\underbrace{-\frac{2}{|x|}P-2\frac{x}{|x|}\cdot\nabla P.}_{F_{2}}
			\end{split}
		\end{align*}
		From elliptic theory, it is sufficient to prove that $F=F_{1}+F_{2}\in \big(D^{1,2}(\R^{3})\big)^{\prime}$.
		On one hand, for all $\psi(x)\in C_{0}^{\infty}(\R^{3})$, we can apply the Hardy inequality and \eqref{infty} to obtain
		\begin{align*}
			\begin{split}
				\langle F_{1},\psi\rangle&=-\int_{\R^{3}}\Big((V+U_{0})\cdot\nabla(V+U_{0})-(G+B_{0})\cdot\nabla(G+B_{0})\Big)\nabla(|x|\psi)dx\\&\leq\int_{\R^{3}}|x|\big|(V+U_{0})\cdot\nabla(V+U_{0})-(G+B_{0})\cdot\nabla(G+B_{0})\big|\Big(\frac{|\psi|}{|x|}+|\nabla\psi|\Big)dx\\&\leq CC_{finte}\|\nabla\psi\|_{2},
			\end{split}
		\end{align*}
		where
		\begin{align*}
			\begin{split}
				C_{finite}=&\|V\|_{\infty}\||\cdot|\nabla V\|_{2}+\||\cdot|\nabla U_{0}\|_{\infty}\|V\|_{2}+\||\cdot| U_{0}\|_{\infty}\|\nabla V\|_{2}+\||\cdot| U_{0}\|_{\infty}\|\nabla U_{0}\|_{2}\\&+\|G\|_{\infty}\||\cdot|\nabla G\|_{2}+\||\cdot|\nabla B_{0}\|_{\infty}\|G\|_{2}+\||\cdot| B_{0}\|_{\infty}\|\nabla G\|_{2}+\||\cdot| B_{0}\|_{\infty}\|\nabla B_{0}\|_{2}.
			\end{split}
		\end{align*}
		On the other hand,
		\begin{align*}
			\begin{split}
				\langle F_{2},\psi\rangle&=-2\int_{\R^{3}}P\bigg(\frac{\psi}{|x|}-{\rm div}\Big(\frac{x}{|x|}\psi\Big)\bigg)dx\\&=-2\int_{\R^{3}}P\bigg(\frac{\psi}{|x|}-\frac{2}{|x|}\psi-\frac{x}{|x|}\nabla\psi\bigg)dx\\&\leq C\|P\|_{2}\bigg(\Big\|\frac{\psi}{|x|}\Big\|_{2}+\|\nabla\psi\|_{2}\bigg)\\&\leq C\|P\|_{2}\|\nabla\psi\|_{2}.
			\end{split}
		\end{align*}
		Then, the suitable approximate enable us to conclude the proof.
	\end{proof}

	\begin{lem}\label{L3-12}
		Assume that $(U_{0},B_{0})$ satisfies \eqref{initial2}. Let $(V,G,P)$ be the weak solution of the system \eqref{PLS} constructed in Theorem \ref{T1.2}. Then, we have $(|x|V,|x|G)\in \mathbf{H}^{2}(\R^{3})$
		and 
		\begin{equation}\label{L3-2-1}
			\big\|(|\cdot|V,|\cdot|G)\big\|_{\mathbf{H}^{2}(\R^{3})}\leq C(U_{0},B_{0}).
		\end{equation}
		
	\end{lem}
	\begin{proof}
		Notice that
		$$
		\partial_{i}\partial_{j}(|x|V)=\frac{\delta_{ij}|x|-\frac{x_{i}x_{j}}{|x|}}{|x|^{2}}V+\frac{x_{j}}{|x|}\partial_{i}V+\frac{x_{i}}{|x|}\partial_{j}V+|x|\partial_{i}\partial_{j}V.
		$$
		Thus, by Proposition \ref{p3-1} and Lemma \ref{L3-1}, it remains to prove that 
		$$
		(|\cdot|\nabla^{2}V,|\cdot|\nabla^{2}G)\in \mathbf{L}^{2}(\R^{3}).
		$$ In the weak formulation \eqref{W1}-\eqref{W2}, we take the test function as
		$$
		\varphi(x)=-D^{-h}_{k}h_{\varepsilon}^{2}D_{k}^{h}V,\ \ \phi(x)=-D^{-h}_{k}h_{\varepsilon}^{2}D_{k}^{h}G.
		$$
		By performing some straightforward calculations, we have
		\begin{align*}
			\begin{split}
				&\int_{\R^{3}}\nabla V:\nabla(-D^{-h}_{k}h_{\varepsilon}^{2}D_{k}^{h}V)dx+\int_{\R^{3}}\nabla G:\nabla(-D^{-h}_{k}h_{\varepsilon}^{2}D_{k}^{h}G)dx\\=&\int_{\R^{3}}D_{k}^{h}\nabla V:\nabla(h_{\varepsilon}^{2}D_{k}^{h}V)dx+\int_{\R^{3}}D_{k}^{h}\nabla G:\nabla(h_{\varepsilon}^{2}D_{k}^{h}G)dx\\=&\big\|h_{\varepsilon}D_{k}^{h}\nabla V\big\|_{2}^{2}+2\int_{\R^{3}}h_{\varepsilon}D_{k}^{h}\nabla V:\big(\nabla h_{\varepsilon}\otimes D_{k}^{h}V\Big)dx\\&+\big\|h_{\varepsilon}D_{k}^{h}\nabla G\big\|_{2}^{2}+2\int_{\R^{3}}h_{\varepsilon}D_{k}^{h}\nabla G:\Big(\nabla h_{\varepsilon}\otimes D_{k}^{h}G\Big)dx,
			\end{split}
		\end{align*}
		and
		\begin{align*}
			\begin{split}
				&-\frac{1}{2}\int_{\R^{3}}x\cdot\nabla V\cdot\big(-D^{-h}_{k}h_{\varepsilon}^{2}D_{k}^{h}V\big)dx\\=&-\frac{1}{2}\int_{\R^{3}}(x+he_{k})\cdot\nabla D_{k}^{h}V\cdot\big(h_{\varepsilon}^{2}D_{k}^{h}V\big)dx-\frac{1}{2}\int_{\R^{3}}e_{k}\cdot\nabla V\cdot(h_{\varepsilon}^{2}D_{k}^{h}V)dx\\=&-\frac{1}{2}\int_{\R^{3}}h_{\varepsilon}^{2}x\cdot\nabla D_{k}^{h}V\cdot D_{k}^{h}Vdx-\frac{h}{2}\int_{\R^{3}}h_{\varepsilon}^{2}\partial_{k}D_{k}^{h} V\cdot D_{k}^{h}Vdx-\frac{1}{2}\int_{\R^{3}}h_{\varepsilon}^{2}\partial_{k} V\cdot D_{k}^{h}Vdx\\=&\frac{3}{4}\int_{\R^{3}}h_{\varepsilon}^{2} D_{k}^{h}V\cdot D_{k}^{h}Vdx+\frac{1}{2}\int_{\R^{3}}(x\cdot\nabla h_{\varepsilon})\cdot D_{k}^{h}V\cdot (h_{\varepsilon}D_{k}^{h}V)dx\\&+\frac{h}{4}\int_{\R^{3}}(\partial_{k}h_{\varepsilon}^{2})D_{k}^{h} V\cdot D_{k}^{h}Vdx-\frac{1}{2}\int_{\R^{3}}h_{\varepsilon}^{2}\partial_{k} V\cdot D_{k}^{h}Vdx\\=&\frac{5}{4}\int_{\R^{3}}h_{\varepsilon}^{2} D_{k}^{h}V\cdot D_{k}^{h}Vdx-\frac{3}{4}\int_{\R^{3}}\frac{\varepsilon|x|^{3}}{(1+\varepsilon|x|^{2})^{\frac{7}{4}}} D_{k}^{h}V\cdot(h_{\varepsilon} D_{k}^{h}V)dx\\&+\frac{h}{4}\int_{\R^{3}}(\partial_{k}h_{\varepsilon}^{2})D_{k}^{h} V\cdot D_{k}^{h}Vdx-\frac{1}{2}\int_{\R^{3}}h_{\varepsilon}^{2}\partial_{k} V\cdot D_{k}^{h}Vdx\\=&\frac{1}{2}\|h_{\varepsilon}D_{k}^{h}V\|_{2}^{2}+\frac{3}{4}\|\eta_{\varepsilon}h_{\varepsilon}D_{k}^{h}V\|_{2}^{2}
				\\&+\frac{h}{4}\int_{\R^{3}}(\partial_{k}h_{\varepsilon}^{2})D_{k}^{h} V\cdot D_{k}^{h}Vdx-\frac{1}{2}\int_{\R^{3}}h_{\varepsilon}^{2}\partial_{k} V\cdot D_{k}^{h}Vdx.
			\end{split}
		\end{align*}
		Thus,
		\begin{align*}
			\begin{split}
				&-\frac{1}{2}\int_{\R^{3}}x\cdot\nabla V\cdot\big(-D^{-h}_{k}h_{\varepsilon}^{2}D_{k}^{h}V\big)dx-\frac{1}{2}\int_{\R^{3}}V\cdot(-D^{-h}_{k}h_{\varepsilon}^{2}D_{k}^{h}V)dx\\=&-\frac{1}{2}\int_{\R^{3}}x\cdot\nabla V\cdot\big(-D^{-h}_{k}h_{\varepsilon}^{2}D_{k}^{h}V\big)dx-\frac{1}{2}\int_{\R^{3}}h_{\varepsilon}^{2} D_{k}^{h}V\cdot D_{k}^{h}Vdx\\=&\frac{3}{4}\|\eta_{\varepsilon}h_{\varepsilon}D_{k}^{h}V\|_{2}^{2}+\frac{h}{4}\int_{\R^{3}}(\partial_{k}h_{\varepsilon}^{2})D_{k}^{h} V\cdot D_{k}^{h}Vdx-\frac{1}{2}\int_{\R^{3}}h_{\varepsilon}^{2}\partial_{k} V\cdot D_{k}^{h}Vdx.
			\end{split}
		\end{align*}
		Similarly,
		\begin{align*}
			\begin{split}
				&-\frac{1}{2}\int_{\R^{3}}x\cdot\nabla G\cdot\big(-D^{-h}_{k}h_{\varepsilon}^{2}D_{k}^{h}G\big)dx-\frac{1}{2}\int_{\R^{3}}G\cdot(-D^{-h}_{k}h_{\varepsilon}^{2}D_{k}^{h}G)dx\\=&\frac{3}{4}\|\eta_{\varepsilon}h_{\varepsilon}D_{k}^{h}G\|_{2}^{2}+\frac{h}{4}\int_{\R^{3}}(\partial_{k}h_{\varepsilon}^{2})D_{k}^{h} G\cdot D_{k}^{h}Gdx-\frac{1}{2}\int_{\R^{3}}h_{\varepsilon}^{2}\partial_{k} G\cdot D_{k}^{h}Gdx.
			\end{split}
		\end{align*}
		Hence, we have
		\begin{align}\label{L3-2-2}
			\begin{split}
				&\big\|\big(h_{\varepsilon}D_{k}^{h}\nabla V,h_{\varepsilon}D_{k}^{h}\nabla G\big)\big\|_{2}^{2}+\frac{3}{4}\big\|\big(\eta_{\varepsilon}h_{\varepsilon}D_{k}^{h}V,\eta_{\varepsilon}h_{\varepsilon}D_{k}^{h}G\big)\big\|_{2}^{2}=\sum_{i=1}^{8}K_{i}
			\end{split}
		\end{align}
		where
		\begin{align*}
			\begin{split}
				&K_{1}:=\int_{\R^{3}}P\mathrm{div}(-D_{k}^{-h}h_{\varepsilon}^{2}D_{k}^{h}V) dx,\\&
				K_{2}:=-2\int_{\R^{3}}h_{\varepsilon}D_{k}^{h}\nabla V:\big(\nabla h_{\varepsilon}\otimes D_{k}^{h}V\Big)dx-2\int_{\R^{3}}h_{\varepsilon}D_{k}^{h}\nabla G:\Big(\nabla h_{\varepsilon}\otimes D_{k}^{h}G\Big)dx,\\&
				K_{3}:=-\frac{h}{4}\int_{\R^{3}}(\partial_{k}h_{\varepsilon}^{2})D_{k}^{h} V\cdot D_{k}^{h}Vdx-\frac{h}{4}\int_{\R^{3}}(\partial_{k}h_{\varepsilon}^{2})D_{k}^{h} G\cdot D_{k}^{h}Gdx,\\&
				K_{4}:=\frac{1}{2}\int_{\R^{3}}h_{\varepsilon}^{2}\partial_{k} V\cdot D_{k}^{h}Vdx+\frac{1}{2}\int_{\R^{3}}h_{\varepsilon}^{2}\partial_{k} G\cdot D_{k}^{h}Gdx,\\&
				K_{5}:=-\int_{\R^{3}}(V+U_{0})\cdot\nabla(V+U_{0})\cdot(-D_{k}^{-h}h_{\varepsilon}^{2}D_{k}^{h}V) dx,\\&
				K_{6}:=\int_{\R^{3}}(G+B_{0})\cdot\nabla(G+B_{0})\cdot(-D_{k}^{-h}h_{\varepsilon}^{2}D_{k}^{h}V) dx,\\&
				K_{7}:=-\int_{\R^{3}}(V+U_{0})\cdot\nabla(G+B_{0})\cdot(-D_{k}^{-h}h_{\varepsilon}^{2}D_{k}^{h}G) dx,\\&K_{8}:=\int_{\R^{3}}(G+B_{0})\cdot\nabla(V+U_{0})\cdot(-D_{k}^{-h}h_{\varepsilon}^{2}D_{k}^{h}G)dx.
			\end{split}
		\end{align*}
		Thanks to the incompressible conditions, we can get
		\begin{align*}
			\begin{split}
				K_{1}&=2\int_{\R^{3}}D_{k}^{h}P\frac{x}{|x|(1+\varepsilon|x|^{2})^{\frac{3}{4}}}\cdot h_{\varepsilon}D_{k}^{h}Vdx-3\int_{\R^{3}}D_{k}^{h}P\frac{\varepsilon|x|x}{(1+\varepsilon|x|^{2})^{\frac{7}{4}}}\cdot h_{\varepsilon}D_{k}^{h}Vdx\\&\leq C\big\|D_{k}^{h}P\big\|_{2}\big\||\cdot|D_{k}^{h}V\big\|_{2}\leq C\|\nabla P\|_{2}\big\||\cdot|\nabla V\big\|_{2}.
			\end{split}
		\end{align*}
		Since $|\nabla h_{\varepsilon}|\leq \frac{5}{2}$, we can apply
	 the H\"{o}lder inequality and Young's inequality to obtain
		\begin{align*}
			\begin{split}
				K_{2}&\leq C\Big(\big\|h_{\varepsilon}D_{k}^{h}\nabla V\big\|_{2}\big\|D_{k}^{h}V\big\|_{2}+\big\|h_{\varepsilon}D_{k}^{h}\nabla G\big\|_{2}\big\|D_{k}^{h}G\big\|_{2}\Big)\\&\leq C\big\|\big(\nabla V,\nabla G\big)\big\|_{2}^{2}+\frac{1}{64}\big\|\big(h_{\varepsilon}D_{k}^{h}\nabla V,h_{\varepsilon}D_{k}^{h}\nabla G\big)\big\|_{2}^{2}.
			\end{split}
		\end{align*}
		For $K_{3}$, we take $h$ to be sufficiently small to derive
		\begin{align*}
			\begin{split}
				K_{3}=&-\frac{h}{2}\int_{\R^{3}}\frac{x_{k}}{|x|(1+\varepsilon|x|^{2})^{\frac{3}{4}}}h_{\varepsilon}D_{k}^{h} V\cdot D_{k}^{h}Vdx+\frac{3h}{4}\int_{\R^{3}}\frac{\varepsilon|x|x_{k}}{(1+\varepsilon|x|^{2})^{\frac{7}{4}}}h_{\varepsilon}D_{k}^{h} V\cdot D_{k}^{h}Vdx\\&-\frac{h}{2}\int_{\R^{3}}\frac{x_{k}}{|x|(1+\varepsilon|x|^{2})^{\frac{3}{4}}}h_{\varepsilon}D_{k}^{h} G\cdot D_{k}^{h}Gdx+\frac{3h}{4}\int_{\R^{3}}\frac{\varepsilon|x|x_{k}}{(1+\varepsilon|x|^{2})^{\frac{7}{4}}}h_{\varepsilon}D_{k}^{h} G\cdot D_{k}^{h}Gdx\\ \leq& C\Big(\big\|\eta_{\varepsilon}h_{\varepsilon}D_{k}^{h}V\big\|_{2}\big\|D_{k}^{h}V\big\|_{2}+\big\|\eta_{\varepsilon}h_{\varepsilon}D_{k}^{h}G\big\|_{2}\big\|D_{k}^{h}G\big\|_{2}\Big)\\ \leq& C\big\|\big(\nabla V,\nabla G\big)\big\|_{2}^{2}+\frac{1}{64}\big\|\big(\eta_{\varepsilon}h_{\varepsilon}D_{k}^{h}V,\eta_{\varepsilon}h_{\varepsilon}D_{k}^{h}G\big)\big\|_{2}^{2}.
			\end{split}
		\end{align*}
		According to \eqref{p-3-1}, 
		\begin{align*}
			\begin{split}
				K_{4}&\leq C\Big(\big\||\cdot|\partial_{k} V\big\|_{2}\big\||\cdot|D_{k}^{h}V\big\|_{2}+\big\||\cdot|\partial_{k} G\big\|_{2}\big\||\cdot|D_{k}^{h}G\big\|_{2}\Big)\\&\leq C\big\|\big(|\cdot|\nabla V,|\cdot|\nabla G\big)\big\|_{2}^{2}.
			\end{split}
		\end{align*}
		Next, we focus on $K_{5}$. By the estimates of $V$ and $U_{0}$, we have
		\begin{align*}
			\begin{split}
				&-\int_{\R^{3}}V\cdot\nabla V\cdot(-D_{k}^{-h}h_{\varepsilon}^{2}D_{k}^{h}V) dx=\int_{\R^{3}}V\cdot\nabla (-D_{k}^{-h}h_{\varepsilon}^{2}D_{k}^{h}V)\cdot V dx\\=&\int_{\R^{3}}V(x+he_{k})\cdot\nabla (h_{\varepsilon}^{2}D_{k}^{h}V)\cdot D_{k}^{h}V dx+\int_{\R^{3}}D_{k}^{h}V\cdot\nabla (h_{\varepsilon}^{2}D_{k}^{h}V)\cdot V dx\\=&2\int_{\R^{3}}h_{\varepsilon}V(x+he_{k})\cdot\nabla h_{\varepsilon}\otimes D_{k}^{h}V\cdot D_{k}^{h}V dx+\int_{\R^{3}}h_{\varepsilon}^{2}V(x+he_{k})\cdot D_{k}^{h}\nabla V\cdot D_{k}^{h}V dx
				\\&+2\int_{\R^{3}}h_{\varepsilon}D_{k}^{h}V\cdot\nabla h_{\varepsilon}\otimes D_{k}^{h}V\cdot V dx+\int_{\R^{3}}h_{\varepsilon}^{2}D_{k}^{h}V\cdot D_{k}^{h}\nabla V\cdot V dx\\ \leq& C\|V\|_{\infty}\big\||\cdot|D_{k}^{h} V\big\|_{2}\|D_{k}^{h}V\|_{2}+C\|V\|_{\infty}\big\||\cdot|D_{k}^{h} V\big\|_{2}\big\|h_{\varepsilon}D_{k}^{h}\nabla V\big\|_{2}\\ \leq& C\Big(\|V\|_{\infty}\big\||\cdot|\nabla V\big\|_{2}\|\nabla V\|_{2}+\|V\|_{\infty}^{2}\big\||\cdot|\nabla V\big\|_{2}^{2}\Big)+\frac{1}{64}\big\|h_{\varepsilon}D_{k}^{h}\nabla V\big\|_{2}^{2}
			\end{split}
		\end{align*}
		and
		\begin{align*}
			\begin{split}
				&-\int_{\R^{3}}V\cdot\nabla U_{0}\cdot(-D_{k}^{-h}h_{\varepsilon}^{2}D_{k}^{h}V) dx\\=&-\int_{\R^{3}}V(x+he_{k})\cdot D_{k}^{h}\nabla U_{0}\cdot(h_{\varepsilon}^{2}D_{k}^{h}V) dx-\int_{\R^{3}}D_{k}^{h}V\cdot \nabla U_{0}\cdot(h_{\varepsilon}^{2}D_{k}^{h}V) dx\\ \leq& \big\||\cdot|^{2}\nabla^{2}U_{0}\big\|_{\infty}\|V\|_{2}\|\nabla V\|_{2}+\big\||\cdot|^{2}\nabla U_{0}\big\|_{\infty}\|\nabla V\|_{2}^{2},
			\end{split}
		\end{align*}
		where we used higher order derivative estimates of $U_{0}$, see \eqref{initial2}.
		Similarly,
		\begin{align*}
			\begin{split}
				&-\int_{\R^{3}}U_{0}\cdot\nabla V\cdot(-D_{k}^{-h}h_{\varepsilon}^{2}D_{k}^{h}V) dx\\=&-\int_{\R^{3}}U_{0}(x+he_{k})\cdot D_{k}^{h}\nabla V\cdot(h_{\varepsilon}^{2}D_{k}^{h}V) dx-\int_{\R^{3}}D_{k}^{h}U_{0}\cdot \nabla V\cdot(h_{\varepsilon}^{2}D_{k}^{h}V) dx\\ \leq& \big\||\cdot|U_{0}\big\|_{\infty}\big\||\cdot|\nabla V\big\|_{2}\|h_{\varepsilon}D_{k}^{h}\nabla V\|_{2}+\big\||\cdot|^{2}\nabla U_{0}\big\|_{\infty}\|\nabla V\|_{2}^{2}\\ \leq& C\big\||\cdot|U_{0}\big\|_{\infty}^{2}\big\||\cdot|\nabla V\big\|_{2}^{2}+\big\||\cdot|^{2}\nabla U_{0}\big\|_{\infty}\|\nabla V\|_{2}^{2}+\frac{1}{64}\|h_{\varepsilon}D_{k}^{h}\nabla V\|_{2}^{2}
			\end{split}
		\end{align*}
		and
		\begin{align*}
			\begin{split}
				&-\int_{\R^{3}}U_{0}\cdot\nabla U_{0}\cdot(-D_{k}^{-h}h_{\varepsilon}^{2}D_{k}^{h}V) dx\\=&-\int_{\R^{3}}U_{0}(x+he_{k})\cdot D_{k}^{h}\nabla U_{0}\cdot(h_{\varepsilon}^{2}D_{k}^{h}V) dx-\int_{\R^{3}}D_{k}^{h}U_{0}\cdot \nabla U_{0}\cdot(h_{\varepsilon}^{2}D_{k}^{h}V) dx\\ \leq& \big\||\cdot|U_{0}\big\|_{\infty}\|\nabla^{2}U_{0}\|_{2}\big\||\cdot|\nabla V\big\|_{2}+\big\||\cdot|^{2}\nabla U_{0}\big\|_{\infty}\|\nabla U_{0}\|_{2}\|\nabla V\|_{2}.
			\end{split}
		\end{align*}
		Thus, we get
		$$
		K_{5}\leq C(U_{0},B_{0})+\frac{1}{32}\|h_{\varepsilon}D_{k}^{h}\nabla V\|_{2}^{2}.
		$$
		Since $E$ and $B_{0}$ possess the same estimates as $V$ and $U_{0}$, respectively, the following estimates can be obtained in a same way
		\begin{align*}
			\begin{split}
				K_{6}+K_{7}+K_{8}\leq C(U_{0},B_{0})+\frac{1}{16}\big\|\big(h_{\varepsilon}D_{k}^{h}\nabla V,h_{\varepsilon}D_{k}^{h}\nabla G\big)\big\|_{2}^{2}.
			\end{split}
		\end{align*}
		Collecting the estimates $K_{1}-K_{8}$ and returning to \eqref{L3-2-2}, we finally get
		\begin{align}\label{L3-2-3}
			\begin{split}
				\big\|\big(h_{\varepsilon}D_{k}^{h}\nabla V,h_{\varepsilon}D_{k}^{h}\nabla G\big)\big\|_{2}^{2}+\big\|\big(\eta_{\varepsilon}h_{\varepsilon}D_{k}^{h}V,\eta_{\varepsilon}h_{\varepsilon}D_{k}^{h}G\big)\big\|_{2}^{2}\leq C(U_{0},B_{0}).
			\end{split}
		\end{align}
		Letting $\varepsilon$ and $h$ tend to zero in \eqref{L3-2-3}, the proof of this lemma is complete.
	\end{proof}
	
	{\bf Step 3:} the $\mathbf{H}^{2}$-estimate of $\big(|\cdot|\nabla V,|\cdot|\nabla G\big)$.
	
		First, we will consider the following linear system
	\begin{align}\label{LS}
		\begin{split}
			\left.
			\begin{aligned}
				-\Delta W-\frac{1}{2}W-\frac{1}{2}x\cdot\nabla W+\bar{W}\cdot\nabla W+\nabla P=F\\
				\mbox{div}\,W =0\\
			\end{aligned}\ \right\} \ \ \mbox{in}\ \R^{3}.
		\end{split}
	\end{align}
	By applying the same argument of lemma \ref{L3-1}, the following proposition can be established.
	\begin{prop}\label{P-3-4}
		Let $F\in L^{2}(\R^{3})$ and $\bar{W}\in H^{2}(\R^{3})$ with $\rm{div}\,\bar{W} =0$. Assume that $(W,P)$ is a weak solution of \eqref{LS}, i.e. $(W,P)\in H^{1}(\R^{3})\times L^{2}(\R^{3})$ with $|x|\nabla W\in L^{2}(\R^{3})$ and
		\begin{align}\label{P-W}
			\begin{split}
				&\int_{\R^{3}}\nabla W:\nabla\varphi dx-\frac{1}{2}\int_{\R^{3}}W\cdot\varphi dx-\frac{1}{2}\int_{\R^{3}}x\cdot\nabla W\cdot\varphi dx\\&=\int_{\R^{3}}P\mathrm{div}\varphi dx-\int_{\R^{3}} \bar{W}\cdot\nabla W\cdot\varphi dx+\int_{\R^{3}}F\cdot\varphi dx
			\end{split}
		\end{align}
		holds for all vector function $\varphi\in H^{1}(\R^{3})$.
		Then, we have $W\in H^{2}(\R^{3})$ and 
		\begin{equation}\label{P-3-5}
			\|W\|_{H^{2}(\R^{3})}\leq C,
		\end{equation}
		where the constant $C=C(\|W\|_{H^{1}(\R^{3})},\|\bar{W}\|_{H^{2}(\R^{3})},\|F\|_{L^{2}(\R^{3})})>0$.
	\end{prop}
	\begin{proof}
		The proof of this proposition follows the general approach outlined in Lemma \ref{L3-1}. By selecting the test function as $\varphi(x)=-D_{k}^{-h}D_{k}^{h}W(x)$ in \eqref{P-W}, and focus on the term $\bar{W}\cdot\nabla W$, we have
		\begin{align*}
			\begin{split}
				&-\int_{\R^{3}} \bar{W}\cdot\nabla W\cdot(-D_{k}^{-h}D_{k}^{h}W)dx=\int_{\R^{3}} \bar{W}\cdot\nabla (-D_{k}^{-h}D_{k}^{h}W)\cdot Wdx\\
				=&-\int_{\R^{3}}\bar{W}(x+he_{k})\cdot\nabla D_{k}^{h}W\cdot D_{k}^{h}Wdx-\int_{\R^{3}}D_{k}^{h}\bar{W}\cdot\nabla D_{k}^{h}W\cdot Wdx\\ \leq&  C\big\|D_{k}^{h}\bar{W}\big\|_{3}\big\|D_{k}^{h}\nabla W\big\|_{2}\|W\|_{6}\\ \leq& C\|\bar{W}\|_{H^{2}(\R^{3})}^{2}\| W\|_{H^{1}(\R^{3})}^{2}+\frac{1}{64}\big\|D_{k}^{h}\nabla W\big\|_{2}^{2},
			\end{split}
		\end{align*}
		where we use the fact that $\rm{div}\,\bar{W} =0$. Thus, the same argument presented in Lemma \ref{L3-1} allows us to derive \eqref{P-3-5}.
	\end{proof}
	\begin{prop}\label{P-3-6}
		Assume that $(U_{0},B_{0})$ satisfies \eqref{initial2}. Let $(V,G,P)$ be the weak solution of the system \eqref{PLS} constructed in Theorem \ref{T1.2}.
		Then, we have 
		$$(V,G)\in \mathbf{H}^{3}(\R^{3})$$
		and 
		 $$(\bar{V},\bar{G})\triangleq\big(|\cdot|\nabla V,|\cdot|\nabla G\big)\in \mathbf{H}^{2}(\R^{3}).$$
	\end{prop}
	\begin{proof}
		First, since $(V,E)$ solves
		\begin{align*}
			\begin{split}
				-\Delta V&=\frac{1}{2}V+\frac{1}{2}x\cdot\nabla V-\mathbb{P}\big((V+U_{0})\cdot\nabla(V+U_{0})-(G+B_{0})\cdot\nabla(G+B_{0})\big)\\&=F_{1}
			\end{split}
		\end{align*}
		and
		\begin{align*}
			\begin{split}
				-\Delta G&=\frac{1}{2}G+\frac{1}{2}x\cdot\nabla G-(V+U_{0})\cdot\nabla(G+B_{0})+(G+B_{0})\cdot\nabla(V+U_{0})\\&=F_{2}.
			\end{split}
		\end{align*}
		Thanks to Lemma \ref{L3-1}, Lemma \ref{L3-12} and \eqref{initial2}, it is easy to prove that $F_{1},F_{2}\in H^{1}(\R^{3})$. Thus, by classical elliptic theory and the fact $(V,G)\in \mathbf{H}^{2}(\R^{3})$, we can immediately conclude that $(V,G)\in \mathbf{H}^{3}(\R^{3})$, and
		\begin{equation*}
			\big\|(V,G)\big\|_{\mathbf{H}^{3}(\R^{3})}\leq C(U_{0},B_{0}).
		\end{equation*}
		Next, we show that $(\bar{V},\bar{G})\in \mathbf{H}^{2}(\R^{3})$. Let us denote
		$$
		\bar{V}_{k}=|x|\partial_{k}V,\ \ \bar{G}_{k}=|x|\partial_{k}G,\ \ \mbox{and}\ \ \bar{P}_{k}=|x|\partial_{k}P. 
		$$
		Since
		\begin{align*}
			\begin{split}
				-|x|\partial_{k}\Delta V&=|x|\partial_{k}\big(\frac{1}{2}V+\frac{1}{2}x\cdot\nabla V-\nabla P-V\cdot\nabla V+H\big)\\&=\bar{V}_{k}+\frac{1}{2}|x|x\cdot\partial_{k}\nabla V-|x|\partial_{k}\nabla P-\bar{V}_{k}\cdot\nabla V-|x|V\cdot\partial_{k}\nabla V+|x|\partial_{k}H_{1},
			\end{split}
		\end{align*}
		where
		$$
		H_{1}=-V\cdot\nabla U_{0}-U_{0}\cdot\nabla(V+U_{0})+(G+B_{0})\cdot\nabla(G+B_{0}).
		$$
		Due to Lemma \ref{L3-1} and Lemma \ref{L3-12}, we know that $(\bar{V}_{k},\bar{P}_{k})\in H^{1}(\R^{3})\times L^{2}(\R^{3})$ is a weak solution to the following equation
		\begin{align*}
			\begin{split}
				&-\Delta \bar{V}_{k}-\frac{1}{2}\bar{V}_{k}-\frac{1}{2}x\cdot\nabla\bar{V}_{k}+V\cdot\nabla\bar{V}_{k}+\nabla\bar{P}_{k}\\=&-\frac{2}{|x|}\partial_{k}V-2\frac{x}{|x|}\nabla\partial_{k}V+\frac{1}{2}\bar{V}_{k}-\frac{1}{2}x\cdot\big(\frac{x}{|x|}\otimes\partial_{k}V\big)\\&+V\cdot\big(\frac{x}{|x|}\otimes\partial_{k}V\big)+\bar{V}_{k}\cdot\nabla V+\frac{x}{|x|}\partial_{k}P+|x|\partial_{k}H_{1}\\=&\bar{H}_{1}.
			\end{split}
		\end{align*}
		We claim that $\bar{H}_{1}\in L^{2}(\R^{3})$. By the fact that $|x|V\in H^{2}(\R^{3})$, $V\in H^{3}(\R^{3})$, using the Hardy inequality and H\"{o}lder inequality, we have
		\begin{align*}
			\begin{split}
				\Big\|-\frac{2}{|x|}\partial_{k}V-2\frac{x}{|x|}\nabla\partial_{k}V+\frac{1}{2}\bar{V}_{k}-\frac{1}{2}x\cdot\big(\frac{x}{|x|}\otimes\partial_{k}V\big)+V\cdot\big(\frac{x}{|x|}\otimes\partial_{k}V\big)&+\bar{V}_{k}\cdot\nabla V\Big\|_{2}\\&\leq C(U_{0},B_{0}).
			\end{split}
		\end{align*}
		From Lemma \ref{L3-12}, we also have $\nabla(|x|P)\in L^{2}(\R^{3})$ which implies that
		$$
		\Big\|\frac{x}{|x|}\partial_{k}P\Big\|_{2}\leq C(U_{0},B_{0}).
		$$
		For the last term $|x|\partial_{k}H_{1}$, we have
		\begin{align*}
			\begin{split}
				\Big\||x|\partial_{k}H_{1}\Big\|_{2}\leq& \big\||x|\nabla U_{0}\big\|_{\infty}\|\nabla V\|_{2}+\big\||x|\nabla^{2} U_{0}\big\|_{\infty}\|V\|_{2}+\big\||x|\nabla U_{0}\big\|_{\infty}\|\nabla(V+U_{0})\|_{2}\\&+\big\||x|U_{0}\big\|_{\infty}\big\|\nabla^{2}(V+U_{0})\big\|_{2}+\big\||x|\nabla G\big\|_{2}\|\nabla G\|_{\infty}+\big\||x|G\big\|_{\infty}\|\nabla^{2}G\|_{2}\\&+\big\||x|\nabla B_{0}\big\|_{\infty}\|\nabla G\|_{2}+\big\||x|\nabla^{2} B_{0}\big\|_{\infty}\|G\|_{2}\\&+\big\||x|\nabla B_{0}\big\|_{\infty}\|\nabla(G+B_{0})\|_{2}+\big\||x|B_{0}\big\|_{\infty}\big\|\nabla^{2}(G+B_{0})\big\|_{2}\\ \leq& C(U_{0},B_{0}).
			\end{split}
		\end{align*}
		Thus, by Proposition \ref{P-3-4}, we can get $\bar{V}_{k}\in H^{2}(\R^{3})$ and
		$$
		\|\bar{V}\|_{H^{2}(\R^{3})}\leq C(U_{0},B_{0}).
		$$
		Similarly, we have
		$\bar{G}_{k}\in H^{1}(\R^{3})$ solves
		\begin{align*}
			\begin{split}
				&-\Delta \bar{G}_{k}-\frac{1}{2}\bar{G}_{k}-\frac{1}{2}x\cdot\nabla\bar{G}_{k}+V\cdot\nabla\bar{G}_{k}\\=&-\frac{2}{|x|}\partial_{k}G-2\frac{x}{|x|}\nabla\partial_{k}G+\frac{1}{2}\bar{G}_{k}-\frac{1}{2}x\cdot\big(\frac{x}{|x|}\otimes\partial_{k}G\big)\\&+V\cdot\big(\frac{x}{|x|}\otimes\partial_{k}G\big)+|x|\partial_{k}V\cdot\nabla G+|x|\partial_{k}H_{2}\\=&\bar{H}_{2},
			\end{split}
		\end{align*}
		where
		$$
		H_{2}=-V\cdot\nabla B_{0}-U_{0}\cdot\nabla(G+B_{0})+(G+B_{0})\cdot\nabla(V+U_{0}).
		$$
		We can also conclude that $\|\bar{H}_{2}\|_{2}\leq C(U_{0},B_{0})$. Therefore, by applying Proposition \ref{P-3-4}, we also have 
		$$
		\|\bar{G}\|_{H^{2}(\R^{3})}\leq C(U_{0},B_{0}).
		$$	
		Thus, Proposition \ref{P-3-6} has been proven.
	\end{proof}
	
    According to Lemma \ref{L3-12} and Proposition \ref{P-3-6}, together with the Sobolev embedding theorem, Theorem \ref{T-3} can be obtained directly. Next, 
    we consider the Stokes system with a singular force as follows:
    \begin{align}\label{Stokes}
    	\begin{split}
    		\left.
    		\begin{aligned}
    			\partial_{t} w-\Delta w+\nabla p&=t^{-1}{\rm div}_{x}F\big(x/\sqrt{t}\big)\\
    			{\rm div}\,w&=0\\
    		\end{aligned}\ \right\}\ \mbox{in}\ \R^{3}\times (0,+\infty),
    	\end{split}
    \end{align}
    and provide a crucial lemma, which plays an important role in improving decay estimates for weak solutions of equation \eqref{PLS}. 
    \begin{lem}\label{stokes}
    	Let $F\in L^{q}({R^{3}})$ for $1<q<3$. Then
    		\begin{enumerate}
    		\item[(i)]for any $T<\infty$, problem \eqref{Stokes} admits a solution $w\in L^{\infty}(0,T;L^{q}(\R^{3}))$ which has the from
    		\begin{equation}\label{stokes.1}
    			w(x,t)=\int_{0}^{t}e^{(t-s)\Delta}\mathbb{P}\mathrm{div}_{x}s^{-1}F\Big(\frac{\cdot}{\sqrt{s}}\Big)ds.
    		\end{equation}
    		If $\hat{w}(x,t)\in L^{\infty}(0,T;L^{q_{1}}(\R^{3}))$ with any $q_{1}\geq1$ is another solution of \eqref{Stokes} such that $\lim_{t\rightarrow 0^{+}}\|\hat{w}(\cdot,t)\|_{L^{q_{1}}(\R^{3})}=0$, then $w\equiv\hat{w}$ in $\R^{3}\times(0,\infty)$.
    		
    		\item[(ii)] if for all $x\in\R^{3}$, $F(x)$ satisfies $|F(x)|\leq C(1+|x|)^{-2}$, then $w(x,t)$ is given by \eqref{stokes.1}. Let $W(x)=w(x,1)$, then
    		\begin{equation}\label{stokes.2}
    			|W(x)|\leq C(1+|x|)^{-2}.
    		\end{equation}
    		Moreover, if $F\in C^{1}(\R^{3})$ satisfies $
    		{\rm div}_{x}F(\frac{x}{\sqrt{t}})=t^{-\frac{1}{2}}f(\frac{x}{\sqrt{t}})
    		$ with $|f(x)|\leq C(1+|x|)^{-2}$, we have
    		\begin{equation}\label{stokes.3}
    			|\nabla W(x)|\leq C(1+|x|)^{-2}.
    		\end{equation}
    		\item[(iii)] if $F\in C^{1,\alpha}(\R^{3})$ for any $0<\alpha\leq1$, satisfies $$	|\nabla^{k}(-\Delta)^{\frac{\beta}{2}}F(x)|\leq C(1+|x|)^{-k-\beta-2},\ \ k=0,1 \ \mbox{and}\ \beta\in(0,\alpha),$$  then, we have the optimal decay estimate as follows:
    		\begin{align}\label{stokes.}
    			\begin{split}
    				&|W(x)|\leq C(1+|x|)^{-3}.
    			\end{split}
    		\end{align}
    		\end{enumerate}
    \end{lem}

    \begin{rem}	
    	In fact, the result of Lemma \ref{stokes} is  clearly applicable to the heat equation (i.e., without the pressure term in system \eqref{Stokes}), as the heat kernel and the Oseen tensor exhibit a similar structure and pointwise estimates.
    \end{rem}
    \begin{proof}[Proof of Lemma \ref{stokes}]
    	For (i), please refer to Proposition 4.1 of \cite{LMZ1}. 
    	
    	For (ii), based on the pointwise estimate of the Oseen kernel \eqref{oseen} and (i), if $|F(x)|\leq C(1+|x|)^{-2}$, then, we have
    	\begin{align*}
    		\begin{split}
    			|W(x)|&=|w(x,1)|=\Big|\int_{0}^{1}e^{(1-s)\Delta}
    			\mathbb{P}{\rm div}_{x}s^{-1}F\Big(\frac{\cdot}{\sqrt{s}}\Big)ds\Big|\\&=\Big|\int_{0}^{1}\int_{\R^{3}}\nabla S(x-y,1-s)s^{-1}F\Big(\frac{y}{\sqrt{s}}\Big)dyds\Big|
    			\\&\leq C\int_{0}^{1}\int_{\R^{3}}\frac{1}{(\sqrt{1-s}+|x-y|)^{4}}\frac{1}{(\sqrt{s}+|y|)^{2}}dyds\\&= C\int_{0}^{1}\int_{|y|\leq\frac{|x|}{2}}\cdot\cdot\cdot dyds+\int_{0}^{1}\int_{\frac{|x|}{2}\leq|y|\leq2|x|}\cdot\cdot\cdot dyds+\int_{0}^{1}\int_{|y|\geq2|x|}\cdot\cdot\cdot dyds\\&=L_{1}+L_{2}+L_{3}.
    		\end{split}
    	\end{align*}
    	If $|x|\leq M$ for some constant $M\geq0$, we claim that there exists a constant $C_{\ast}$, which depends only on $M$ such that
    	\begin{equation}\label{bound}
    		|W(x)|\leq C_{\ast}.
    	\end{equation} 
    	For $L_{1}$, 
    	\begin{align*}
    		\begin{split}
    			L_{1}&\leq C\int_{0}^{\frac{1}{2}}\int_{|y|\leq\frac{|x|}{2}}\frac{1}{(\sqrt{s}+|y|)^{2}}dyds+C\int_{\frac{1}{2}}^{1}\int_{|y|\leq\frac{|x|}{2}}\frac{1}{(\sqrt{1-s}+|x-y|)^{4}}dyds\\&\leq C\int_{|y|\leq\frac{|x|}{2}}\frac{1}{|y|^{2}}dy+C\int_{0}^{\frac{1}{2}}\int_{|y|\leq\frac{|x|}{2}}\frac{1}{(\sqrt{s}+|x|)^{4}}dyds\leq C_{\ast}.
    		\end{split}
    	\end{align*}
    	Similarly,
    	\begin{align*}
    		\begin{split}
    			L_{2}&\leq C|x|+C\int_{0}^{\frac{1}{2}}\int_{|y|\leq3|x|}\frac{1}{(\sqrt{s}+|y|)^{4}}dyds\leq  C_{\ast}
    		\end{split}
    	\end{align*}
    	and
    	\begin{align*}
    		\begin{split}
    			L_{3}&\leq C\int_{|y|\geq2|x|}\frac{1}{(1+|y|)^{4}}|y|^{-2}dy
    			+C\int_{0}^{\frac{1}{2}}\int_{|y|\geq2|x|}\frac{1}{(\sqrt{s}+|y|)^{4}}dyds\leq C_{\ast}.
    		\end{split}
    	\end{align*}
    	Thus, \eqref{bound} has been proven. We will now assume that $|x|>>1$, and examine the asymptotic properties of $|W(x)|$ at infinity. In this case,
    	\begin{align*}
    		\begin{split}
    			&L_{1}\leq C|x|^{-4}\int_{0}^{1}\int_{|y|\leq\frac{|x|}{2}}\frac{1}{(\sqrt{s}+|y|)^{2}}dyds
    			\leq C|x|^{-3},\\&
    			L_{3}\leq C|x|^{-2}\int_{0}^{1}\int_{|y|\geq2|x|}\frac{1}{(\sqrt{1-s}+|y|)^{4}}dyds\leq C|x|^{-3},
    		\end{split}
    	\end{align*}
    	and 
    	\begin{align*}
    		\begin{split}
    			L_{2}&\leq C|x|^{-2}\int_{0}^{1}\int_{|y|\leq3|x|}\frac{1}{(\sqrt{1-s}+|y|)^{4}}dyds\leq C|x|^{-2}.
    		\end{split}
    	\end{align*}
    	Collecting all estimates of $L_{1}-L_{3}$, we conclude that \eqref{stokes.2}.
    	If $F\in C^{1}(\R^{3})$ satisfies $
    	{\rm div}_{x}F(\frac{x}{\sqrt{t}})=t^{-\frac{1}{2}}f(\frac{x}{\sqrt{t}})
    	$ with $|f(x)|\leq C(1+|x|)^{-2}$, there exists another solution in the following form:
    	\begin{equation}\label{stokes.4}
    		\tilde{w}(x,t)=\int_{0}^{t}e^{(t-s)\Delta}\mathbb{P}s^{-\frac{3}{2}}f\Big(\frac{\cdot}{\sqrt{s}}\Big)ds.
    	\end{equation}
    	Notice that
    	\begin{align*}
    		\begin{split}
    			\|\tilde{w}(\cdot,t)\|_{L^{p}(\R^{3})}&=\bigg\|\int_{0}^{t}e^{(t-s)\Delta}\mathbb{P}s^{-\frac{3}{2}}f\Big(\frac{\cdot}{\sqrt{s}}\Big)ds\bigg\|_{L^{p}(\R^{3})}\\&\leq C\int_{0}^{t}\bigg\|(t-s)^{-\frac{3}{2}}S\Big(\frac{\cdot}{\sqrt{t-s}}\Big)\bigg\|_{L^{q}(\R^{3})}\bigg\|s^{-\frac{3}{2}}f\Big(\frac{\cdot}{\sqrt{s}}\Big)\bigg\|_{L^{r}(\R^{3})}ds
    		\end{split}\\&\leq C\|S(x,1)\|_{L^{q}(\R^{3})}\|f(x)\|_{L^{r}(\R^{3})}\int_{0}^{t}(t-s)^{\frac{3}{2q}-\frac{3}{2}}s^{\frac{3}{2r}-\frac{3}{2}}ds\\&\leq C\|S(x,1)\|_{L^{q}(\R^{3})}\|f(x)\|_{L^{r}(\R^{3})}t^{\frac{3}{2p}-\frac{1}{2}},
    	\end{align*}
    	where $1+\frac{1}{p}=\frac{1}{q}+\frac{1}{r}$ with $q>1$ and $r>3/2$.
    	Thus, we have $\tilde{w}(x,t)\in L^{\infty}(0,T;L^{p}(\R^{3}))$ for all $p\in(\frac{3}{2},3]$. By the uniqueness, we obtain $w(x,t)=\tilde{w}(x,t)$. 
    	According to \eqref{stokes.4}, 
    	\begin{align*}
    		\begin{split}
    			|\nabla W(x)|&=\Big|\int_{0}^{1}\int_{\R^{3}}\nabla S(x-y,1-s)s^{-\frac{3}{2}}f\Big(\frac{y}{\sqrt{s}}\Big)dyds\Big|
    			\\&\leq C\int_{0}^{1}s^{-\frac{1}{2}}\int_{\R^{3}}\frac{1}{(\sqrt{1-s}+|x-y|)^{4}}\frac{1}{(\sqrt{s}+|y|)^{2}}dyds\\&\leq C(1+|x|)^{-2}.
    		\end{split}
    	\end{align*}
    	Thus, (ii) has been proved. 
    	
    	For (iii), if we set $\tilde{F}(x,s)=s^{-1}F(x/ \sqrt{s})$, we can derive
    	\begin{equation*}\label{bound4-1}
    		|\nabla^{k}(-\Delta)^{\frac{\beta}{2}}\tilde{F}(x,s)|\leq C(\sqrt{s}+|x|)^{-k-\beta-2},\ \ \mbox{for}\  k=0,1\ \mbox{and}\ \ \beta\in(0,\alpha).
    	\end{equation*}
    	By applying the basic properties of Fourier multipliers and the uniqueness, we can rewrite $W(x)$ as
    	$$
    	W(x)=\int_{0}^{1}\int_{\R^{3}}(-\Delta)^{-\frac{\beta}{2}}\nabla S(x-y,1-s)(-\Delta)^{\frac{\beta}{2}}\tilde{F}(y,s)dyds.
    	$$
    	We decompose $W(x)$ as follows:
    	\begin{align*}
    		\begin{split}
    			W(x)=&\int_{0}^{1}\int_{|y-x|\leq\frac{|x|}{2}}(-\Delta)^{-\frac{\beta}{2}}\nabla S(x-y,1-s)\Big((-\Delta)^{\frac{\beta}{2}}\tilde{F}(y,s)-(-\Delta)^{\frac{\beta}{2}}\tilde{F}(x,s)\Big)dyds\\&+\int_{0}^{1}\int_{\{|y-x|\geq\frac{|x|}{2}\}\cap\{|y|\geq|x|\}}(-\Delta)^{-\frac{\beta}{2}}\nabla S(x-y,1-s)(-\Delta)^{\frac{\beta}{2}}\tilde{F}(y,s)dyds\\&+\int_{0}^{1}\int_{\{|y-x|\geq\frac{|x|}{2}\}\cap\{|y|\leq|x|\}}(-\Delta)^{-\frac{\beta}{2}}\nabla S(x-y,1-s)(-\Delta)^{\frac{\beta}{2}}\tilde{F}(y,s)dyds\\&-\int_{0}^{1}\int_{|y-x|\geq\frac{|x|}{2}}(-\Delta)^{-\frac{\beta}{2}}\nabla S(x-y,1-s)(-\Delta)^{\frac{\beta}{2}}\tilde{F}(x,s)dyds\\=&\sum_{i=1}^{4}K_{i},
    		\end{split}
    	\end{align*}
    where we use the fact	that
    $\int_{\mathbb{R}^3}(-\Delta)^{-\frac{\beta}{2}} \nabla Sdx=0$. 
    Here, we only consider the asymptotic behavior of $W(x)$ at infinity, as $|W(x)|$ is bounded for $|x|\leq M$. 
    According to mean value theorem and \eqref{frac}, we have
    \begin{align*}
    	\begin{split}
    	|K_{1}|&\leq C\int_{0}^{1}\int_{|y-x|\leq\frac{|x|}{2}}\frac{1}{(\sqrt{1-s}+|x-y|)^{4-\beta}}\frac{|x-y|}{(\sqrt{s}+|\theta y+(1-\theta)x|)^{3+\beta}}dyds\\&\leq C|x|^{-3-\beta}\int_{0}^{1}\int_{|y-x|\leq\frac{|x|}{2}}\frac{1}{(\sqrt{1-s}+|x-y|)^{3-\beta}}dyds\leq C(\beta)|x|^{-3},
    	\end{split}
    \end{align*}
    where we use the fact that
    $$
    |\theta y+(1-\theta)x|\geq|x|-\theta|x-y|\geq\frac{|x|}{2}.
    $$
     For $K_{3}$, we have
    $$
    |K_{3}|\leq C|x|^{-4+\beta}\int_{0}^{1}\int_{|y|\leq|x|}\frac{1}{(\sqrt{s}+|y|)^{2+\beta}}dyds\leq C(\beta)|x|^{-3}.
    $$
    Similarly,
    	$$
    	|K_{2}|+|K_{4}|\leq C|x|^{-2-\beta}\int_{0}^{1}\int_{|y-x|\geq\frac{|x|}{2}}\frac{1}{(\sqrt{1-s}+|x-y|)^{4-\beta}}dyds\leq C(\beta)|x|^{-3}.
    	$$
    	Thus, we conclude the optimal decay estimate for $W(x)$ that 
    	\eqref{stokes.}.
    \end{proof}
    Once Lemma \ref{stokes} has been established, we can 
    improve the decay rate for $(V,G)$ at infinity. 
    	\begin{thm}\label{T1-1}
    	 Let $(V,G,P)$ be the weak solution of \eqref{PLS} constructed in Theorem \ref{T1.2}.
    	 \begin{enumerate}
    	 	\item[(i)] if $(U_{0},B_{0})$ satisfies \eqref{initial2}, then we have
    	 	\begin{align}\label{T11}
    	 		\begin{split}
    	 			&|V(x)|+|G(x)|\leq C(U_{0},B_{0})(1+|x|)^{-3}\log(2+|x|),\\&
    	 			|\nabla V(x)|+|\nabla G(x)|\leq C(U_{0},B_{0})(1+|x|)^{-3}.
    	 		\end{split}
    	 	\end{align}
    	 	\item[(ii)] if $(U_{0},B_{0})$ satisfies \eqref{initial3}, we can derive the optimal estimates for $(V,G)$, that is
    	 	\begin{align}\label{T2}
    	 		\begin{split}
    	 			|V(x)|+|G(x)|\leq C(U_{0},B_{0})(1+|x|)^{-3}.
    	 		\end{split}
    	 	\end{align}
    	 	\item[(iii)] if $(U_{0},B_{0})$ satisfies \eqref{initial4}, then
    	 	\begin{equation}\label{L3-13-1}
    	 		|\nabla^{k} P(x)|\leq C(U_{0},B_{0})(1+|x|)^{-k-1},\ \ k=0,1.
    	 	\end{equation}
    	 \end{enumerate}
    \end{thm}
	\begin{proof}
		First, we focus on the decay estimate \eqref{T11}. Letting
		$$
		v(x,t)=\frac{1}{\sqrt{t}}V\Big(\frac{x}{\sqrt{t}}\Big),\ \ g(x,t)=\frac{1}{\sqrt{t}}G\Big(\frac{x}{\sqrt{t}}\Big),
		$$
		we obtain a solution to the following system
		\begin{align}\label{stokes-2}
			\begin{split}
				\left.
				\begin{aligned}
					\partial_{t} v-\Delta v+\nabla p&=t^{-1}{\rm div}_{x}F\big(x/\sqrt{t}\big)\\
					\partial_{t} g-\Delta g&=t^{-1}{\rm div}_{x}H\big(x/\sqrt{t}\big)\\
					{\rm div}\,v&={\rm div}\,g=0\\
				\end{aligned}\ \right\}\ \mbox{in}\ \R^{3}\times (0,+\infty),
			\end{split}
		\end{align}
		where 
		$$
		F=(V+U_{0})\otimes(V+U_{0})-(G+B_{0})\otimes(G+B_{0})
		$$
		and
		$$
		H=(V+U_{0})\otimes(G+B_{0})-(G+B_{0})\otimes(V+U_{0}).
		$$
		Let us decompose system \eqref{stokes-2} into the Stokes equation
		\begin{align}\label{stokes-3}
			\begin{split}
				\left.
				\begin{aligned}
					\partial_{t} v-\Delta v+\nabla p&=t^{-1}{\rm div}_{x}F\big(x/\sqrt{t}\big)\\
					{\rm div}\,v&=0\\
				\end{aligned}\ \right\}\ \mbox{in}\ \R^{3}\times (0,+\infty)
			\end{split}
		\end{align}
		and the heat equations
		\begin{align}\label{heats-2}
			\begin{split}
				\left.
				\begin{aligned}
					\partial_{t} g-\Delta g&=t^{-1}{\rm div}_{x}H\big(x/\sqrt{t}\big)\\
					{\rm div}\,g&=0\\
				\end{aligned}\ \right\}\ \mbox{in}\ \R^{3}\times (0,+\infty).
			\end{split}
		\end{align}
			Based on Theorem \ref{T-3} and the assumption \eqref{initial2}, one has
		$$
		|F(x)|+|H(x)|\leq C(1+|x|)^{-2}.
		$$
		Thus, due to Lemma \ref{stokes}, the solution $(v,g)$ of the system \eqref{stokes-2} can be expressed as follows:
		\begin{equation*}
			v(x,t)=\int_{0}^{t}e^{(t-s)\Delta}\mathbb{P}{\rm div}_{x}s^{-1}F\Big(\frac{\cdot}{\sqrt{s}}\Big)ds,\ \ g(x,t)=\int_{0}^{t}e^{(t-s)\Delta}{\rm div}_{x}s^{-1}H\Big(\frac{\cdot}{\sqrt{s}}\Big)ds.
		\end{equation*}
		Moreover,
		\begin{equation}\label{T3.15-1}
		|V(x)|+|G(x)|\leq C(1+|x|)^{-2}.
	\end{equation}
		Note that
		$$
		t^{-1}{\rm div}_{x}F(\frac{x}{\sqrt{t}})=t^{-\frac{3}{2}}f(\frac{x}{\sqrt{t}}),\ \ t^{-1}{\rm div}_{x}H(\frac{x}{\sqrt{t}})=t^{-\frac{3}{2}}h(\frac{x}{\sqrt{t}}),
		$$
		where
		$$
		f=(V+U_{0})\cdot\nabla (V+U_{0})-(G+B_{0})\cdot\nabla (G+B_{0})
		$$
		and
		$$
		h=(V+U_{0})\cdot\nabla (G+B_{0})-(G+B_{0})\cdot\nabla (V+U_{0}).
		$$
		By \eqref{initial2} and Theorem \ref{T-3}, we also have
		$$
		|f(x)|+|h(x)|\leq C(1+|x|)^{-2}.
		$$
	Thanks to Lemma \ref{stokes}, the solution $(v,g)$ can be equivalently expressed as
		\begin{equation*}
			v(x,t)=\int_{0}^{t}e^{(t-s)\Delta}\mathbb{P}s^{-\frac{3}{2}}f\Big(\frac{\cdot}{\sqrt{s}}\Big)ds,\ \ 	g(x,t)=\int_{0}^{t}e^{(t-s)\Delta}s^{-\frac{3}{2}}h\Big(\frac{\cdot}{\sqrt{s}}\Big)ds.
		\end{equation*}
		Moreover, 
		\begin{align*}
			\begin{split}
				|\nabla V(x)|+|\nabla G(x)|=&\Big|\int_{0}^{1}\int_{\R^{3}}\nabla S(x-y,1-s)s^{-\frac{3}{2}}f\Big(\frac{y}{\sqrt{s}}\Big)dyds\Big|\\&+\Big|\int_{0}^{1}\int_{\R^{3}}\nabla \Gamma(x-y,1-s)s^{-\frac{3}{2}}h\Big(\frac{y}{\sqrt{s}}\Big)dyds\Big|
				\\ \leq& C\int_{0}^{1}s^{-\frac{1}{2}}\int_{\R^{3}}\frac{1}{(\sqrt{1-s}+|x-y|)^{4}}\frac{1}{(\sqrt{s}+|y|)^{2}}dyds\\ \leq& C(1+|x|)^{-2}.
			\end{split}
		\end{align*}
		Thus, we can improve the order of the decay estimates for $f(x)$ and $h(x)$ to 
		\begin{equation}\label{decay3}
		|f(x)|+|h(x)|\leq C(1+|x|)^{-3}.
		\end{equation}
		Once \eqref{decay3} has been established, we can immediately derive
		\begin{align*}
			\begin{split}
				|V(x)|+|G(x)|=&\Big|\int_{0}^{1}\int_{\R^{3}} S(x-y,1-s)s^{-\frac{3}{2}}f\Big(\frac{y}{\sqrt{s}}\Big)dyds\Big|\\&+\Big|\int_{0}^{1}\int_{\R^{3}} \Gamma(x-y,1-s)s^{-\frac{3}{2}}h\Big(\frac{y}{\sqrt{s}}\Big)dyds\Big|
				\\ \leq& C\int_{0}^{1}\int_{\R^{3}}\frac{1}{(\sqrt{1-s}+|x-y|)^{3}}\frac{1}{(\sqrt{s}+|y|)^{3}}dyds\\ \leq& C(1+|x|)^{-3}\log(2+|x|),
			\end{split}
		\end{align*}
		and
		\begin{align*}
			\begin{split}
				|\nabla V(x)|+|\nabla G(x)|=&\Big|\int_{0}^{1}\int_{\R^{3}} \nabla S(x-y,1-s)s^{-\frac{3}{2}}f\Big(\frac{y}{\sqrt{s}}\Big)dyds\Big|\\&+\Big|\int_{0}^{1}\int_{\R^{3}} \nabla \Gamma(x-y,1-s)s^{-\frac{3}{2}}f\Big(\frac{y}{\sqrt{s}}\Big)dyds\Big|
				\\ \leq& C\int_{0}^{1}\int_{\R^{3}}\frac{1}{(\sqrt{1-s}+|x-y|)^{4}}\frac{1}{(\sqrt{s}+|y|)^{3}}dyds\\ \leq& C(1+|x|)^{-3}.
			\end{split}
		\end{align*}
		Thus, \eqref{T11} has been proved.
		
 Next, we will prove the optimal decay estimates \eqref{T2} under the assumption \eqref{initial3}.
	Note that we can decompose
	\begin{align*}
		\begin{split}
			V(x)=&V_{1}(x)+V_{2}(x)\\=&\int_{0}^{1}\int_{\R^{3}} S(x-y,1-s)s^{-\frac{3}{2}}\tilde{f}\Big(\frac{y}{\sqrt{s}}\Big)dyds+\int_{0}^{1}\int_{\R^{3}} \nabla  S(x-y,1-s)\tilde{F}(y,s)dyds,
		\end{split}
	\end{align*}
	where
	$$
	\tilde{f}=(V+U_{0})\cdot\nabla V+V\cdot\nabla U_{0}-(G+B_{0})\cdot\nabla G-G\cdot\nabla B_{0}
	$$
	and
	$$
	\tilde{F}(y,s)=s^{-1}\big(U_{0}\otimes U_{0}-B_{0}\otimes B_{0}\big)(y/\sqrt{s}).
	$$
	Since $|\tilde{f}(x)|\leq C(1+|x|)^{-4}$, some similar calculations enable us to derive
	\begin{align}\label{stokes.4-2}
		\begin{split}
			&|V_{1}(x)|\leq C(1+|x|)^{-3}.
		\end{split}
	\end{align}
	Now, we focus on $V_{2}$. According to Lemma \ref{stokes}, it is sufficient to demonstrate that
	\begin{equation}\label{only}
			|\nabla^{k}(-\Delta)^{\frac{\beta}{2}}\tilde{F}(x,s)|\leq C(\sqrt{s}+|x|)^{-k-\beta-2},\ \ \mbox{for}\  k=0,1\ \mbox{and}\ \ \beta\in(0,\alpha).
	\end{equation}
	Since $B_{0}$  possesses the same properties as $U_0{}$, for simplicity, we omit $B_{0}\otimes B_{0}$ in $\tilde{F}$. According to the definition of $(-\Delta)^{\alpha}$, we have 
	\begin{align*}
		\begin{split}
		(-\Delta)^{\frac{\beta}{2}}(U_{0}\otimes U_{0})(x)&=2(-\Delta)^{\frac{\beta}{2}}U_{0}\otimes U_{0}+\int_{\R^{3}}\frac{(U_{0}(y)-U_{0}(x))\otimes(U_{0}(x)-U_{0}(y))}{|x-y|^{3+\beta}}dy\\&=\mathrm{I}+\mathrm{II}.
		\end{split}
	\end{align*}
	From \eqref{initial3}, we can immediately derive
	$$
	|\mathrm{I}|\leq C(1+|x|)^{-2-\beta}.
	$$
	For $\mathrm{II}$, we split it into
	\begin{align*}
		\begin{split}
			\mathrm{II}&=\bigg(\int_{|y-x|\leq\frac{|x|}{2}}+\int_{\frac{|x|}{2}\leq|y-x|\leq2|x|}+\int_{|y-x|\geq2|x|}\bigg)\frac{(U_{0}(y)-U_{0}(x))\otimes(U_{0}(x)-U_{0}(y))}{|x-y|^{3+\beta}}dy\\&=\mathrm{II}_{1}+\mathrm{II}_{2}+\mathrm{II}_{3}.
		\end{split}
	\end{align*}
	When $|x|\leq1$, we can easy derive $|II|\leq C$. Thus, we suppose $|x|>>1$ and study its decay rate at infinity. By applying the mean value theorem, we have
	\begin{align*}
		\begin{split}
	\mathrm{II}_{1}&\leq\int_{|y-x|\leq\frac{|x|}{2}}\big|\nabla U_{0}(\theta y+(1-\theta)x)\big|^{2}|x-y|^{-1-\beta}dy\\&\leq C\frac{1}{|x|^{4}}\int_{|y-x|\leq\frac{|x|}{2}}|x-y|^{-1-\beta}dy\leq C|x|^{-2-\beta}.		
	\end{split}
\end{align*}
	For $\mathrm{II}_{2}$, one has
		\begin{align*}
		\begin{split}
			\mathrm{II}_{2}&\leq C|x|^{-3-\beta}\int_{\frac{|x|}{2}\leq|y-x|\leq2|x|}|U_{0}(x)|^{2}+|U_{0}(x)||U_{0}(y)|+|U_{0}(y)|^{2}dy\\&\leq C|x|^{-2-\beta}+C|x|^{-4-\beta}\int_{|y|\leq3|x|}|y|^{-1}dy+C|x|^{-3-\beta}\int_{|y|\leq3|x|}|y|^{-2}dy\\&\leq C|x|^{-2-\beta}.
		\end{split}
	\end{align*}
	We can derive $|y|\geq|x|$ when $|y-x|\geq2|x|$. Thus
	\begin{align*}
		\begin{split}
			\mathrm{II}_{3}&\leq C|x|^{-2}\int_{|y-x|\geq2|x|}\frac{1}{|x-y|^{3+\beta}}dy\\&\leq C|x|^{-2-\beta}.
		\end{split}
	\end{align*}
	Combining the above calculation, we obtain
	$$
	|(-\Delta)^{\frac{\beta}{2}}(U_{0}\otimes U_{0})(x)|\leq C(1+|x|)^{-2-\beta},
	$$
	which means that
	\begin{equation*}
		|(-\Delta)^{\frac{\beta}{2}}\tilde{F}(x,s)|\leq C(\sqrt{s}+|x|)^{-2-\beta}.
	\end{equation*}
	By similar calculus, we can also deduce that
	\begin{equation*}
		|\nabla(-\Delta)^{\frac{\beta}{2}}\tilde{F}(x,s)|\leq C(\sqrt{s}+|x|)^{-3-\beta}.
	\end{equation*}
	Thus, \eqref{only} has been proved.
	By Lemma \ref{stokes}, we can directly obtain
	\begin{align*}
		\begin{split}
			&|V_{2}(x)|\leq C(1+|x|)^{-3}.
		\end{split}
	\end{align*}
	The above estimate along with \eqref{stokes.4-2} allow us to yield \eqref{T2}.
	At the same time, based on the same properties of the heat kernel, we also possess
	\begin{equation*}
	|G(x)|\leq C(1+|x|)^{-3}.
\end{equation*}
	
	Finally, we establish the decay estimate for the pressure $P$. Thanks to \eqref{press},
	the Newtonian potential enables us to obtain that 
	$$
	P(x)=\frac{1}{4\pi}\int_{\R^{3}}\frac{1}{|x-y|}\mathfrak{F}(y)dy,
	$$
	where
	\begin{align*}
		\begin{split}
			\mathfrak{F}(y)=&\sum_{i,j=1}^{3}\Big(\partial_{j}V_{i}\partial_{i}V_{j}+\partial_{j}V_{i}\partial_{i}U_{0j}+\partial_{j}U_{0i}\partial_{i}V_{j}+\partial_{j}U_{0i}\partial_{i}U_{0j}\\&-\partial_{j}G_{i}\partial_{i}G_{j}-\partial_{j}G_{i}\partial_{i}B_{0j}-\partial_{j}B_{0i}\partial_{i}G_{j}-\partial_{j}B_{0i}\partial_{i}B_{0j}\Big).
		\end{split}
	\end{align*}
	Thus, by Lemmas \ref{HYI} and \ref{SGN}, we conclude that
	\begin{align}\label{press1}
		\begin{split}
			\|P\|_{\infty}\leq& C\bigg\|\frac{1}{|\cdot|}\bigg\|_{L^{3,\infty}(\R^{3})}\|\mathfrak{F}\|_{L^{\frac{3}{2},1}(\R^{3})}\\ \leq& C\Big(\|\nabla V\|_{L^{6,2}(\R^{3})}\|\nabla V\|_{L^{2}(\R^{3})}+\|\nabla V\|_{L^{6,2}(\R^{3})}\|\nabla U_{0}\|_{L^{2}(\R^{3})}\\&+\|\nabla U_{0}\|_{L^{6,2}(\R^{3})}\|\nabla U_{0}\|_{L^{2}(\R^{3})}+\|\nabla G\|_{L^{6,2}(\R^{3})}\|\nabla G\|_{L^{2}(\R^{3})}\\&+\|\nabla G\|_{L^{6,2}(\R^{3})}\|\nabla B_{0}\|_{L^{2}(\R^{3})}+\|\nabla B_{0}\|_{L^{6,2}(\R^{3})}\|\nabla B_{0}\|_{L^{2}(\R^{3})}\Big)\\ 
			\leq& C\Big(\big\|(\nabla V,\nabla G)\big\|_{\mathbf{L}^{6,2}(\R^{3})}^{2}+\big\|(\nabla V,\nabla G)\big\|_{\mathbf{L}^{2}(\R^{3})}^{2}\\&+\big\|(\nabla U_{0},\nabla B_{0})\big\|_{\mathbf{L}^{6,2}(\R^{3})}^{2}+\big\|(\nabla U_{0},\nabla B_{0})\big\|_{\mathbf{L}^{2}(\R^{3})}^{2}
			\Big)\\ \leq& C\Big(\big\|(\nabla V,\nabla G)\big\|_{\mathbf{H}^{1}(\R^{3})}^{2}+\big\|(\nabla U_{0},\nabla B_{0})\big\|_{\mathbf{H}^{1}(\R^{3})}^{2}
			\Big)\\ \leq& C(U_{0},B_{0}).
		\end{split}
	\end{align}
	Letting
	$$
	(\bar{V},\bar{G})=(|\cdot|\nabla V,|\cdot|\nabla G) \ \ \mbox{and}\ \ (\bar{U_{0}},\bar{B_{0}})=(|\cdot|\nabla U_{0},|\cdot|\nabla B_{0}).
	$$
	Due to Proposition \ref{P-3-6} and the estimate \eqref{initial3}, we can easily obtain
	$$
	(\nabla\bar{V},\nabla\bar{G})\in \mathbf{L^{2}}(\R^{3})\ \ \mbox{and}\ \ (\nabla\bar{U}_{0},\nabla\bar{B}_{0})\in \mathbf{L^{2}}(\R^{3}).
	$$
	Hence,
	\begin{align}\label{press2}
		\begin{split}
			|x||P(x)|\leq& C\Big(\int_{\R^{3}}|x-y|\frac{1}{|x-y|}|\mathfrak{F}(y)|dy+\int_{\R^{3}}\frac{1}{|x-y|}|y||\mathfrak{F}(y)|dy\Big)\\ \leq& C\Big(\|\mathfrak{F}\|_{L^{1}(\R^{3})}+\big\|1/|\cdot|\big\|_{L^{3,\infty}(\R^{3})}\big\||\cdot|\mathfrak{F}\big\|_{L^{\frac{3}{2},1}(\R^{3})}\Big)\\
			\leq& C\Big(\big\|(\nabla V,\nabla G)\big\|_{\mathbf{L}^{2}(\R^{3})}^{2}+\big\|(\bar{V}, \bar{G})\big\|_{\mathbf{L}^{6,2}(\R^{3})}^{2}\\&+\big\|(\nabla U_{0},\nabla B_{0})\big\|_{\mathbf{L}^{2}(\R^{3})}^{2}+\big\|( \bar{U}_{0}, \bar{B}_{0})\big\|_{\mathbf{L}^{6,2}(\R^{3})}^{2}
			\Big)\\ \leq& C\Big(\big\|(\nabla V,\nabla G)\big\|_{\mathbf{L}^{2}(\R^{3})}^{2}+\big\|(\nabla\bar{V}, \nabla\bar{G})\big\|_{\mathbf{L}^{2}(\R^{3})}^{2}\\&+\big\|(\nabla U_{0},\nabla B_{0})\big\|_{\mathbf{L}^{2}(\R^{3})}^{2}+\big\|( \nabla\bar{U}_{0},\nabla \bar{B}_{0})\big\|_{\mathbf{L}^{2}(\R^{3})}^{2}
			\Big)\\ \leq& C(U_{0},B_{0}).
		\end{split}
	\end{align}
Notice that
	$$
	\nabla P(x)=\frac{1}{4\pi}\int_{\R^{3}}\frac{x-y}{|x-y|^{3}}\mathfrak{F}(y)dy.
	$$
Similarly, 
	\begin{align}\label{press3}
		\begin{split}
			\|\nabla P\|_{\infty}&\leq C\bigg\|\frac{1}{|\cdot|^{2}}\bigg\|_{L^{\frac{3}{2},\infty}(\R^{3})}\|\mathfrak{F}\|_{L^{3,1}(\R^{3})}\\&\leq C\Big(\big\|(\nabla V,\nabla G)\big\|_{\mathbf{L}^{6,2}(\R^{3})}^{2}+\big\|(\nabla U_{0},\nabla B_{0})\big\|_{\mathbf{L}^{6,2}(\R^{3})}^{2}
			\Big)\\&\leq C\Big(\big\|(\nabla^{2} V,\nabla^{2} G)\big\|_{\mathbf{L}^{2}(\R^{3})}^{2}+\big\|(\nabla^{2} U_{0},\nabla^{2} B_{0})\big\|_{\mathbf{L}^{2}(\R^{3})}^{2}
			\Big)\\&\leq C(U_{0},B_{0}),
		\end{split}
	\end{align}
	and
	\begin{align}\label{press4}
		\begin{split}
			|x|^{2}|\nabla P|\leq& C\Big(\int_{\R^{3}}|x-y|^{2}\frac{1}{|x-y|^{2}}|\mathfrak{F}(y)|dy+\int_{\R^{3}}\frac{1}{|x-y|^{2}}|y|^{2}|\mathfrak{F}(y)|dy\Big)\\ \leq& C\bigg(\big\|(\nabla V,\nabla G)\big\|_{\mathbf{L}^{2}(\R^{3})}^{2}+\big\|(\nabla U_{0},\nabla B_{0})\big\|_{\mathbf{L}^{2}(\R^{3})}^{2}\\&+\big\|1/|\cdot|^{2}\big\|_{L^{\frac{3}{2},\infty}(\R^{3})}\big\||\cdot|^{2}\mathfrak{F}\big\|_{L^{3,1}(\R^{3})}\bigg)\\ \leq& C\bigg(\big\|(\nabla V,\nabla G)\big\|_{\mathbf{L}^{2}(\R^{3})}^{2}+\big\|(\nabla U_{0},\nabla B_{0})\big\|_{\mathbf{L}^{2}(\R^{3})}^{2}\\&+\big\|(\bar{V}, \bar{G})\big\|_{\mathbf{L}^{6,2}(\R^{3})}^{2}+\big\|( \bar{U}_{0}, \bar{B}_{0})\big\|_{\mathbf{L}^{6,2}(\R^{3})}^{2}\bigg)\\ \leq& C(U_{0},B_{0}).
		\end{split}
	\end{align}
	Collecting \eqref{press1}-\eqref{press4}, we get \eqref{L3-13-1}.
	This theorem has been proven.
	\end{proof}

		\section{Conclusion}
		In this section, we prove our conclusion by using Theorems \ref{T1.2} and \ref{T1-1}.
		\begin{proof}[Proof of Theorem \ref{T1.1}]
		First, we focus on the existence of forward self-similar solutions. Let $(V,G,P)$ be constructed in Theorem \ref{T1.2} and denote
			$$
			v(x,t)=\frac{1}{\sqrt{t}}V\Big(\frac{x}{\sqrt{t}}\Big),\ \ g(x,t)=\frac{1}{\sqrt{t}}G\Big(\frac{x}{\sqrt{t}}\Big),\ \ p(x,t)=\frac{1}{t}P\Big(\frac{x}{\sqrt{t}}\Big).
			$$
			Setting
			\begin{equation*}
				u(x,t)=\frac{1}{\sqrt{t}}(U_{0}+V)\Big(\frac{x}{\sqrt{t}}\Big)=u_{I}(x,t)+v(x,t)
			\end{equation*}
			and
			\begin{equation*}
				b(x,t)=\frac{1}{\sqrt{t}}(B_{0}+G)\Big(\frac{x}{\sqrt{t}}\Big)=b_{I}(x,t)+g(x,t).
			\end{equation*}
		 We can readily verify that $(u,b,p)$ is a self-similar solution to the MHD equations \eqref{E1.1} due to its scaling properties. Thanks to Lemma \ref{inital}, we have
			\begin{equation}\label{E4.1}
				\big(u_{I}(x,t), b_{I}(x,t)\big)\in BC_{w}([0,+\infty),\mathbf{L}^{3,\infty}(\R^{3})).
			\end{equation}
			Moreover,
			$$
			\big(u_{I}(x,t), b_{I}(x,t)\big)\rightarrow\big(u_{0},b_{0}\big), \ \ \mbox{as}\ \  t\rightarrow0^{+},
			$$
			in the weak star topology of $\mathbf{L}^{3,\infty}(\R^{3})$.
			Next, we will prove
			\begin{equation*}\label{E4.2}
				\big(v(x,t), g(x,t)\big)\in BC_{w}([0,+\infty),\mathbf{L}^{3,\infty}(\R^{3})).
			\end{equation*}
			Since $(V,G)\in\mathbf{H}_{\sigma}^{1}(\R^{3})$, it follows from the Sobolev embedding theorem that
			$$
			\|(V,G)\|_{p}\leq C\|(V,G)\|_{\mathbf{H}^{1}(\R^{3})},\ \ \mbox{for each}\ 2\leq p\leq6.
			$$
			Thus
				\begin{align}\label{E4.3}
				\begin{split}
					\big\|\big(v(t),g(t)\big)\big\|_{\mathbf{L}^{p}(\R^{3})}&=\Big(\int_{\R^{3}}\Big|\frac{1}{\sqrt{t}}(V,G)(\frac{x}{\sqrt{t}})\Big|^{p}dx\Big)^{\frac{1}{p}}\\&
					=t^{\frac{3}{2p}-\frac{1}{2}}\big\|(V,G)\big\|_{p},
				\end{split}
			\end{align}
			which implies that $\big(v(t),g(t)\big)\in L^{\infty}([0,\infty),\mathbf{L}^{3,\infty}(\R^{3}))$, since $L^{3}\hookrightarrow L^{3,\infty}$. It remains to demonstrate that $\big(v(x,t), g(x,t)\big)$ is weak star continuous in $\mathbf{L}^{3,\infty}(\R^{3})$ with respect to $t$. Indeed, we only need to consider the continuity at $0$. We claim that
			$$
			\big(v(t),g(t)\big)\rightarrow0, \ \ \mbox{as}\ \  t\rightarrow0^{+},
			$$
			in the weak star topology of $\mathbf{L}^{3,\infty}(\R^{3})$.
			By interpolation theory
			$$
			\big(L^{1}(\R^{3}),L^{2}(\R^{3})\big)_{\frac{2}{3},1}=L^{\frac{3}{2},1}(\R^{3}).
			$$
			Thus, for any $\Upsilon\in \mathbf{L}^{\frac{3}{2},1}(\R^{3})$, we can approximate it using functions $\Upsilon_{\epsilon}\in \mathbf{L}^{1}\cap \mathbf{L}^{2}(\R^{3})$.	
		Thanks to \eqref{E4.3}, we have
			$$
			\int_{\R^{3}}\big(v(t),g(t)\big)\Upsilon_{\epsilon}dx\leq t^{\frac{1}{4}}\|(V,G)\|_{2}\|\Upsilon_{\epsilon}\|_{2}.
			$$
			By the Dominated convergence theorem, one obtains
			$$
			\int_{\R^{3}}\big(v(t),g(t)\big)\Upsilon dx\leq Ct^{\frac{1}{4}}\rightarrow0, \ \ \mbox{as}\ \ t\rightarrow0^{+}.
			$$
			Therefore,
			$$\big(v(x,t),g(x,t)\big)\in BC_{w}([0,\infty);\mathbf{L}^{3,\infty}(\R^{3})),$$
			together with the fact \eqref{E4.1}, we get $(u,b)\in BC_{w}([0,\infty);\mathbf{L}^{3,\infty}(\R^{3}))$.
			Moreover, we have  $(V,G)\in \mathbf{C}^{\infty}(\R^{3})$ by Theorem \ref{T1.3}, which implies that 
			$$
			\big(u(x,t), b(x,t)\big)\in \mathbf{C}^{\infty}\big(\R^{3}\times(0,\infty)\big).
			$$
			Now, we verify \eqref{T1}. By \eqref{E4.3}, we have
			$$
			\|u-e^{t\Delta}u_{0}\|_{p}=\Big\|\frac{1}{\sqrt{t}}V\Big(\frac{x}{\sqrt{t}}\Big)\Big\|_{p}\leq Ct^{\frac{3}{2p}-\frac{1}{2}},
			$$
			and
			$$
			\|\nabla u-\nabla e^{t\Delta}u_{0}\|_{2}=\Big\|\frac{1}{\sqrt{t}}\nabla_{x} V\Big(\frac{x}{\sqrt{t}}\Big)\Big\|_{2}\leq Ct^{-\frac{1}{4}}.
			$$
			Similarly,
			$$
			\|b-e^{t\Delta}b_{0}\|_{p}\leq Ct^{\frac{3}{2p}-\frac{1}{2}}\ \  \mbox{and} \ \ \|\nabla b-\nabla e^{t\Delta}b_{0}\|_{2}\leq Ct^{-\frac{1}{4}}.
			$$

			Next, we will focus on the pointwise estimation for self-similar solutions $(u,b)$. Notice that, if $(u_{0},b_{0})\in \mathbf{C}^{0,1}_{loc}(\R^{3}\setminus\{0\})$, by scaling property and the decay estimate \eqref{initial2}, we can derive
			\begin{align}\label{decay1}
				\begin{split}
				&|u_{I}(x,t)|+|b_{I}(x,t)|\leq C\frac{1}{(\sqrt{t}+|x|)},\\&
				|\nabla u_{I}(x,t)|+|\nabla b_{I}(x,t)|\leq C\frac{1}{(\sqrt{t}+|x|)^{2}}.
			\end{split}
		\end{align}
		According to Theorem \ref{T1-1}, we have
		\begin{align}\label{decay2}
			\begin{split}
		&|u(x,t)-e^{t\Delta}u_{0}|+|b(x,t)-e^{t\Delta}b_{0}|\\=&|v(x,t)|+|g(x,t)|	\\=&\Big|\frac{1}{\sqrt{t}}V\Big(\frac{x}{\sqrt{t}}\Big)\Big|+\Big|\frac{1}{\sqrt{t}}G\Big(\frac{x}{\sqrt{t}}\Big)\Big|
			\\ \leq& C\frac{t}{(\sqrt{t}+|x|)^{3}}\log{\Big(2+\frac{|x|}{\sqrt{t}}\Big)},
			\end{split}
		\end{align}
		and
		\begin{align}\label{decay4}
			\begin{split}
			&|\nabla u(x,t)-\nabla e^{t\Delta}u_{0}|+|\nabla b(x,t)-\nabla e^{t\Delta}b_{0}|\\=&|\nabla v(x,t)|+|\nabla g(x,t)|	\\=&\Big|\frac{1}{t}\nabla V\Big(\frac{x}{\sqrt{t}}\Big)\Big|+\Big|\frac{1}{t}\nabla G\Big(\frac{x}{\sqrt{t}}\Big)\Big|
				\\ \leq& C\frac{\sqrt{t}}{(\sqrt{t}+|x|)^{3}}.
			\end{split}
		\end{align}
		Thus, \eqref{T2-3} has been proven. Moreover, 
		from \eqref{decay1}-\eqref{decay4}, we can also derive \eqref{T3}. Using Theorem \ref{T1-1} again, if $(u_{0},b_{0})\in \mathbf{C}^{1,\alpha}_{loc}(\R^{3}\setminus\{0\})$ with $0<\alpha\leq1$, i.e. $(U_{0},B_{0})$ satisfies the decay estimate \eqref{initial3}, we can get
		\begin{align*}
			\begin{split}
				|u(x,t)-e^{t\Delta}u_{0}|+|b(x,t)-e^{t\Delta}b_{0}|
				\leq C\frac{t}{(\sqrt{t}+|x|)^{3}},
			\end{split}
		\end{align*}
		and when $\alpha=1$,
		\begin{align*}
			\begin{split}
				|\nabla^{k} p(x,t)|=\Big|t^{-\frac{k+2}{2}}\nabla^{k} P\Big(\frac{x}{\sqrt{t}}\Big)\Big|
				\leq Ct^{-\frac{1}{2}}\frac{1}{(\sqrt{t}+|x|)^{k+1}},\ \ k=0,1.
			\end{split}
		\end{align*}
		Thus, \eqref{T2-4} and \eqref{Press} have been proven, and we conclude the proof of Theorem \ref{T1.1}.
		\end{proof}

		\setcounter{equation}{0}
		\setcounter{equation}{0}

		\begin{appendices}
			\appendix
			\renewcommand{\appendixname}{}
			\renewcommand{\theequation}{\thesection \arabic{equation}}
			\setcounter{equation}{0}

	\end{appendices}

	\end{CJK}
\end{document}